\documentclass[reqno, twoside, a4paper,10pt]{amsart}

\usepackage{amsmath}
\usepackage{amssymb}
\usepackage{amsthm}
\usepackage{graphicx}
\usepackage{subfig}
\usepackage{enumerate}
\usepackage{booktabs}
\usepackage{tabulary}
\usepackage{nicematrix}
\usepackage{hyperref}

\hypersetup{%
pdftitle={Deformations of sandwiched surface singularities and the minimal model program},%
pdfauthor={Heesang Park, Dongsoo Shin},%
pdfkeywords={deformation, minimal model program, P-modification, picture deformation, sandwiched surface singularity, smoothing of negative weights, weighted homogeneous surface singularity},%
colorlinks=true,%
citecolor=black,%
filecolor=black,%
linkcolor=black,%
urlcolor=black,%
}

\usepackage{comment}
\usepackage{xspace}
\usepackage{mathtools}

\usepackage{mathptmx}       
\usepackage{helvet}         
\usepackage{courier}        

\usepackage{tikz}

\usetikzlibrary{patterns,decorations.pathreplacing,calligraphy}
\usetikzlibrary{decorations.pathmorphing}

\tikzset{bullet/.style={
shape = circle,fill = black, inner sep = 0pt, outer sep = 0pt, minimum size = 0.4em, line width = 0pt, draw=black!100}}

\tikzset{rectangle/.style={
shape = rectangle,fill = white, inner sep = 0pt, outer sep = 0pt, minimum size = 0.4em, line width = 0.1em, draw=black!100}}

\tikzset{smallbullet/.style={
shape = circle,fill = black, inner sep = 0pt, outer sep = 0pt, minimum size = 0.15em, line width = 0pt, draw=black!100}}

\tikzset{circle/.style={
shape = circle,fill = none, inner sep = 0pt, outer sep = 0pt, minimum size = 0.4em, line width = 0.5pt, draw=black!100}}

\tikzset{empty/.style={
shape = circle,fill = white, inner sep = 0pt, outer sep = 0pt, minimum size = 0.35em, line width = 0pt, draw=white!100}}

\tikzset{xmark/.style={
shape = x,fill = white, inner sep = 0pt, outer sep = 0pt, minimum size = 0em, line width = 0pt, draw=white!100}}

\tikzset{longrectangle/.style={
inner sep = 1em,
rectangle,
minimum size=1em,
very thick,
draw=black!100, 
}}

\tikzset{label distance=-0.15em}
\tikzset{every label/.append style = {font = \scriptsize}}

\usepackage{calc}
\newlength{\mylength}
\tikzset{labelAbove/.style={label={[xshift=-.5\mylength]above:#1}}}
\tikzset{labelBelow/.style={label={[xshift=-.5\mylength]below:#1}}}

\usepackage{tikz-cd}
\tikzcdset{every label/.append style = {font = \normalsize}}

\newtheorem{theorem}{Theorem}[section]
\newtheorem{lemma}[theorem]{Lemma}
\newtheorem{proposition}[theorem]{Proposition}
\newtheorem{corollary}[theorem]{Corollary}
\newtheorem{maintheorem}{Theorem}

\newtheorem{conjecture}[theorem]{Conjecture}

\newtheorem{problem}{Problem}

\theoremstyle{definition}

\newtheorem{definition}[theorem]{Definition}
\newtheorem{remark}[theorem]{Remark}

\newtheorem{example}[theorem]{Example}

\numberwithin{equation}{section}

\def\sheaf#1{\ensuremath \mathcal#1}

\newcommand{\abs}[1]{\ensuremath \left\lvert #1 \right\rvert}

\makeatletter
\providecommand{\leftsquigarrow}{%
  \mathrel{\mathpalette\reflect@squig\relax}%
}
\newcommand{\reflect@squig}[2]{%
  \reflectbox{$\m@th#1\rightsquigarrow$}%
}
\makeatother

\DeclareMathOperator{\Spec}{Spec}
\DeclareMathOperator{\Proj}{Proj}

\DeclareMathOperator{\Def}{Def}

\DeclareMathOperator{\Ext}{Ext}

\begin{document}

\title[Deformations of sandwiched surface singularities and the MMP]{Deformations of sandwiched surface singularities and the minimal model program}

\author[H. Park]{Heesang Park}

\address{Department of Mathematics, Konkuk University, Seoul 05029, Republic of Korea}

\email{HeesangPark@konkuk.ac.kr}

\author[D. Shin]{Dongsoo Shin}

\address{Department of Mathematics, Chungnam National University, Daejeon 34134, Republic of Korea}

\email{dsshin@cnu.ac.kr}



\begin{abstract}
We investigate the correspondence between three theories of deformations of rational surface singularities: de Jong and van Straten's picture deformations, Kollár's P-resolutions, and Pinkham's smoothings of negative weights. We provide an explicit method for obtaining, from a given deformation in one theory, deformations in other theories that parameterize the same irreducible components of the deformation space of the singularity. We employ the semi-stable minimal model program significantly for this purpose. We prove Kollár conjecture for various sandwiched surface singularities as an application.
\end{abstract}

\maketitle

\tableofcontents

\section{Introduction}

The set of irreducible components of the reduced miniversal deformation space of a rational surface singularity can be described in at least three ways: de Jong and van Straten's picture deformations, Kollár's P-resolutions, and Pinkham's smoothings of negative weights. In this paper, the correspondence between these three deformation theories is investigated. If a deformation is given in one theory, we propose an exact algorithm for constructing deformations in other theories that correspond to the same irreducible component.

We start with deformations of sandwiched surface singularities. A \emph{sandwiched surface singularity} is a rational surface singularity that admits a birational morphism to the complex projective plane. de Jong and van Straten~\cite{deJong-vanStraten-1998} proved that smoothings of sandwiched surface singularities are induced from special deformations of germs of plane curve singularities (called \emph{picture deformations}). We investigate correspondence between picture deformations and other deformation theories that can be applied to sandwiched surfaces singularities.

Kollár~\cite{Kollar-1991} conjectured that deformations of any rational surface singularities are induced by particular partial modifications (called \emph{P-modifications}). We present a method for determining the picture deformation of a sandwiched surface singularity corresponding to the deformation given by a \emph{P-resolution} (essentially, a normal P-modification with a mild singularity) of the singularity via the semi-stable minimal model program for complex 3-folds. See Sections~\ref{section:P-modifications}--\ref{section:illustrations}.

On the other hand, Pinkham~\cite{Pinkham-1978} showed that smoothings of weighted homogeneous surface singularities can be obtained by adding terms of lower weights to the defining equations. These smoothings are referred to as \emph{smoothings of negative weights}. If a weighted homogeneous surface singularity has a `big' node (cf. Definition~\ref{definition:WHSS-big-node}), then it is sandwiched. We provide a connection between picture deformations and smoothings of negative weights of weighted homogeneous surface singularities with the big nodes. See Sections~\ref{section:WHSS-usual-sandwiched-structure}--\ref{section:Pic-Def-to-SNW}.

Finally, we present a method for obtaining smoothings of negative weights from given P-resolutions of weighted homogeneous surface singularities with the big nodes via the semi-stable MMP. It generalizes the results presented in PPSU~\cite{PPSU-2018} for cyclic quotient surface singularities. See Section~\ref{section:P-resolution->smoothings-of-negative-weights}.

Specifically, we establish a correspondence between all previously known deformation theories of cyclic quotient surface singularities: Picture deformations by de Jong and van Straten~\cite{deJong-vanStraten-1998}, P-resolutions by Kollár and Shepherd-Barron~\cite{KSB-1988}, Smoothings of negative weights by Pinkham~\cite{Pinkham-1978}, and Equations by Christophersen~\cite{Christophersen-1991} and Stevens~\cite{Stevens-1991}. The results in this paper may be regarded as an algebro-geometric version of Némethi and Popescu-Pampu~\cite{NPP-2010-PLMS}, where they investigated the relation between picture deformations and Equations using the Lisca's classification~\cite{Lisca-2008} of minimal symplectic fillings of cyclic quotient surface singularities. See Section~\ref{section:CQSS}.

As an application, we suggest a procedure for proving Kollár conjecture. Indeed we show in this paper that Kollár conjecture holds for the singularity $W(p,q,r)$, which is a weighted homogeneous surface singularity that admits a rational homology disk smoothing. See Sections~\ref{section:Wpqr}--\ref{section:Wpqr-K-Conjecture}. In JPS~\cite{Jeon-Park-Shin-2022}, they prove Kollár conjecture for all weighted homogeneous surface singularity admitting rational homology disk smoothings using a similar strategy developed in this paper. Also Jeon and Shin~\cite{Jeon-Shin-2022} prove Kollár conjecture for weighted homogeneous surface singularities with the `enough' big nodes in a similar fashion.

A more detailed introduction follows.

\subsection{Sandwiched surface singularities and their deformations}

A \emph{sandwiched surface singularity} is a normal surface singularity that admits a birational morphism to $\mathbb{C}^2$. Sandwiched surface singularities are introduced and classified in Spivakovsky~\cite{Spivakovsky-1990}. Sandwiched surface singularities are rational singularities and they are characterized by their dual resolution graphs, which are referred to as \emph{sandwiched graphs} by Spivakovsky~\cite{Spivakovsky-1990}; see Definition~\ref{definition:sandwiched-graph}. Sandwiched surface singularities include cyclic quotient surface singularities, weighted homogeneous surface singularities with big nodes, and rational surface singularities with the reduced fundamental cycles.

de Jong and van Straten~\cite{deJong-vanStraten-1998} showed that any one-parameter deformations of a sandwiched surface singularity $(X,p)$ are obtained from specific one-parameter deformations of a germ of a plane curve singularity $C \subset \mathbb{C}^2$ through the origin $(0,0)$, where each irreducible components $C_i$ of $C$ are decorated by some natural numbers $l_i$. The pair $(C = \cup C_i, l = \cup l_i)$ is called a \emph{decorated curve} of the singularity $(X,p)$. Furthermore, all smoothings of the singularity $(X,p)$ are provided by certain particular one-parameter deformations of the decorated curve $(C,l)$, known as \emph{picture deformations}. See Section~\ref{section:sandwiched} for a summary.

Every picture deformation gives us a certain matrix (called \emph{the incidence matrix}) encoding the incidence relation between the decorated curves and the $(-1)$-curves on a general fiber of the picture deformation; Definition~\ref{definition:incidence-matrix}. So we have a well defined map from the set $\mathcal{C}(X)$ of all irreducible components of the reduced miniversal deformation space of $X$ to the set $\mathcal{I}(X)$ of all incidence matrices of $X$:
\begin{equation*}
\phi_I \colon \mathcal{C}(X) \to \mathcal{I}(X)
\end{equation*}

\emph{Whether $\phi_I$ is always injective remains an outstanding question}; refer de Jong--van Straten~\cite[p.485]{deJong-vanStraten-1998}. This occurs, for instance, for cyclic quotient surface singularities, as far as is known. In this paper, we prove that:

\begin{maintheorem}[Theorem~\ref{theorem:phi_I-injective}]
Suppose that a weighted homogeneous surface singularity has the node with the self-intersection number ($-d$) and $t$ branches. If $d \ge t+2$, then $\phi_I$ is injective (hence, bijective).
\end{maintheorem}

Every incidence matrix satisfies a system of (quadratic and linear) equations on the entries of their row vectors, which are naturally induced from the conditions of picture deformations; Equation~\eqref{equation:combinatorial-incidence-matrix}. A matrix satisfying the equations is called a \emph{combinatorial incidence matrix} of $X$. We denote by $C\mathcal{I}(X)$ the set of all combinatorial incidence matrices of $X$; Definition~\ref{definition:combinatorial-incidence-matrix}. \emph{It is also an open and very delicate problem (depending on the moduli of $X$) which combinatorial incidence matrix can be realized as the incidence matrix of a picture deformation}; cf.~de Jong--van Straten~\cite[p.485]{deJong-vanStraten-1998}. As an application of the relation between picture deformations and P-resolutions developed in this paper, Jeon--Shin~\cite{Jeon-Shin-2022} prove that $\mathcal{I}(X)=C\mathcal{I}(X)$ for weighted homogeneous surface singularities with $d \ge t+3$.

\subsection{From P-resolutions to picture deformations}

Kollár~\cite{Kollar-1991} conjectured that for any rational surface singularity $(X,p)$ there is a one-to-one correspondence between irreducible components of the reduced miniversal deformation space of $(X,p)$ and specific proper modifications of $(X,p)$, called \emph{P-modifications}. In particular, the conjecture states that if $\mathcal{X} \to \Delta$ is a one-parameter deformation of $X$ over a small disk $\Delta$  then there exists a P-modification $U \to X$ together with a $\mathbb{Q}$-Gorenstein deformation $\mathcal{U} \to \Delta$ of $U$ such that the deformation $\mathcal{U} \to \Delta$ is blown down (possibly after a base change) to the given deformation $\mathcal{X} \to \Delta$. Kollár conjecture is known to hold in only a few cases; for further information, see Stevens~\cite{Stevens-2003}. For example, Kollár and Shepherd-Barron~\cite{KSB-1988} showed that the conjecture holds for quotient surface singularities and that the corresponding P-modifications are normal surfaces admitting only \emph{T-singularities} as singularities (which are referred to as \emph{P-resolutions}). Notice that each P-resolution parametrize different components of $\Def(X)$; KSB~\cite[Theorem 3.9]{KSB-1988}. Let $\mathcal{P}(X)$ be the set of all P-resolutions of $X$. Then we have an injective map
\begin{equation*}
\phi_P \colon \mathcal{P}(X) \to \mathcal{C}(X).
\end{equation*}
See Section~\ref{section:P-modifications} for summary.

We present a method for finding picture deformations from given P-resolutions of sandwiched surface singularities by applying the semi-stable minimal model program for complex 3-folds. That is, if a one-parameter smoothing $\mathcal{X} \to \Delta$ of $(X,p)$ is induced by a P-resolution $U \to X$, then one can find the picture deformation of a decorated curve $(C,l)$ corresponding to the smoothing $\mathcal{X} \to \Delta$ of $(X,p)$ by applying the semi-stable MMP.

More specifically, it is as follows. We first compactify the singularity $(X,p)$ and its decorated curve $(C,l)$ to a projective singular surface $(Y,p)$ and a decorated projective curve $(D,l)$ (called a \emph{compactified decorated curve}), respectively, using the birational morphism $X \to \mathbb{C}^2$. We show that there is no local-to-global obstruction to deforming $(Y,p)$ by showing that $H^2(Y, \mathcal{T}_Y)=0$; Theorem~\ref{theorem:extension-of-deformation}. Then we show that every one-parameter smoothing of the projective surface $(Y,p)$ induced by a smoothing of the singular point $p \in Y$ is obtained from a picture deformation of the compactified decorated curve $(D,l)$ as before; See Section~\ref{section:compactification} for details. Similarly, in Section~\ref{section:P-resolution->Picture-deformations}, we show that any P-resolution $U \to X$ of the singularity $X$ can be compactified to a projective surface $Z \to Y$ (called \emph{compactified P-resolution}). Any deformation of $U$ can be extended to that of $Z$. Then we show that, if a one-parameter smoothing $\mathcal{X} \to \Delta$ of $(X,p)$ is induced by a P-resolution $U \to X$, then one obtain the picture deformation of the compactified decorated curve $(D,l)$ of $(Y,p)$, hence that of the decorated curve $(C,l)$ of $(X,p)$, corresponding to the deformation $\mathcal{X} \to \Delta$ by applying the semi-stable MMP to the $\mathbb{Q}$-Gorenstein smoothing $\mathcal{Z} \to \Delta$ of the compactified P-resolution $Z \to Y$.

\begin{maintheorem}[Theorem~\ref{theorem:From-P-resol-To-Pic-def}]
One can run the semi-stable MMP to a one-parameter deformation $\mathcal{Z} \to \Delta$ of a compactified P-resolution $Z$ of $(X,p)$ until one obtains the corresponding picture deformation $\mathcal{D} \to \Delta$ of the compactified decorated curve $(D,l)$.
\end{maintheorem}

Therefore:

\begin{maintheorem}[Theorem~\ref{theorem:Phi_PI=Phi_IoPhi_P}]
There is a map
\begin{equation*}
\phi_{PI} \colon \mathcal{P}(X) \to \mathcal{I}(X)
\end{equation*}
defined by the MMP such that $\phi_{PI} = \phi_I \circ \phi_P$.
\end{maintheorem}

The technique developed in this paper may be applicable to normal P-modifications which are not P-resolutions; See Section~\ref{section:non-P-resolution} for instance.

\subsection{From picture deformations to smoothings of negative weights}

A \emph{weighted homogeneous surface singularity} $X$ is a rational normal surface singularity with a \emph{good} $\mathbb{C}^{\ast}$-action. It has a unique exceptional divisor (called the \emph{node}) which intersects more than two other components. Let $(-d)$ be the self-intersection number of the node and let $t$ be the number of the components that intersect the node. If $d \ge t+1$, then the singularity $X$ has a natural sandwiched structure. So one can describe its smoothings via picture deformations. For details, see Section~\ref{section:WHSS-usual-sandwiched-structure}.

On the other hand, Pinkham~\cite{Pinkham-1974} proved that there exists a versal deformation $\mathcal{X} \to V$ of $X$ admitting a $\mathbb{C}^{\ast}$-action extending that of $X$. So the space $V$ can be decomposed into three parts $V^+, V^0, V^-$ according to the $\mathbb{C}^{\ast}$-action. Pinkham~\cite{Pinkham-1978} showed that every smoothing of $X$ occurs on $V^-$. So it is called a \emph{smoothing of negative weight}. Following Pinkham~\cite{Pinkham-1978}, we define a certain compactification $\overline{X}^P$ of $X$; Definition~\ref{definition:P-compactification}. We then show that any smoothings $\mathcal{X} \to \Delta $ of $X$ can be lifted to smoothings $\overline{\mathcal{X}}^P \to \Delta$ of $\overline{X}^P$  which are locally trivial in a neighborhood of the complement $\widetilde{E}_{\infty} = \overline{X}^P-X$ (called the \emph{compactifying divisor} of $X$); Theorem~\ref{theorem:no-local-to-global}. Conversely, Pinkham~\cite{Pinkham-1978} showed that any smoothing of $X$ is completely determined by an ample embedding of $\widetilde{E}_{\infty}$ into a smooth surface $S$ satisfying certain cohomological conditions. We recall Pinkham's results in Section~\ref{section:WHSS-usual-sandwiched-structure}.

For every smoothing $\mathcal{X} \to \Delta$ of $X$, we define a matrix (called a \emph{homology matrix} of $\mathcal{X} \to \Delta$) that encode how $(-1)$-curves in a general fiber $\overline{\mathcal{X}}^P_t$ of $\overline{\mathcal{X}}^P \to \Delta$ intersect $\widetilde{E}_{\infty}$; Definition~\ref{definition:homology-matrix}. Let $\mathcal{H}(X)$ be the set of all homology matrices of $X$. So we can define a natural map
\begin{equation*}
\phi_H \colon \mathcal{C}(X) \to \mathcal{H}(X),
\end{equation*}
which may represent combinatorial pictures of smoothings of negative weights. Indeed the homology matrix of $\widetilde{E}_{\infty}$ in $\overline{\mathcal{X}}^P_t$ determines completely the component of $\Def(X)$ itself where the corresponding smoothing of $X$ lies:

\begin{maintheorem}[Theorem~\ref{theorem:phi_H-injective-cyclic}, Theorem~\ref{theorem:phi_H-injective-WHSS}]
If $X$ is a cyclic quotient surface singularity or a weighted homogeneous surface singularity with $d \ge t+2$, then $\phi_H$ is injective; hence, it is bijective.
\end{maintheorem}

We will compare picture deformations and smoothings of negative weights via a map from $\mathcal{I}(X)$ to $\mathcal{H}(X)$:

\begin{maintheorem}[Theorem~\ref{theorem:phi_IH}, Theorem~\ref{theorem:phi_I-injective}]
Let $X$ be a cyclic quotient surface singularity of a weighted homogeneous surface singularity with the big node. Then there is a map
\begin{equation*}
\phi_{IH} \colon \mathcal{I}(X) \to \mathcal{H}(X)
\end{equation*}
such that the $\phi_H = \phi_{IH} \circ \phi_I$. If $X$ is cyclic or weighted with $d \ge t+2$, then the map $\phi_{IH}$ is bijective.
\end{maintheorem}

Since $X$ is cyclic or weighted with $d \ge t+2$, then $\phi_I$ and $\phi_H$ are bijective too. So we can determine the corresponding smoothing of negative weight of $X$ from a given picture deformation of $X$, and vice versa. In order to prove the above theorem, we introduce three compactifications of $X$ and we investigate their relations; Section~\ref{section:Topology-compactifications}.

\subsection{From P-resolutions to smoothings of negative weights}

By applying the semi-stable MMP, we can obtain the smoothing of negative weight of $X$ for each P-resolution of $X$. PPSU~\cite{PPSU-2018} introduced this technique for cyclic quotient surface singularities. We employ the same approach for handling weighted homogeneous cases.

For a given P-resolution $Y$ of $X$, there is a compactification $\overline{Y}^P$ of $Y$ such that it is also a P-resolution of the compactification $\overline{X}^P$ of $X$. We show that every smoothing $\overline{\mathcal{X}}^P \to \Delta$ of $\overline{X}^P$ can be lifted to A smoothing $\overline{\mathcal{Y}}^P \to \Delta$ of $\overline{Y}^P$; Proposition~\ref{proposition:no-obstruction-PtoH}. We then show that one can run the semi-stable MMP to $\overline{\mathcal{Y}}^P \to \Delta$ until we obtain a deformation $\mathcal{W} \to \Delta$ whose central fiber $W_0$ is smooth; Theorem~\ref{theorem:P->H}. By tracking how $(-1)$-curves are changed during the MMP process, we obtain the information how $(-1)$-curves intersect $\widetilde{E}_{\infty}$ on a general fiber $\overline{Y}^P_t$ of $\overline{\mathcal{Y}}^P \to \Delta$, which gives us the corresponding homology matrix of $\overline{\mathcal{X}}^P \to \Delta$.

\begin{maintheorem}[Theorem~\ref{theorem:phi_PH}]
Let $X$ be a cyclic quotient surface singularity or a weighted homogeneous surface singularity with $d \ge t+1$. Then there is a map
\begin{equation*}
\phi_{PH} \colon \mathcal{P}(X) \to \mathcal{H}(X)
\end{equation*}
defined via the MMP such that $\phi_{PH} = \phi_H \circ \phi_P$.
\end{maintheorem}

\subsection{Summary of the correspondence}

There are three ways of parameterizing the set $\mathcal{C}(X)$ of all irreducible components of the reduced miniversal deformation space $\Def{X}$ of $X$: $\mathcal{I}(X)$ (incidence matrices),  $\mathcal{P}(X)$ (P-resolutions), $\mathcal{H}(X)$ (homology matrices). So we have three maps: $\phi_P \colon \mathcal{P}(X) \to \mathcal{C}(X)$, $\phi_I \colon \mathcal{C}(X) \to \mathcal{I}(X)$, $\phi_H \colon \mathcal{C}(X) \to \mathcal{H}(X)$. In this paper, we investigate the correspondence between the three sets so that we can define three maps: $\phi_{PI} \colon \mathcal{P}(X) \to \mathcal{I}(X)$, $\phi_{IH} \colon \mathcal{I}(X) \to \mathcal{H}(X)$, $\phi_{PH} \colon \mathcal{P}(X) \to \mathcal{H}(X)$. As a result, we have the commutative diagram of three deformation theories in Figure~\ref{figure:summary-diagram}.

\begin{figure}
\centering
\begin{tikzpicture}[scale=2]
\node[] (C) at (0,0) [] {$\mathcal{C}(X)$};
\node[] (I) at (0,1) [] {$\mathcal{I}(X)$};
\node[] (P) at (-0.85,-0.5) [] {$\mathcal{P}(X)$};
\node[] (H) at (0.85,-0.5) [] {$\mathcal{H}(X)$};

\node[] (phi_I) at (-0.01,0.5) [label=right:{\normalsize $\phi_I$}] {};
\node[] (phi_P) at (-0.43,-0.3) [label=above:{\normalsize $\phi_P$}] {};
\node[] (phi_H) at (0.43,-0.3) [label=above:{\normalsize $\phi_H$}] {};

\node[] (phi_PI) at (-0.95,0.4) [label=above:{\normalsize $\phi_{PI}$}] {};
\node[] (phi_IH) at (0.95,0.4) [label=above:{\normalsize $\phi_{IH}$}] {};
\node[] (phi_PH) at (0,-0.9) [label=below:{\normalsize $\phi_{PH}$}] {};

\draw[->] (P)--(C);
\draw[->] (C)--(I);
\draw[->] (C)--(H);

\draw[->] (P) to [bend left=45] (I);
\draw[->] (I) to [bend left=45] (H);
\draw[->] (P) to [bend right=45] (H);
\end{tikzpicture}

\caption{Deformation theories and their relations}
\label{figure:summary-diagram}
\end{figure}
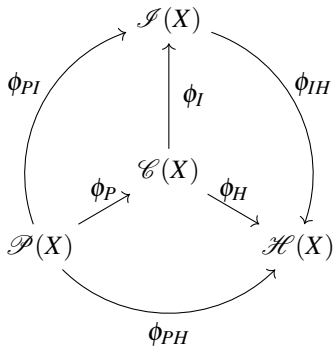

\subsection{Deformations of cyclic quotient surface singularities}

Let $X$ be a cyclic quotient surface singularity. It has been known that the maps in Figure~\ref{figure:summary-diagram} are all bijective. In addition, Christophersen~\cite{Christophersen-1991} and Stevens~\cite{Stevens-1991}  described the equations of each component in $\mathcal{C}(X)$. As a result, they showed that there is one-to-one correspondence between $\mathcal{C}(X)$ and a set $\mathcal{K}(X)$ consisting of certain sequences $\underline{k}=(k_1,\dotsc,k_e)$ of positive integers whose Hirzebruch-Jung continued fractions are zero. See Section~\ref{section:CQSS} for summary.

We will show how to construct the corresponding sequence $\underline{k}$ from a given incidence matrix/P-resolution/homology matrix in Section~\ref{section:CQSS}. So we can define maps $\phi_{IK} \colon \mathcal{I}(X) \to \mathcal{K}(X)$, $\phi_{PK} \colon \mathcal{P}(X) \to \mathcal{K}(X)$ (which is already given in PPSU~\cite{PPSU-2018}), and $\phi_{HK} \colon \mathcal{H}(X) \to \mathcal{K}(X)$ so that we have the extended commutative diagram in Figure~\ref{figure:summary-diagram-CQSS}.

Notice that Némethi and Popescu-Pampu~\cite{NPP-2010-PLMS} could compare picture deformations and Equations via Lisca’s classification~\cite{Lisca-2008} of minimal symplectic fillings of lens spaces; but, we find direct algorithms inside algebraic geometry.

\begin{figure}
\centering
\begin{tikzpicture}[scale=2]
\node[] (C) at (-0.425,0.25) [] {$\mathcal{C}(X)$};
\node[] (K) at (0.425,0.25) [] {$\mathcal{K}(X)$};

\node[] (I) at (0,1) [] {$\mathcal{I}(X)$};
\node[] (P) at (-0.85,-0.5) [] {$\mathcal{P}(X)$};
\node[] (H) at (0.85,-0.5) [] {$\mathcal{H}(X)$};

%

\draw[<->] (C)--(K);

\draw[->] (P)--(C);
\draw[->] (C)--(I);
\draw[->] (C)--(H);

\draw[->] (P)--(K);
\draw[->] (I)--(K);
\draw[->] (H)--(K);

\draw[->] (P) to [bend left=45] (I);
\draw[->] (I) to [bend left=45] (H);
\draw[->] (P) to [bend right=45] (H);
\end{tikzpicture}

\caption{Deformation theories of cyclic quotient surface singularities}
\label{figure:summary-diagram-CQSS}
\end{figure}
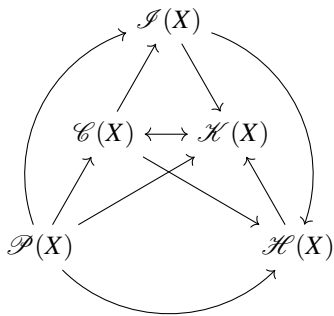

\subsection{How to prove Kollár conjecture}

If the map $\phi_P \colon \mathcal{P}(X) \to \mathcal{C}(X)$ is surjective for a rational surface singularity $X$, then Kollár conjecture holds for $X$.

We suggest a strategy for proving Kollár conjecture using the correspondence developed in this paper. At first, we show that the map $\phi_I \colon \mathcal{C}(X) \to \mathcal{I}(X)$ is injective (hence bijective) for a given singularity $X$. For instance, if $X$ is a weighted homogeneous surface singularity with $d \ge t+2$, then we proved that $\phi_I$ is injective using the injectivity of $\phi_H \colon \mathcal{C}(X) \to \mathcal{H}(X)$; Theorem~\ref{theorem:phi_I-injective}. Next we find all combinatorial incidence matrices of $X$. This is a rather tedious combinatorial problem. Finally, we construct P-resolutions that correspond to each given combinatorial incidence matrices, where the MMP procedure developed in this paper may be used for guaranteeing that (i) the given combinatorial incidence matrix is realized as an incidence matrix and (ii) the P-resolution and the incidence matrix parameterize the same component in $\mathcal{C}(X)$. Then we show that the map $\phi_{PI} \colon \mathcal{P}(X) \to \mathcal{I}(X)$ is surjective. Since $\phi_P = \phi_I \circ \phi_{PI}$, we can conclude that $\phi_P$ is also surjective. Hence we have proven Kollár conjecture for $X$.

For instance, we prove:

\begin{maintheorem}[Theorem~\ref{theorem:Wpqr-K-conjecture}]
Kollár conjecture holds for $W(p,q,r)$.
\end{maintheorem}

Here $W(p,q,r)$ is a particular weighted homogeneous surface singularity that admits a rational homology disk smoothing, which is introduced by Wahl in 1980's. Weighted homogeneous surface singularities admitting rational homology disk smoothings are classified roughly by Stipsicz--Szabó--Wahl~\cite{SSW-2008} and completely by Bhupal--Stipsicz~\cite{Bhupal-Stipsicz-2011} and Fowler~\cite{Fowler-2013}; but their deformations have not been studied much so far. Here we provide all combinatorial incidence matrices of $W(p,q,r)$ and we show that each combinatorial incidence matrices are induced from P-resolutions except one that is induced from a normal P-modification that gives us the rational homology disk smoothing.

In addition, J. Jeon and D. Shin \cite{Jeon-Shin-2022} prove Kollár conjecture for weighted homogeneous surface singularities with $d \ge t+3$ by the same strategy of this paper.

\subsection{Organization}

We start with reviewing basics of sandwiched surface singularities in Section~\ref{section:sandwiched}. We then define their compatible compactifications and we prove that there is no local-to-global obstruction to deforming them in Section~\ref{section:compactification}. In Section~\ref{section:P-modifications} and Section~\ref{section:semistable-MMP}, we recall the main tools of this paper: P-modifications and the semi-stable MMP. We also introduce Kollár conjecture. One of the main parts of this paper is Section~\ref{section:P-resolution->Picture-deformations}. We will show that one can obtain the corresponding picture deformations from given P-resolutions by applying the semi-stable MMP. From Section~\ref{section:WHSS-usual-sandwiched-structure} we discuss the correspondence between picture deformations and smoothings of negative weights. We begin with basics of weighted homogeneous surface singularities. Then we introduce three different compactifications in Section~\ref{section:compactification-WHSS} and we investigate their relations in Section~\ref{section:Topology-compactifications}. After reviewing basic facts on smoothings of negative weights in Section~\ref{section:smoothings-of-negative-weights}, we show the correspondence of picture deformations and P-resolutions with smoothings of negative weights in Section~\ref{section:Pic-Def-to-SNW} and Section~\ref{section:P-resolution->smoothings-of-negative-weights}. We deal with all deformation theories of cyclic quotient surface singularities in Section~\ref{section:CQSS}. We will prove Kollár conjecture for $W(p,q,r)$ in Section~\ref{section:Wpqr} and Section~\ref{section:Wpqr-K-Conjecture}. Finally, we present various examples in Section~\ref{section:misc-examples}.

\subsection*{Notations}

In the dual graph of the resolution of a complex surface singularity, each node
\begin{equation*}
\tikz \node[bullet] at (0,0) [labelAbove={$-b$}]{};
\end{equation*}
represents a rational curve with self-intersection number $-b$. For simplicity, we may denote the dual graph
\begin{equation}\label{equation:notaion}
\begin{tikzpicture}
\node[bullet] (00) at (0,0) [labelAbove={$-b_1$}]{};

\node[empty] (050) at (0.5,0) [] {};
\node[empty] (10) at (1,0) [] {};

\node[bullet] (150) at (1.5,0) [labelAbove={$-b_r$}] {};

\draw[-] (00)--(050);
\draw[dotted] (050)--(10);
\draw[-] (10)--(150);
\end{tikzpicture}
\end{equation}
by $b_1-\dotsb-b_r$.

We decorate nodes by rectangles if we contract them to singular points.  For instance, we represents a cyclic quotient surface singularity whose dual graph is given in Equation~\eqref{equation:notaion} by
\begin{equation*}
\begin{tikzpicture}
\node[rectangle] (00) at (0,0) [labelAbove={$-b_1$}]{};

\node[empty] (050) at (0.5,0) [] {};
\node[empty] (10) at (1,0) [] {};

\node[rectangle] (150) at (1.5,0) [labelAbove={$-b_r$}] {};

\draw[-] (00)--(050);
\draw[dotted] (050)--(10);
\draw[-] (10)--(150);
\end{tikzpicture}
\end{equation*}
For simplicity again, we may denote it by $[b_1-\dotsb-b_r]$.

A \emph{Hirzebruch-Jung continued fraction} of the pair $(n,b)$ of two positive integers with $n > b$ which are relatively prime is defined by
\begin{equation*}
\frac{n}{b} = b_1 - \cfrac{1}{b_2 - \cfrac{1}{\ddots - \cfrac{1}{b_r}}}
\end{equation*}
and it is denoted by $n/b=[b_1,\dotsc,b_r]$.

\subsection*{Acknowledgements}

The authors of the paper would like to thank Laura Starkston and Giancarlo Urzúa for all useful and helpful comments, especially for pointing out some errors and unclear parts of the previous version of the paper. Thanks to their comments, this extension of the previous version was completed. Heesang Park was supported by Basic Science Research Program through the National Research Foundation of Korea funded by the Ministry of Education: NRF-2021R1F1A1063959. Dongsoo Shin was supported by the National Research Foundation of Korea grant funded by the Korea government: 2018R1D1A1B07048385 and 2021R1A4A3033098. The authors would like to thank the Korea Institute for Advanced Study for warm hospitality when they were associate members in KIAS.

\section{Sandwiched surface singularities}\label{section:sandwiched}

We summarize basics on sandwiched surface singularities and their deformations. Refer Spivakovsky~\cite{Spivakovsky-1990} and de Jong--van Straten~\cite{deJong-vanStraten-1998} for details.

\begin{definition}
A normal surface singularity $(X,p)$ is said to be \emph{sandwiched} if it is analytically isomorphic to a germ of an algebraic surface $X$ that admits a birational map $X \to \mathbb{C}^2$.
\end{definition}

Sandwiched surface singularities are rational. They are characterized by their dual resolution graphs.

\begin{definition}[{Spivakovsky~\cite[Definition~1.9]{Spivakovsky-1990}}]\label{definition:sandwiched-graph}
A graph $\Gamma$ is called a \emph{sandwiched graph} if it is the dual resolution graph of a rational surface singularity that can be blown down to a smooth point by adding new vertices with weights $(-1)$ on the proper places.
\end{definition}

\begin{proposition}[{Spivakovsky~\cite[Proposition~1.11]{Spivakovsky-1990}}]
A normal surface singularity is sandwiched if and only if its dual resolution graph is sandwiched.
\end{proposition}


It has been known that every cyclic quotient surface singularity is sandwiched. Also sandwiched surface singularities include weighted homogeneous surface singularities with big nodes and rational surface singularities with the reduced fundamental cycles.

\begin{example}\label{example:3-4-2}
Let $(X,p)$ be the cyclic quotient surface singularity of type $\frac{1}{19}(1,7)$ and let $\Gamma$ be the dual graph of its minimal resolution of $(X,p)$. We construct a new graph $\Gamma^{\ast}$ by connecting a vertex with weight $(-1)$ to $\Gamma$ as follows:
\begin{equation*}
\begin{tikzpicture}[scale=0.75]
\node[bullet] (10) at (1,0) [labelBelow={$-3$}] {};
\node[bullet] (20) at (2,0) [labelBelow={$-4$}] {};
\node[bullet] (30) at (3,0) [labelBelow={$-2$}] {};

\node[bullet] (1-1) at (1,0.5) [labelAbove={$-1$}] {};
\node[bullet] (175-1) at (1.75,0.5) [labelAbove={$-1$}] {};
\node[bullet] (225-1) at (2.25,0.5) [labelAbove={$-1$}] {};
\node[bullet] (3-1) at (3,0.5) [labelAbove={$-1$}] {};

\draw [-] (10)--(20);
\draw [-] (20)--(30);

\draw [-] (10)--(1-1);
\draw [-] (20)--(175-1);
\draw [-] (20)--(225-1);
\draw [-] (30)--(3-1);
\end{tikzpicture}
\end{equation*}
The graph $\Gamma^{\ast}$ can be blown-down to a smooth point. Therefore $(X,p)$ is sandwiched.
\end{example}

\subsection{Decorated curves}

For a sandwiched surface singularity $(X,p)$, de Jong and van Straten~\cite{deJong-vanStraten-1998} introduce a pair $(C,l)$ of a plane curve singularity $C=\cup_{i=1}^{s} C_i$ and an assignment $l \colon \{C_i \mid i=1,\dotsc,e\} \to \mathbb{N}$ so that the singularity $(X,p)$ is represented by the singularity $X(C,l)$ induced from $(C,l)$. We recall how to get $(X,p)$ from $(C,l)$. For details, see de Jong-van Straten~\cite{deJong-vanStraten-1998}.

\begin{definition}[{de Jong-van Straten~\cite[Definition~1.4]{deJong-vanStraten-1998}}]
A \emph{decorated germ} is a pair $(C,l)$ consisting of a plane curve singularity $C = \cup_{i=1}^{s} C_i \subset \mathbb{C}^2$ passing through the origin and an assignment $l \colon \{C_i\} \to \mathbb{N}$ such that $l(C_i) \ge m(C_i)$, where $m(C_i)$ is the sum of multiplicities of the branch $C_i$ in the multiplicity sequence of the minimal resolution of $C$.
\end{definition}

Let $\widetilde{Z}(C,l) \to \mathbb{C}^2$ be the modification obtained from the minimal embedded resolution of $C$ followed by $l(C_i)-m(C_i)$ consecutive blow-ups at $C_i$.

\begin{definition}
The analytic space $X(C,l)$ is obtained from $Z(C,l) - \widetilde{C}$ by blowing down the union of all exceptional divisors not intersecting the strict transform $\widetilde{C} \subset \widetilde{Z}(C,l)$.
\end{definition}

The analytic space $X(C,l)$ may be smooth or have several singularities. For example, if $C$ is the ordinary cusp $y^2-x^3=0$, then $X(C,2)$ and $X(C,3)$ are smooth.

However de Jong-van Straten~\cite{deJong-vanStraten-1998} prove that if $l(C_i) \ge M(C_i)+1$, where $M(C_i)$ is the sum of multiplicities of the branch $C_i$ in the multiplicity sequence of the minimal good resolution of $C$, then the maximal compact set not intersecting $\widetilde{C}$ is connected. So $X(C,l)$ has only one singularity, which is clearly a sandwiched surface singularity.

Conversely, de Jong and van Straten~\cite{deJong-vanStraten-1998} also prove that every sandwiched surface singularity is represented by $X(C,l)$ for some decorated germ $(C,l)$.

\begin{definition}
For a sandwiched surface singularity $(X,p)$, a decorated germ $(C,l)$ such that $X=X(C,l)$ is called a \textit{decorated curve} of $(X,p)$.
\end{definition}

Using a sandwiched graph structure, one may get a decorated curve $(C,l)$ for $(X,p)$. Since the dual graph of the minimal resolution $(V,E)$ of $(X,p)$ is sandwiched, it is blown down to a smooth point after adding some $(-1)$-vertices. On the other hand, $(V,E)$ can be embedded into a blow-up $(\widetilde{\mathbb{C}}^2, F)$ of $\mathbb{C}^2$ over the origin (including its infinitely near point), where $F$ is the set of the exceptional divisors. For each $(-1)$-curve $F_i \in F$, choose a \emph{curvetta} $\widetilde{C}_i$ (i.e., a small piece of a curve) transverse to $F_i$. We put $\widetilde{C}=\bigcup_{F_i \in F} \widetilde{C}_i$ and $C=\rho(\widetilde{C})=\bigcup_{F_i \in F} C_i$, where $C_i=\rho(\widetilde{C}_i)$. Then $C$ may be considered as a germ of plane curves through the origin $0$. We decorate $C_i$ with the number $l_i$ of the sum of the multiplicities of blowing-up points sitting on the strict transform of $C_i$.

\begin{example}[Continued from Example~\ref{example:3-4-2}]\label{example:3-4-2-decorated-curve}
A decorated curve $(C,l)$ of the cyclic quotient surface singularity of type $\frac{1}{19}(1,7)$ is given by $C=C_1 \cup C_2 \cup C_3 \cup C_4$ with $l_1=2, l_2=3, l_3=3, l_4=4$.
\begin{center}
\includegraphics[scale=1.25]{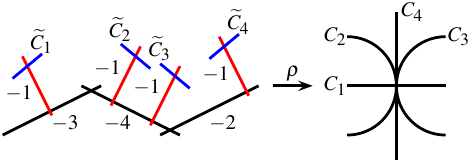} 
\end{center}
\end{example}

Notice that a decorated curve $(C,l)$ of $(X,p)$ is not uniquely determined.

\subsection{Picture deformations}

Any one-parameter deformation of a sandwiched surface singularity $(X,p)$ is induced from a one-parameter deformation of its decorated curve $(C,l)$. We summarize the deformation theory in de Jong-van Straten~\cite{deJong-vanStraten-1998}. Refer also Möhring~\cite{Mohring-2004} for a nice survey.

One may regard the decoration $l$ of a decorated curve $(C,l)$ as the union of the unique subscheme of length $l_i$ supported on the preimage of $0$ on the normalization of $C_i$.

\begin{definition}[{de Jong-van Straten~\cite[Definition~4.2]{deJong-vanStraten-1998}}]\label{definition:one-parameter-deformation}
Let $(C,l)$ be a decorated curve of a sandwiched surface singularity $(X,p)$. Then
A \emph{one-parameter deformation} $(\mathcal{C}, \mathcal{L})$ of $(C,l)$ over a small disk $\Delta$ centered at the origin $0$ consists of

\begin{enumerate}[(1)]
\item a $\delta$-constant deformation $\mathcal{C} \to \Delta$ of $C$, that is, $\delta(C_t)$ is constant for all $t \in \Delta,$
\item a flat deformation $\mathcal{L} \subset \mathcal{C} \times \Delta$ over $\Delta$ of the scheme $l$  such that
\item $\mathcal{M} \subset \mathcal{L}$, where the relative total multiplicity scheme $\mathcal{M}$ of $\mathcal{C} \times \Delta \to \mathcal{C}$ is defined as the closure $\bigcup_{t \in \Delta\setminus{0}} m(C_t)$.
\end{enumerate}
\end{definition}

Here a \emph{$\delta$-constant} of a germ of a plane curve singularity $(C,0) \subset (\mathbb{C}, 0)$ is given by the classical formula:
\begin{equation*}
\delta(C,0) = \sum_{Q} \frac{m(C,Q)(m(C,Q)-1)}{2}
\end{equation*}
where $Q$ varies among all the points infinitely near $0$ and $m(C,Q)$ denotes the multiplicity of the strict transform of $C$ at $Q$.

Each one-parameter deformations of $(C,l)$ induces that of $(X,p)$ and vice versa:

\begin{theorem}[{de Jong-van Straten~\cite[Theorem~4.4]{deJong-vanStraten-1998}}]
For any one-parameter deformation $(\mathcal{C}, \mathcal{L}) \to \Delta$ of a decorated curve $(C,l)$ of a sandwiched surface singularity $(X,p)$, there exists a flat one-parameter deformation $\mathcal{X} \to \Delta$ of $(X,p)$ such that

\begin{enumerate}[(1)]
\item $X_0=X$,
\item $X_t=X(C_t, l_t)$ for all $t \in \Delta \setminus 0$.
\end{enumerate}
Moreover, every one-parameter deformation of $(X,p)$ is obtained in this way.
\end{theorem}
Here $X(C_t,l_t)$ denotes the blow-up of $\mathbb{C}^2$ in the ideals $I(C_{t,p}, l_{t,p}) \subset \mathcal{O}_{\mathbb{C}^2,p}$ where $p$ ranges over all points of $C_t$ where $l_t$ is not zero.

Especially, a special deformation of a decorated curve $(C,l)$ induces a smoothing of $(X,p)$:

\begin{definition}[{de Jong-van Straten~\cite[Definition~4.6]{deJong-vanStraten-1998}}]
A one-parameter deformation $(\mathcal{C}, \mathcal{L})$ of $(C,l)$ is called a \emph{picture deformation} if for generic $t \in T \setminus 0$ the divisor $l_t$ on $\widetilde{C}_t$ is reduced.
\end{definition}

\begin{proposition}[{de Jong-van Straten~\cite[Lemma~4.7]{deJong-vanStraten-1998}}]
A generic smoothing of $(X,p)$ is realised by a picture deformation of its decorated curve $(C,l)$.
\end{proposition}

Möhring~\cite{Mohring-2004} describes a geometric way to construct a deformation of $(X,p)$ from that of $(C,l)$ via a blow-up along a special complete ideal. For this, we denote $\Sigma(\mathcal{J})$ for a coherent ideal sheaf $\mathcal{J} \subset \mathcal{O}_X$ the complex subspace of $X$ defined by $\mathcal{J}$; i.e., $\Sigma(\mathcal{J}) = (V(\mathcal{J}), \mathcal{O}_X/\mathcal{J}|_{V(\mathcal{J})})$. Then:

\begin{theorem}[{Möhring~\cite[Theorem~3.5.1]{Mohring-2004}}]\label{theorem:Mohring}
Let $(\mathcal{C}, \mathcal{L}) \to \Delta$ be a one-parameter deformation of $(C,l)$ over a small disk $(\Delta,0) \subset (\mathbb{C},0)$  and let $\mathcal{J}$ be a coherent sheaf of ideals on a well chosen neighborhood $U$ of zero in $\mathbb{C}^2 \times \Delta$ with fibers $I(C_t,l_t)$. Then the blow-up of $\mathbb{C}^2 \times \Delta$ respectively $U$ in $\Sigma(\mathcal{J})$ is flat over $\Delta$ with fibers $X(C_t,l_t)$ for all $t \in \Delta$.
\end{theorem}

The Milnor fibre of a smoothing of $(X,p)$ is given as follows. At first, we choose a closed Milnor ball $B(0,\epsilon)$ for the germ $(C,0)$. For a sufficiently small $t \neq 0$, a general fiber $C_t$ over $t$ will have a representative in the ball $B(0,\epsilon)$ denoted by $B_t$.

\begin{theorem}[{de Jong-van Straten~\cite[Proposition~5.1]{deJong-vanStraten-1998}}]
The Milnor fibre of a smoothing of $(X,p)$ corresponding to a picture deformation $(\mathcal{C},\mathcal{L})$ of its decorated curve $(C,l)$ is diffeomorphic to the complement of the strict transform $\widetilde{C}_t$ of $C_t$ on $B_t$ blown up in the points corresponding to $l_t$.
\end{theorem}

By the definition of a picture deformation of $(C,l)$, the singularities of a general fiber $C_t$ for $t \neq 0$ are only ordinary multiple points. So it is easy to draw a picture of $C_t$.

\begin{example}[Continued from Example~\ref{example:3-4-2-decorated-curve}]\label{example:3-4-2-picture-deformation}
There are three picture deformations of the decorated curve $(C,l)$ in Example~\ref{example:3-4-2-decorated-curve}, where red dots are the subscheme $l_s$ for $s \neq 0$.
\begin{center}
\includegraphics[scale=1.25]{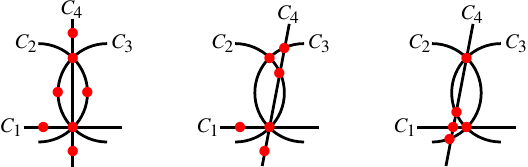} 
\end{center}
The Milnor numbers of the Milnor fibers corresponding to the above picture deformations are $3$, $2$, $1$, respectively.
\end{example}

\subsection{Incidence matrices and combinatorial incidence matrices}\label{section:incidence-matrices}

The combinatorial feature of picture deformations of $X(C,l)$ may be encoded by certain matrices. Let $(\mathcal{C}, \mathcal{L})$ be a picture deformation of $(C,l)$. Suppose that $C=\cup_{i=1}^{s} C_i$ and $C_t = \cup_{i=1}^{s} C_{i,t}$. Denote by $\{P_1,\dotsc,P_n\}$ the images in $B_t$ of the points in the support of $l_t$.

\begin{definition}[{de Jong-van Straten~\cite[p.~483]{deJong-vanStraten-1998}}]\label{definition:incidence-matrix}
The \emph{incidence matrix} of a picture deformation $(\mathcal{C}, \mathcal{L})$ is the matrix $I(\mathcal{C}, \mathcal{L}) \in M_{s,n}(\mathbb{Z})$ whose $(i,j)$ entry is equal to the multiplicity of $P_j$ as a point of $C_{i,t}$.
\end{definition}

\begin{remark}\label{remark:incidence-matrix}
Since a general fiber $X(C_t,l_t)$ is given by the blow-up of $(C_t,l_t)$ along the support $l_t$, the incidence matrix records how the $(-1)$-curves on $X(C_t,l_t)$ intersects the curves $\widetilde{C}_{i,t}$.
\end{remark}

\begin{example}[Continued from Example~\ref{example:3-4-2-picture-deformation}]\label{example:3-4-2-incidence-matrices}
Let $(X,p)$ be a cyclic quotient surface singularity $\frac{1}{19}(1,7)$.
The incidence matrices corresponding to the picture deformations are
\begin{equation*}
\begin{bmatrix}
0 & 0 & 0 & 0 & 0 & 1 & 1\\
0 & 0 & 0 & 1 & 1 & 0 & 1 \\
0 & 0 & 1 & 0 & 1 & 0 & 1 \\
1 & 1 & 0 & 0 & 1 & 0 & 1 \\
\end{bmatrix},
\begin{bmatrix}
0 & 0 & 0 & 0 & 1 & 1 \\
0 & 0 & 1 & 1 & 0 & 1 \\
0 & 1 & 0 & 1 & 0 & 1 \\
1 & 1 & 1 & 0 & 0 & 1 \\
\end{bmatrix},
\begin{bmatrix}
0 & 0 & 0 & 1 & 1 \\
0 & 1 & 1 & 0 & 1 \\
1 & 0 & 1 & 0 & 1 \\
1 & 1 & 1 & 1 & 0 \\
\end{bmatrix}
\end{equation*}
\end{example}

Since every picture deformation gives us an incidence matrix, we have the following well-defined map:

\begin{definition}
Let $\mathcal{C}(X)$ be the set of irreducible components of the reduced versal deformation space $\Def(X)$ and let $\mathcal{I}(X)$ be the set of all incidence matrices of $X$ of a given sandwiched structure. The \emph{incidence map} of $X$ is a map
\begin{equation*}
\phi_{I} \colon \mathcal{C}(X) \to \mathcal{I}(X),
\end{equation*}
where, for each $S \in \mathcal{C}(X)$, $\phi_I(S)$ is defined by the incidence matrix corresponding to a picture deformation that parametrizes $S$.
\end{definition}

One of the fundamental problems is:

\begin{problem}[{de Jona--van Straten~\cite[p.485]{deJong-vanStraten-1998}}]\label{problem:phi_I}
Determine when $\phi_I$ is injective.
\end{problem}

We will prove in Theorem~\ref{theorem:phi_I-injective} that $\phi_I$ is injective for certain weighted homogeneous surface singularities.

Every incidence matrix satisfies the following necessary conditions given in de Jong--van Straten~\cite[p.484]{deJong-vanStraten-1998}, which are induced from Definition~\ref{definition:one-parameter-deformation} and the fact that $C_t$($t \neq 0$) has only ordinary singularities:
\begin{equation}\label{equation:combinatorial-incidence-matrix}
\begin{aligned}
&\text{$\sum_{j=1}^{n} \frac{m_{ij}(m_{ij}-1)}{2} = \delta_i$ for all $i$};\\
&\text{$\sum_{j=1}^{n} m_{ij} m_{kj} = C_i \cdot C_k$ for all $i \neq k$};\\
&\text{$\sum_{j=1}^{n} m_{ij}=l_i$ for all $i$},
\end{aligned}
\end{equation}
where $\delta_i$ is the $\delta$-invariant of the branch $C_i$ and $C_i \cdot C_k$ is the intersection multiplicity at $0$ of the branches $C_i$ and $C_k$.

\begin{definition}\label{definition:combinatorial-incidence-matrix}
A \emph{combinatorial incidence matrix} of $X$ with a sandwiched structure $(C,l)$ is a matrix satisfying Equation~\ref{equation:combinatorial-incidence-matrix}. We denote by $C\mathcal{I}(X)$ the set of all combinatorial incidence matrices of $X$.
\end{definition}

It is also an open problem whether a combinatorial incidence matrix can be realized by a picture deformation of $X$; cf.~de Jong--van Straten~\cite[p.485]{deJong-vanStraten-1998}.

\begin{problem}\label{problem:=}
Determine when $\mathcal{I}(X)=C\mathcal{I}(X)$.
\end{problem}

Jeon and Shin~\cite{Jeon-Shin-2022} proved that the equality $\mathcal{I}(X)=C\mathcal{I}(X)$ holds for certain weighted homogeneous surface singularities.

\section{Compactifications compatible with deformations}\label{section:compactification}

We introduce a compactification of a sandwiched surface singularity $(X,p)$ induced from its sandwiched structure which is ``compatible'' with its deformations given by picture deformations.

Let's choose and fix a decorated curve $(C,l)$ of $(X,p)$. Let $C=\cup_{i=1}^{s} C_i$. Each member $C_i$ is a small piece of a plane curve in $\mathbb{C}^2$ passing through the origin $(0,0)$. Let's consider $\mathbb{C}^2$ as an affine open subset of $\mathbb{CP}^2$ (say, $z_2 \neq 0$ in $\{[z_0,z_1,z_2] \in \mathbb{CP}^2\})$. Then there is a projective plane curve $D_i \subset \mathbb{CP}^2$ such that $D_i \cap \mathbb{C}^2 = C_i$. Let $D=\cup_{i=1}^{s} D_i$.

We construct a projective singular surface $(Y,p)$ from the pair $(D,l)$ in a similar way as we did for $(X,p)$ from $(C,l)$. Explicitly, we first blow up $\mathbb{CP}^2$ on $[0,0,1]$ (including its infinitely near points) so that we get the minimal embedded resolution of $D$ over $[0,0,1]$. We next blow up consecutively at the each branch $D_i$ by $l_i-m_i$ times. Then we get the modification $W \to \mathbb{CP}^2$. Since we assume that $l_i \ge M_i+1$, the union of the exceptional components that do not meet the strict transform of $D$ forms a connected configuration of curves, which is exactly the exceptional divisor $E$ of the minimal resolution of $(X,p)$. The projective singular surface $(Y,p)$ is obtained by contracting the exceptional divisor $E$. To sum up, if $(V,E) \to (X,p)$ is the minimal resolution of $(X,p)$, then we have a diagram
\begin{equation*}
\begin{tikzcd}
\tikzcdset{font=\normalfont}
(V,E) \arrow[hookrightarrow]{r} \arrow{d} & (W,E) \arrow{d} \\
(X,p) \arrow[hookrightarrow]{r} & (Y,p)
\end{tikzcd}
\end{equation*}

\begin{definition}
The projective surface $(Y,p)$ constructed as above is called a \emph{compatible compactification} of $(X,p)$ corresponding to $(C,l)$. The pair $(D,l)$ is called the \emph{compactified decorated curve} of $(Y,p)$.
\end{definition}

\subsection{Compatibility with deformations}

The compactification $(Y,p)$ of $(X,p)$ is compatible with deformations of $(X,p)$. That is, we will show that every smoothing of $X$ can be extended to that of $Y$ in a way that any compactified decorated curves do not disappear during the deformations of $Y$.

\begin{theorem}\label{theorem:extension-of-deformation}
Any deformation of $(X,p)$ can be extended to a deformation of $(Y,p)$.
\end{theorem}

The proof relies on the following basic fact.

\begin{proposition}[cf.~{Wahl~\cite[Proposition~6.4]{Wahl-1981}}]
Let $Y$ be a normal projective surface with a normal surface singularity $p \in Y$. If $H^2(Y, \sheaf{T_Y})=0$, then any deformation of the singularity $p$ may be realized by a deformation of $Y$.
\end{proposition}

Let $E$ be the exceptional divisor of the contraction $\pi \colon (W,E) \to (Y,p)$ and let $E \cup F$ be the exceptional divisor of the sequence of blow-ups $W \to \mathbb{CP}^2$. That is, $F$ is the union of the $(-1)$-curves in $W$.

\begin{proposition}\label{propisition:H2(Y,T_Y)=H2(W,T_W(-log(E)))}
We have
\begin{equation}\label{equation:H2(Y,T_Y)=H2(W,T_W(-log(E)))}
H^2(Y, \sheaf{T_Y})=H^2(W, \sheaf{T_W}(-\log{E})).
\end{equation}
\end{proposition}

\begin{proof}
We first show that
\begin{equation}\label{equation:H2(Y,T_Y)=H2(W,pi_*T_W(-log(E)))}
H^2(Y, \sheaf{T_Y}) = H^2(Y, \pi_{\ast}{\sheaf{T_W}(-\log{E})}).
\end{equation}
From the standard exact sequence
\begin{equation*}
0 \to \sheaf{T_W}(-\log{E}) \to \sheaf{T_W} \to \oplus_{i} \sheaf{N_{E_i/W}} \to 0
\end{equation*}
we have an exact sequence
\begin{equation*}
0 \to \pi_{\ast}{\sheaf{T_W}(-\log{E})} \to \pi_{\ast}{\sheaf{T_W}} \to \Delta
\end{equation*}
where $\Delta$ is supported on the singular point $p$. By restricting on the image, we have a short exact sequence
\begin{equation*}
0 \to \pi_{\ast}{\sheaf{T_W}(-\log{E})} \to \pi_{\ast}{\sheaf{T_W}} \to \Delta' \to 0
\end{equation*}
where $\Delta'$ is still supported on the singular point $p$. Taking long exact sequence, we have
\begin{equation*}
H^2(Y, \pi_{\ast}{\sheaf{T_W}}) = H^2(Y, \pi_{\ast}{\sheaf{T_W}(-\log{E})}).
\end{equation*}
On the other hand, there is a natural isomorphism between $\pi_{\ast}{\sheaf{T_W}}$ and $\sheaf{T_Y}$ by Burns-Wahl~\cite[Proposition~1.2]{Burns-Wahl-1974}. Therefore we have Equation~\eqref{equation:H2(Y,T_Y)=H2(W,pi_*T_W(-log(E)))}.

We next show that
\begin{equation}\label{equation:H2(Y,pi_*T_W(-log(E)))=H2(W,T_W(-log(E)))}
H^2(Y, \pi_{\ast}\sheaf{T_W}(-\log{E})) = H^2(W, \sheaf{T_W}(-\log{E})),
\end{equation}
from which combined with Equation~\eqref{equation:H2(Y,T_Y)=H2(W,pi_*T_W(-log(E)))} we have Equation~\eqref{equation:H2(Y,T_Y)=H2(W,T_W(-log(E)))} as desired.

From Leray spectral sequence, there is an exact sequence
\begin{equation*}
\begin{split}
0 &\to H^1(Y, \pi_{\ast}{\sheaf{T_W}(-\log{E})}) \to H^1(W, \sheaf{T_W}(-\log{E})) \to H^0(Y, R^1\pi_{\ast}{\sheaf{T_W}(-\log{E})}) \\
&\to H^2(Y, \pi_{\ast}\sheaf{T_W}(-\log{E})) \to H^2(W, \sheaf{T_W}(-\log{E})) \to H^0(Y, R^2\pi_{\ast}{\sheaf{T_W}(-\log{E})})
\end{split}
\end{equation*}
Let $V$ be a regular neighborhood of $E$. In the above sequence, we have
\begin{equation*}
H^0(Y, R^i\pi_{\ast}{\sheaf{T_W}(-\log{E})})=H^i(V,\sheaf{T_V}(-\log{E}))
\end{equation*}
for $i=1,2$ by Hartshorne~\cite[III~Proposition~8.5]{Hartshorne-1977}. Furthermore, since $V$ is a deformation retract of the one-dimensional exceptional divisor $E$, we have
\begin{equation*}
H^2(V,\sheaf{T_V}(-\log{E}))=0
\end{equation*}
So Equation~\eqref{equation:H2(Y,pi_*T_W(-log(E)))=H2(W,T_W(-log(E)))} follows if one can prove:

\textbf{Claim A}. $H^1(W, \sheaf{T_W}(-\log{E})) \to H^1(V, \sheaf{T_V}(-\log{E}))$ is surjective.

Let $L=\{z_2=0\}$ be the line at infinity in $W$. The local cohomology $H^2_L(\sheaf{T_W}(-\log{E}))$ fits into the exact sequence
\begin{equation*}
H^1(W, \sheaf{T_W}(-\log{E})) \to H^1(W \setminus L, \sheaf{T_W}(-\log{E})) \to H^2_L(\sheaf{T_W}(-\log{E})).
\end{equation*}
On the other hand, $V \subset W \setminus L$. Furthermore, $W \setminus L$ is a deformation retract of $V$ because $V$ is a regular neighborhood of $E$. So $H^1(W \setminus L, \sheaf{T_W}(-\log{E}))=H^1(V, \sheaf{T_V}(-\log{E}))$. Therefore we have an exact sequence
\begin{equation*}
H^1(W, \sheaf{T_W}(-\log{E})) \to H^1(V, \sheaf{T_V}(-\log{E})) \to H^2_L(\sheaf{T_W}(-\log{E})).
\end{equation*}
By excision principle, we have $H^2_L(\sheaf{T_W}(-\log{E}))=H^2_L(\sheaf{T_W})$. So for proving the above Claim~A, it is enough to show:

\textbf{Claim B}. $H^2_L(\sheaf{T_W})=0$.

Notice that
\begin{equation*}
H^2_L(\sheaf{T_W}) = \varinjlim \Ext^2(\sheaf{O_{nL}}, \sheaf{T_W})
\end{equation*}
where the limit is taken over all positive integers $n \ge 1$. From the exact sequence
\begin{equation*}
0 \to \sheaf{O_W}(-nL) \to \sheaf{O_W} \to \sheaf{O_{nL}} \to 0
\end{equation*}
we have
\begin{equation*}
\Ext^1(\sheaf{O_W}, \sheaf{T_W}) \xrightarrow{\phi} \Ext^1(\sheaf{O_W}(-nL), \sheaf{T_W}) \to \Ext^2(\sheaf{O_{nL}}, \sheaf{T_W}) \to 0
\end{equation*}
because $\Ext^2(\sheaf{O_W},\sheaf{T_W})=H^2(W, \sheaf{T_W})=0$. On the other hand, the map $\phi$ fits into the exact sequence
\begin{equation*}
H^1(W, \sheaf{T_W}) \xrightarrow{\phi} H^1(W, \sheaf{T_W} \otimes \sheaf{O_W}(nL)) \to H^1(W, \sheaf{T_W} \otimes \sheaf{O_{nL}}(nL)).
\end{equation*}
So if one can show that $H^1(W, \sheaf{T_W} \otimes \sheaf{O_{nL}}(nL))=0$ for all $n \ge 1$, then $\phi$ is surjective for all $n \ge 1$, which implies $\Ext^2(\sheaf{O_{nL}}, \sheaf{T_W})=0$ for all $n \ge 1$. So Claim~B follows.

\textbf{Claim C}. $H^1(W, \sheaf{T_W} \otimes \sheaf{O_{nL}}(nL))=0$ for all $n \ge 1$.

From the decomposition sequence
\begin{equation*}
0 \to \sheaf{O_{nL}}(-L) \to \sheaf{O_{(n+1)L}} \to \sheaf{O_L} \to 0
\end{equation*}
tensored with $\sheaf{T_W}((n+1)L)$, we have an exact sequence
\begin{equation}\label{equation:surjective-map}
H^1(W, \sheaf{T_W} \otimes \sheaf{O_{nL}}(nL)) \to H^1(W, \sheaf{T_W} \otimes \sheaf{O_{(n+1)L}}((n+1)L)) \to H^1(W, \sheaf{T_W} \otimes \sheaf{O_L}((n+1)L))
\end{equation}

Here we first show that
\begin{equation}\label{equation:H1(W,T_W*O_L(mL))=0}
H^1(W, \sheaf{T_W} \otimes \sheaf{O_L}(mL))=0
\end{equation}
for all $m \ge 1$: From the tangent-normal sequence
\begin{equation*}
0 \to \sheaf{T_L} \to \sheaf{T_W} \otimes \sheaf{O_L} \to \sheaf{N_{L/W}} \to 0
\end{equation*}
tensored with $\sheaf{O_L}(mL)$, we have an exact sequence
\begin{equation*}
H^1(L, \sheaf{T_L}(mL)) \to H^1(W, \sheaf{T_W} \otimes \sheaf{O_L}(mL)) \to H^1(L, \sheaf{O_L}((m+1)L).
\end{equation*}
But $H^1(L, \sheaf{T_L}(mL))=H^1(L, \sheaf{O_L}((m+1)L)=0$ because $L \cong \mathbb{CP}^1$ and $L \cdot L = 1$. So we have $H^1(W, \sheaf{T_W} \otimes \sheaf{O_L}(mL))=0$ for all $m \ge 1$.

Then we always have a surjective map
\begin{equation*}
H^1(W, \sheaf{T_W} \otimes \sheaf{O_{nL}}(nL)) \to H^1(W, \sheaf{T_W} \otimes \sheaf{O_{(n+1)L}}((n+1)L)) \to 0.
\end{equation*}
for all $n \ge 1$ from the above exact sequence in \eqref{equation:surjective-map}. On the other hand, for $n=1$, we have $H^1(W, \sheaf{T_W} \otimes \sheaf{O_L}(L))=0$ as proven above in \eqref{equation:H1(W,T_W*O_L(mL))=0}. Therefore we have $H^1(W, \sheaf{T_W} \otimes \sheaf{O_L}(nL))=0$ for all $n \ge 1$ as asserted in Claim~C.
\end{proof}

\begin{remark}
Proposition~\ref{propisition:H2(Y,T_Y)=H2(W,T_W(-log(E)))} may be regarded as a sandwiched version of Lee--Park~\cite[Theorem~2]{Lee-Park-2007}, where they dealt with only quotient surface singularities.
\end{remark}

We now prove Theorem~\ref{theorem:extension-of-deformation}.

\begin{proof}[Proof of Theorem~\ref{theorem:extension-of-deformation}]
By Proposition~\ref{propisition:H2(Y,T_Y)=H2(W,T_W(-log(E)))}, we will show that $H^2(W, \sheaf{T_W}(-\log{E}))=0$. Let $F$ be the divisor consisting of disjoint $(-1)$-curves in $W$. There is a surjective map
\begin{equation*}
H^2(W, \sheaf{T_W}(-\log(E+F))) \to H^2(W, \sheaf{T_W}(-\log{E})) \to 0,
\end{equation*}
which follows from the exact sequence
\begin{equation*}
0 \to \sheaf{T_W}(-\log(E+F)) \to \sheaf{T_W}(-\log{E}) \to \oplus_i \sheaf{N}_{F_i/W} \to 0.
\end{equation*}
But blowing ups and downs do not change $h^2$ of logarithmic tangent sheaves; cf.~Flenner-Zaidenberg~\cite[Lemma~1.5]{Flenner-Zaidenberg-1994} for example. So we have
\begin{equation*}
h^2(W, \sheaf{T_W}(-\log(E+F)))=h^2(\mathbb{CP}^2, \sheaf{T_{\mathbb{CP}^2}})=0
\end{equation*}
which implies that $H^2(W, \sheaf{T_W}(-\log{E}))=H^2(Y, \sheaf{T_Y})=0$, as we asserted.
\end{proof}

\subsection{Deformations preserving compactified decorated curves}

Since there is no local-to-global obstruction for the compatible compactification $(Y,p)$ of $(X,p)$, every smoothing of $(X,p)$ can be extended to that of $(Y,p)$.

\begin{theorem}\label{theorem:stablility}
For each smoothing of $(X,p)$, there is an induced smoothing of $(Y,p)$ such that the compactified decorated curves of $(Y,p)$ are not disappeared by the smoothing.
\end{theorem}

\begin{proof}
By Theorem~\ref{theorem:Mohring}, every smoothing of $(X,p)$ is induced from the blowing up of $\mathbb{C}^2 \times \Delta$ along the ideal $\mathcal{J}$. Here each fiber $\mathbb{C}^2_t$ over $t(\neq 0) \in \Delta$ includes a general fiber $C_t$ of the corresponding picture deformation $(\mathcal{C},\mathcal{L})$ of the decorated curve $(C,l)$. Homogenizing the plane curves $C_t$, we have a deformation of the compactified decorated curve $(D,l)$ inside $\mathbb{CP}^2 \times \Delta$. Blowing up $\mathbb{CP}^2 \times \Delta$ along the ideal $\mathcal{J}$, we have a smoothing of $(Y,p)$ that contains the given smoothing of $(X,p)$ and any compactified decorated curves do not disappear during the smoothing of $(Y,p)$.
\end{proof}

\begin{definition}
A \emph{compactified picture deformation} of $(D,l)$ induced from a picture deformation of $(C,l)$ is the deformation of $(D,l)$ given as in the proof of the above theorem.
\end{definition}

We can recover the incidence matrix corresponding to a picture deformation of $(C,l)$ from its compactified picture deformation of $(D,l)$.

\section{Kollár conjecture, P-modifications, and P-resolutions}\label{section:P-modifications}

We recall basics on Kollár conjecture and the related topics. We refer to Kollár~\cite{Kollar-1991}, Kollár-Shepherd-Barron~\cite{KSB-1988}, and Behnke-Christophersen~\cite{Behnke-Christophersen-1994} for details.

\subsection{Kollár conjecture}

Kollár and Shepherd-Barron~\cite{KSB-1988} described the irreducible components of the reduced miniversal deformation space of a quotient singularity $(X,p)$ in terms of certain partial resolutions of $(X,p)$. Then Kollár~\cite{Kollar-1991} proposed the following conjecture as a generalization.

\begin{conjecture}[{Kollár~\cite[6.2.1]{Kollar-1991}}]\label{conjecture:Kollar-conjecture-original}
Let $(X,p)$ be a rational surface singularity and let $\mathcal{X}$ be the total space of a one-parameter smoothing of $(X,p)$. Then the canonical algebra
\begin{equation*}
\sum_{n=0}^{\infty} \mathcal{O}_{\mathcal{X}}(n K_{\mathcal{X}})
\end{equation*}
is a finitely generated $\mathcal{O}_{\mathcal{X}}$-algebra.
\end{conjecture}

If the conjecture is true, then the Proj of the above canonical algebra gives a specific partial modification $f \colon U \to X$.

\begin{definition}[{Kollár~\cite[Definition~6.2.10']{Kollar-1991}}]
Let $(X,p)$ be a rational surface singularity and let $f \colon U \to X$ be a proper modification. $U$ is called a \emph{P-modification} if
\begin{enumerate}[(i)]
\item $R^1f_{\ast}{\mathcal{O}_U}=0$,
\item $K_U$ is $f$-ample,
\item $U$ has a smoothing which induces a $\mathbb{Q}$-Gorenstein smoothing of each singularity of $U$.
\end{enumerate}
\end{definition}

Then Stevens~\cite{Stevens-2003} reformulates  Conjecture~\ref{conjecture:Kollar-conjecture-original} using P-modifications as follows.

\begin{conjecture}[Kollár Conjecture; cf.~{Stevens~\cite[p.114]{Stevens-2003}}]
Let $\pi \colon \mathcal{X} \to S$ be the total family over a smoothing component of the versal deformation of a rational surface singularity $X$. Then there exists a proper modification $\mathcal{U}$ of $\mathcal{X}$ with Stein factorisation $\mathcal{U} \xrightarrow{\pi'} T \xrightarrow{\sigma} S$, such that some multiple of $K_{\mathcal{U}}$ is Cartier, and $K_{\mathcal{U}}$ is $\pi'$-ample; $\pi'$ is flat, and every fibre $\mathcal{U}_t$ over $t \in \sigma^{-1}(0)$ is a P-modification of $X$, and the germ $(T, t)$ is a component of the versal deformation of $\mathcal{U}_t$.
\end{conjecture}

The conjecture is true for only a few cases; see Stevens~\cite[\S14]{Stevens-2003}. In particular, Kollár and Shepherd-Barron~\cite{KSB-1988} showed that the conjecture holds for quotient surface singularities. We give an example of proof of Kollár conjecture for particular singularities in Section~\ref{section:Wpqr-K-Conjecture}.

Each $P$-modification parametrizes one component of $\Def(X)$; cf.~Kollár--Shepherd-Barron~\cite[Theorem 3.9.]{KSB-1988}. So we have a well-defined map
\begin{equation*}
\phi_P \colon \mathcal{P}(X) \to \mathcal{C}(X)
\end{equation*}
from the set $\mathcal{P}(X)$ of all P-modifications of $X$ to $\mathcal{C}(X)$. We may simplify Kollár Conjecture as follows.

\begin{problem}\label{problem:phi_P}
Determine whether $\phi_P \colon \mathcal{P}(X) \to \mathcal{C}(X)$ is surjective.
\end{problem}

\subsection{T-singularities}

In case of cyclic quotient surface singularites, the corresponding P-modifications are normal and they have only \emph{T-singularities} as singularities, which are special cyclic quotient surface singularities.
We briefly recall basics on T-singularities.

\begin{definition}[T-singularity]
A \emph{T-singularity} is a quotient surface singularity that admits a $\mathbb{Q}$-Gorenstein one-parameter smoothing.
\end{definition}

Kollár and Shepherd-Barron~\cite{KSB-1988} determined all T-singularities.

\begin{proposition}[{KSB~\cite[Proposition~3.10]{KSB-1988}}]
A T-singularity is either a rational double point or a cyclic quotient surface singularity $\frac{1}{dn^2}(1, dna-1)$ with $d \ge 1$, $n \ge 2$, $1 \le a < n$, and $(n,a)=1$.
\end{proposition}

Due essentially to Wahl~\cite{Wahl-1981}, a T-singularity may be recognized from its minimal resolution:

\begin{proposition}\hfill
\label{proposition:T-algorithm}
\begin{enumerate}[(i)]
\item The singularities $[4]$ and $[3,2,\dotsc,2,3]$ are of class T

\item If $[b_1,\dotsc,b_r]$ is a T-singularity, then $[b_1+1,b_2,\dotsc,b_r,2]$ and $[2,b_1,\dotsc,b_{r-1},b_r+1]$ are also T-singularities.

\item Every singularity of class T that is not a rational double point can be obtained by starting with one of the singularities described in (i) and iterating the steps described in (ii).
\end{enumerate}
\end{proposition}

Among T-singularities,

\begin{definition}[Wahl singularity]
A \emph{Wahl singularity} is a cyclic quotient surface singularity $\frac{1}{n^2}(1, na-1)$.
\end{definition}

A Wahl singularity is a T-singularity that can be obtained by iterating the steps described in Proposition~\ref{proposition:T-algorithm} (ii) starting from $[4]$.

The Riemenschneider's dot diagram for the minimal resolution of a Walh singularity is symmetric to a special dot in the diagram which is called the \emph{$\delta$-dot}. We briefly explain this phenomenon. The Riemenschneider's diagram for $[4]$ is
\begin{equation*}
\begin{tikzpicture}[scale=0.5]
\tikzset{label distance=-0.15em}
\tikzset{font=\scriptsize}
\node[bullet] (00) at (0,0) [] {};
\node[bullet] (10) at (1,0) [label=above:$\delta$] {};
\node[bullet] (20) at (2,0) [] {};
\end{tikzpicture}
\end{equation*}
which is symmetric to the dot decorated by $\delta$. And the step (ii) in Proposition~\ref{proposition:T-algorithm} above can be described in the view of Riemenschneider's dot diagram as follows: The procedure from $[b_1,\dotsc,b_r]$  to $[b_1+1,b_2,\dotsc,b_r,2]$ is just equivalent to adding a dot to the left of the first dot of the top row and a dot under the last dot of the final row.  Also the procedure from $[b_1,\dotsc,b_r]$ to $[2,b_1,\dotsc,b_{r-1},b_r+1]$ can be understood as  adding a dot to the right of the last dot of the final row and adding a dot over the first dot of the first row in the given Riemenschneider's dot diagram. Therefore the Riemenschneider's diagram of a Wahl singularity is still symmetric to the dot decorated by $\delta$.

\begin{example}
The Riemenschneider's dot diagram for the Wahl singularity $\frac{1}{49}(1,34)=[2,2,5,4]$ is symmetric to the dot decorated by $\delta$:
\begin{equation*}
\begin{tikzpicture}[scale=0.5]
\tikzset{font=\scriptsize}
\node[bullet] at (0,3) [labelAbove={}] {};

\node[bullet] at (0,2) [] {};

\node[bullet] at (0,1) [] {};
\node[bullet] at (1,1) [] {};
\node[bullet] at (2,1) [label=above:{$\delta$}] {};
\node[bullet] at (3,1) [] {};

\node[bullet] at (3,0) [] {};
\node[bullet] at (4,0) [] {};
\node[bullet] at (5,0) [] {};

\draw [-] (0,3)--(0,2)--(0,1)--(1,1);
\draw [-] (3,1)--(3,0)--(4,0)--(5,0);
\end{tikzpicture}
\end{equation*}
\end{example}

\begin{definition}[Initial curve]
The \emph{initial curve} of a Wahl singularity is the exceptional curve of its minimal resolution that has the $\delta$-dot in its Riemenschneider dot diagram.
\end{definition}

\begin{example}
The initial curve of a Wahl singularity $[2,2,5,4]$ is the exceptional curve $C$ with $C \cdot C = -5$.
\end{example}

\subsection{$\mathbb{Q}$-Gorenstein deformations}

A Wahl singularity $(X,p)$ is given by
\begin{equation*}
(X,p) \cong \mathbb{C}^2_{x,y}/\dfrac{1}{n^2}(1,a)
\end{equation*}
We have another description
\begin{equation*}
(X,p) \cong (\{\xi \eta = \zeta^n\} \subset \mathbb{C}^3_{\xi,\eta,\zeta})/\dfrac{1}{n}(1,-1,a)
\end{equation*}
where $\xi=x^n$, $\eta=y^n$, $\zeta=xy$. Then a deformation $\mathcal{X} \to \Delta$ of $(X,p)$ is of the form
\begin{equation*}
\{\xi \eta = \zeta^n + t^{\alpha}h(t)\} \subset (\mathbb{C}^3_{\xi,\eta,\zeta}/\dfrac{1}{n}(1,-1,a) \times \mathbb{C}^1_t)
\end{equation*}
for some $\alpha \in \mathbb{N}$ and a convergent power series $h(t)$ with $h(0) \neq 0$. So there is an isomorphism of germs
\begin{equation*}
(p \in \mathcal{X}) \cong (0 \in \{\xi \eta = \zeta^n + t^{\alpha}\} \subset (\mathbb{C}^3_{\xi,\eta,\zeta}/\dfrac{1}{n}(1,-1,a) \times \mathbb{C}^1_t))
\end{equation*}

\begin{definition}[{Mori~\cite[Definition~1a.5]{Mori-1988}}]
The number $\alpha$ is called the \emph{axial multiplicity} of $p \in \mathcal{X}$.
\end{definition}

\subsection{P-resolutions and M-resolutions}

Kollár and Shepherd-Barron~\cite{KSB-1988} showed that the canonical model $\mathcal{U}$ of a one-parameter smoothing $\mathcal{X}$ of a quotient surface singularity $(X,p)$ gives a small modification whose central fiber $U$ is a specific normal P-modification of $X$ that admits only T-singularities as its singularities, called a \emph{P-resolution}.

\begin{definition}[P-resolution; {cf.Kollár-Mori~\cite{KM-1992}}]
Let $(X,p)$ be any rational surface singularity. A \emph{P-resolution} $f \colon U \to X$ be a proper birational morphism $f$ such that $U$ is normal with only T-singularities and $K_{U/X}$ is $f$-ample.
\end{definition}

Every P-resolution of a cyclic quotient surface singularity is dominated by the so-called \emph{maximal resolution} of the singularity (KSB~\cite[Lemma~3.14]{KSB-1988}) whose dual graph is a linear chain.

\begin{example}[Continued from Example~\ref{example:3-4-2}]\label{example:3-4-2-P-resolutions}
Let $(X,p)$ be a cyclic quotient surface singularity $\frac{1}{19}(1,7)$. There are three P-resolutions of $(X,p)$.
\begin{equation*}
3-4-[2], \quad 3-[4]-2, \quad [4]-1-[5,2].
\end{equation*}
\end{example}

It is possible to limit it to the case where the singularities of P-resolutions are only Wahl singularities.

\begin{definition}[M-resolution; cf.~{Behnke-Christophersen~\cite[p.882]{Behnke-Christophersen-1994}}]
\label{definition:M-resolution}
An \emph{M-resolution} of a rational surface singularity $(X,p)$ is a proper birational morphism $f \colon U' \to Y$ such that
\begin{enumerate}
\item $U'$ has only Wahl singularities.

\item $K_{U'}$ is nef relative to $f$, i.e., $K_{U'} \cdot E \ge 0$ for all $f$-exceptional curves $E$.
\end{enumerate}
\end{definition}

\begin{proposition}[{Behnke-Christophersen~\cite[3.1.4]{Behnke-Christophersen-1994}}]
For each P-resolution $U \to X$ there is an M-resolution $U' \to U$, such that there is a surjective morphism $g \colon U' \to U$ with $K_{U'} = g^{\ast}{K_U}$.
\end{proposition}

For a given P-resolution we call the corresponding $M$-resolution by the \emph{(crepant) M-resolution} of the P-resolution. We briefly recall how to construct the crepant M-resolution of a P-resolution. At first each T-singularity has a special crepant M-resolution. Let $Y=\frac{1}{dn^2}(1,dna-1)$ be a T-singularity with $d \ge 2$. The \emph{crepant} M-resolution $U' \to U$ of $X$ is defined by the following partial resolution of $U'$: $U'$ has $d-1$ exceptional components $C_i \cong \mathbb{CP}^1$ ($i=1,\dotsc,d-1$) and $d$ singular points $P_i$ of type $\frac{1}{n^2}(1, na-1)$ as described in the following figure:
\begin{center}
\includegraphics{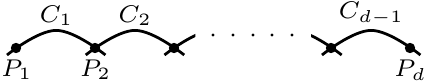}
\end{center}
Explicitly, the proper transforms of $C_i$'s in the minimal resolution $\widetilde{U}'$ of $U'$ are $(-1)$-curves and the minimal resolution $\widetilde{U}'$ is given by
\begin{equation*}
b_1-\dotsb-b_r-1-b_1-\dotsb-b_r-1-\dotsb-1-b_1-\dotsb-b_r
\end{equation*}
where $[b_1,\dotsc,b_r]=\frac{1}{n}(1,na-1)$. Then the M-resolution of a given P-resolution in the above proposition is obtained by taking the crepant M-resolutions for each T-singularities.
%
%
%

We finally introduce P-resolutions appearing in the next section.

\begin{definition}[Extremal P-resolution]
A P-resolution $f \colon Z^+ \to Y$ of a cyclic quotient surface singularity $(Y,0)$ is called an \emph{extremal P-resolution} if $f^{-1}(0)$ is a smooth rational curve $C$ and $Z^+$ has only Wahl singularities.
\end{definition}

Notice that an extremal P-resolution may have at most two Wahl singularities.

\section{Semistable minimal model program}\label{section:semistable-MMP}

We briefly recall the semistable minimal model program for one-parameter families of surfaces. Refer Kollár-Mori~\cite{KM-1992} and Hacking-Tevelev-Urzúa~\cite{HTU-2017}.

\subsection{Semistable extremal neighborhoods}

A three dimensional \emph{extremal neighborhood} is a birational morphism
\begin{equation*}
F \colon (C \subset \mathcal{Z}) \to (Q \in \mathcal{Y})
\end{equation*}
from a germ of a 3-fold $\mathcal{Z}$ along a proper reduced irreducible curve $C$ to a germ of a 3-fold $\mathcal{Y}$ along $Q \in \mathcal{Y}$ such that $F_{\ast}(\mathcal{O}_{\mathcal{Z}})=\mathcal{O}_{\mathcal{Y}}$, $F^{-1}(Q)=C$ (as sets), and $K_{\mathcal{Z}} \cdot C < 0$. Then $C \cong \mathbb{CP}^1$. Let $E_Z$ be a general member of $\lvert -K_{\mathcal{Z}} \rvert$ and let $E_Y = F(E_Z) \in \lvert - K_{\mathcal{Y}} \rvert$. The extremal neighborhood $F \colon (C \subset \mathcal{Z}) \to (Q \in \mathcal{Y})$ is called \emph{semistable} if $Q \in E_Y$ is a Du Val singularity of type $A$; Kollár-Mori~\cite[p.541]{KM-1992}. Semistable extremal neighborhoods are classified as two types as follows. A semistable extremal neighborhoof $F$ is of type \emph{k1A} or \emph{k2A} if the number of singular points of $E_Z$ equals one or two respectively; Kollár-Mori~\cite[p.542]{KM-1992}.

In this paper we are interested in the particular case that an semistable extremal neighborhood is given as total spaces of certain flat family of singular surfaces. Furthermore we would like to assume that the second Betti number $b_2(Z_t)$ of a general fiber of $\mathcal{Z} \to \Delta$ is equal to one:

\begin{definition}[{cf. HTU~\cite[Proposition~2.1]{HTU-2017}, Urzúa~\cite[Definition~2.5]{Urzua-2016-1}}]
Let $(Q \in Y)$ be a two-dimensional germ of a cyclic quotient surface singularity. Let $f \colon Z \to Y$ be a partial resolution of $Q \in Y$ such that $f^{-1}(Q)=C$ is a smooth rational curve with one (or two) Wahl singularity(ies) of $Z$ on it. Suppose that $K_Z \cdot C < 0$. Let $(Z \subset \mathcal{Z}) \to (0 \in \Delta)$ be a $\mathbb{Q}$-Gorenstein smoothing of $Z$ over a small disk $\Delta$. Let $(Y \subset \mathcal{Y}) \to \Delta$ be the corresponding blow-down deformation of $Y$. The induced birational morphism $(C \subset \mathcal{Z}) \to (Q \in \mathcal{Y})$ is called an \emph{extremal neighborhood of type mk1A (or mk2A)}.
\end{definition}

An extremal neighborhood is said to be \emph{flipping} if the exceptional set of $F$ is $C$. If it is not flipping, then the exceptional set of $F$ is of dimension $2$. In this case we call it a \emph{divisorial} extremal neighborhood.

For any flipping extremal neighborhood we always have a proper birational morphism called \emph{flip}:

\begin{proposition}[Kollár-Mori~{\cite[\S11 and Theorem~13.5]{KM-1992}}]
Suppose that $f \colon (C \subset \mathcal{Z}) \to (Q \in \mathcal{Y})$ is a flipping extremal neighborhood of type mk1A or mk2A. Let $f_0 \colon (C \subset Z) \to (Q \in Y)$ be the contraction of $C$ between the central fibers $Z$ and $Y$. Then there exists an extremal P-resolution $f^+ \colon (C^+ \subset Z^+) \to (Q \in Y)$ such that the flip $(C^+ \subset \mathcal{Z}^+) \to (Q \in \mathcal{Y})$ is obtained by the blown-down deformation of a $\mathbb{Q}$-Gorenstein smoothing of $Z^+$. That is, we have the commutative diagram
\begin{equation*}
\begin{tikzcd}
(C \subset \mathcal{Z}) \arrow[rr, dashed] \arrow[dr] \arrow[ddr]
&
& (C^+ \subset \mathcal{Z}^+) \arrow[dl] \arrow[ddl] \\
& (Q \in \mathcal{Y}) \arrow[d]
& \\
& (0 \in \Delta)
&
\end{tikzcd}
\end{equation*}
which is restricted to the central fibers as follows:
\begin{equation*}
\begin{tikzcd}
(C \subset Z) \arrow[rr, dashed] \arrow[dr]
&
& (C^+ \subset Z^+) \arrow[dl] \\
& (Q \in Y)
&
\end{tikzcd}
\end{equation*}
\end{proposition}

For a divisorial mk1A or mk2A, the birational morphism $F \colon \mathcal{Z} \to \mathcal{Y}$ is called a \emph{divisorial contraction} and $F$ is induced from blowing-downs between each smooth fibers. Explicitly:

\begin{proposition}[{see e.g. Urzúa~\cite[Proposition~2.8]{Urzua-2016-1}}]\label{proposition:divisorial-contraction}
If an mk1A or mk2A is divisorial, then $(Q \in Y)$ is a Wahl singularity. In addition, the divisorial contraction $F \colon \mathcal{Z} \to \mathcal{Y}$ induces the blowing-down of a $(-1)$-curve between the smooth fibers of $\mathcal{Z} \to \mathbb{D}$ and $\mathcal{Y} \to \mathbb{D}$.
\end{proposition}

\subsection{Numerical data for semistable extremal neighborhoods}

We will use the notation in Urzúa~\cite[Subsection 2.4]{Urzua-2016-1} for the extremal neighborhoods mk1A and mk2A.

\subsubsection{$Z \to Y$ for mk1A}

Let $\widetilde{Z}$ be the minimal resolution of the Wahl singularity in $Z$ with the exceptional curves $E_1, \dotsc, E_s$ such that $E_j^2 = -e_j$ ($j=1,\dotsc,s$), where the Wahl singularity is given by $\frac{m^2}{ma-1}=[e_1, \dotsc, e_s]$. Since $K_Z \cdot C < 0$ and $C \cdot C < 0$, the strict transform of $C$ in $\widetilde{Z}$ is a $(-1)$-curve intersecting only one component $E_i$ transversally at one point. We denote this data by
\begin{equation*}
[e_1, \dotsc, \overline{e_i}, \dotsc, e_s].
\end{equation*}
If  $Q \in Y$ is a quotient surface singularity of type $\frac{1}{\Omega}(1, \Delta)$, then
\begin{equation*}
\frac{\Delta}{\Omega} =[e_1, \dotsc, e_{i-1}, e_i-1, e_{i+1}, \dotsc, e_s].
\end{equation*}

We can calculate $\Delta$ and $\Omega$ by the sequences on integers defined recursively from the continued fraction $[e_1, \dotsc, e_s]$. At first, we defines the sequence $\{\beta_j\}$ with
\begin{equation*}
\beta_{s+1}(=0) < \beta_s(=1) < \dotsb < \beta_1(=a) < \beta_0(=m)
\end{equation*}
by $\beta_{j+1}=e_j \beta_j-\beta_{j-1}$. In this way we have
\begin{equation*}
\frac{\beta_{j-1}}{\beta_j} = [e_j,\dotsc,e_s].
\end{equation*}
On the other hand, we define similarly sequences of integers $\{\alpha_j\}$ and $\{\gamma_j\}$ by $\alpha_{j+1}=e_j\alpha_j - \alpha_{j-1}$ and $\gamma_{j+1}=e_j\gamma_j - \gamma_{j-1}$ starting with $\alpha_0=0$, $\alpha_1=1$ and $\gamma_0=-1$, $\gamma_1=0$. Then we have
\begin{equation*}
\alpha_0(=0) < \alpha_1(=1) < \dotsb < \alpha_s=(a^{-1}) < \alpha_{s+1}(=m)
\end{equation*}
where $a^{-1}$ is the integer such that $0 < a^{-1} < m$ and $aa^{-1}\equiv 1 \pmod{m}$. In particular,
\begin{equation*}
\frac{\alpha_j}{\gamma_j} = [e_1,\dotsc,e_j].
\end{equation*}
Then
\begin{equation*}
\text{$\Delta = m^2 - \beta_i\alpha_i$ and $\Omega = ma-1-\gamma_i\beta_i$}.
\end{equation*}
Furthermore, if we define
\begin{equation*}
\delta := \frac{\beta_i+\alpha_i}{m}
\end{equation*}
then we have $K_X \cdot C = -\delta/m < 0$ and $C \cdot C = -\Delta/m^2 < 0$.

\subsubsection{$Z \to Y$ for mk2A}

We have two Wahl singularities $m_1^2/(m_1a_1-1)$ and $m_2^2/(m_2a_2-1)$ on $C \subset Z$. Let $m_1^2/(m_1a_1-1)=[e_1, \dotsc, e_{s_1}]$ and $m_2^2/(m_2a_2-1)=[f_1, \dotsc, f_{s_2}]$ and let $E_1, \dotsc, E_{s_1}$ and $F_1, \dotsc, F_{s_2}$ be the corresponding exceptional divisors with $E_i^2=-e_i$ and $F_j^2=-f_j$ for all $i$ and $j$. As before, the strict transform of $C$ in the minimal resolution $\widetilde{X}$ of $X$ is a $(-1)$-curve intersecting only one $E_i$ and $F_j$ at one point. So these two exceptional curves should be the ends of each exceptional chains because the minimal resolution of $Q \in Y$ is a linear chain of $\mathbb{CP}^1$'s. We assume that the $(-1)$-curve intersects $E_1$ and $F_1$. We denote the data for an mk2A by
\begin{equation*}
[(m_2,a_2)]-1-[(m_1,a_1)] = [f_{s_2}, \dotsc, f_1]-1-[e_1, \dotsc, e_{s_1}].
\end{equation*}
Then $(Q \in Y)$ is given by
\begin{equation*}
\frac{\Delta}{\Omega} = [f_{s_2}, \dotsc, f_1, 1, e_1, \dotsc, e_{s_1}].
\end{equation*}

Furthermore, if we define
\begin{equation*}
\delta := m_1a_2 + m_2a_1 - m_1m_2,
\end{equation*}
then
\begin{equation*}
\Delta = m_1^2 + m_2^2 - \delta m_1 m_2, \quad \Omega = (m_2-\delta m_1)(m_2-a_2) + m_1a_1 -1
\end{equation*}
and we have $K_X \cdot C = - \delta/(m_1m_2) < 0$ and $C \cdot C = - \Delta/(m_1^2m_2^2) < 0$.

\subsubsection{$Z^+ \to Y$ for an extremal P-resolution}

Finally, the numerical description of extremal P-resolutions are given as follows.

Let $m_1'^2/(m_1'a_1'-1) = [e_1, \dotsc, e_{s_1}]$ and $m_2'^2/(m_2'a_2'-1)=[f_1, \dotsc, f_{s_2}]$ be two Wahl singularities.  Here we allow that $m_i'=a_i'=1$; that is, we allow that a Wahl singularity (or both) is (are) actually smooth. Similar to an mk2A, we denote an extremal P-resolution by
\begin{equation*}
[f_{s_2}, \dotsc, f_1] - c - [e_1, \dotsc, e_{s_1}]
\end{equation*}
where $-c$ is the self-intersection number of the strict transform of $C^+$ in the minimal resolution of $Z^+$. As before, $(Q \in Y)$ is defined by
\begin{equation*}
\frac{\Delta}{\Omega} = [f_{s_2}, \dotsc, f_1,c,e_1, \dotsc, e_{s_1}].
\end{equation*}

Define
\begin{equation}\label{equation:X+-delta}
\delta := cm_1'm_2'-m_1'a_2'-m_2'a_1'.
\end{equation}
Then
\begin{equation*}
\Delta = m_1'^2 + m_2'^2 + \delta m_1'm_2'
\end{equation*}
and, when both $m_i' \neq 1$,
\begin{equation*}
\Omega = -m_1'^2(c-1) + (m_2' + \delta m_1')(m_2' - a_2') + m_1'a_1'-1.
\end{equation*}
Finally, we have $K_{Z^+} \cdot C^+ = \delta/(m_1'm_2') > 0$ and $C^+ \cdot C^+ = - \Delta/(m_1'^2m_2'^2) < 0$.

\subsection{Determining the types of semistable extremal neighborhoods}

From the numerical data of a semistable extremal neighborhood of type mk1A or mk2A, one can determine whether it is flipping or divisorial and one can also calculate the flip if it is a flipping type. We summarize briefly HTU~\cite[Subsection~3.3]{HTU-2017}. See also Urzúa~\cite[Section~2]{Urzua-2016-1}.

Let $\mathcal{Z}_1 = [(m_2,a_2)]-1-[(m_1,a_1)]$ be an mk2A with $m_2 > m_1$. Here we also allow the mk1A for special case $m_1=a_1=1$. We have $\delta=m_2a_1+m_1a_2-m_1m_2$ as before. Since $K_Z \cdot C = -\delta/(m_1m_2) < 0$, we have $\delta \ge 1$.

Assume that
\begin{equation*}
\delta m_1 - m_2 \le 0
\end{equation*}
from now on. Such an mk2A is called an \emph{initial} mk2A in HTU~\cite{HTU-2017}. We first define two sequences $d(i)$ and $c(i)$ (called the \emph{Mori recursions}) as follows: For $i \ge 2$,
\begin{align*}
&d(1)=m_1, \quad d(2)=m_2, \quad d(i+1)=\delta d(i)-d(i-1), \\
&c(1)=a_1, \quad c(2)=m_2-a_2, \quad c(i+1)=\delta c(i)-c(i-1).
\end{align*}

\begin{definition}[Mori sequence for $\delta > 1$]
If $\delta > 1$, then a \emph{Mori sequence} for an initial mk2A $\mathcal{Z}_1$ is a sequence of mk2A's $\mathcal{Z}_i$ with Wahl singularities defined by the pairs $(m_i, a_i)$ and $(m_{i+1}, a_{i+1})$ where
\begin{equation*}
m_{i+1}=d(i+1), a_{i+1}=d(i+1)-c(i+1), \quad m_i=d(i), a_i=c(i).
\end{equation*}
Notice that $m_{i+1} > m_i$.
\end{definition}

\begin{definition}[Mori sequence for $\delta=1$]
If $\delta=1$, then a \emph{Mori sequence} for an initial mk2A $\mathcal{Z}_1$ consists of just one more mk2A $\mathcal{Z}_2$ defined by the pairs $(m_2, a_2)$ and $(m_3, a_3)$, where $m_2=d(2)$, $a_2=c(2)$, and $m_3=d(2)-d(1)$, $a_3=d(2)-d(1)+c(1)-c(2)$.
\end{definition}

\begin{proposition}[{HTU~\cite[Proposition~3.15, Theorem~3.20]{HTU-2017}}]\label{proposition:flip-for-Mori-sequence}
If $\delta m_1 - m_2 < 0$, then the mk2A $\mathcal{Z}_i$'s ($i \ge 1$) in the Mori sequence of an initial mk2A $\mathcal{Z}_1$ are of flipping type sharing the same $\delta$, $\Omega$, $\Delta$ associated to $\mathcal{Z}_1$ and, after flipping each $\mathcal{Z}_i$, they have the same extremal P-resolution $Z^+$ admitting two Wahl singularities defined by the pairs $(m_1', a_1')$ and $(m_2', a_2')$ where $m_2'=m_1$, $a_2'=m_1-a_1$, and $m_1'=m_2-\delta m_1$, $a_1' \equiv m_2-a_2 -\delta a_1 \pmod{m_1'}$. Here, in case of $m_1=a_1=1$, we set $a_2'=1$.
\end{proposition}

\begin{proposition}[{HTU~\cite[Proposition~3.13]{HTU-2017}}]
If $\delta m_1 - m_2=0$, then $\mathcal{Z}_i$'s are of divisorial type with $m_1=\delta$, $m_2=\delta^2$, and $a_2=\delta^2-(\delta a_1 -1)$. By Proposition~\ref{proposition:divisorial-contraction}, the contraction $\mathcal{Z}_i \to \mathcal{Y}$ over $\mathbb{D}$ corresponding to $\mathcal{Z}_i$ induces the blow down $Z_{i,t} \to Y_t$ over $0 \neq t \in \mathbb{D}$ of a $(-1)$-curve between smooth fibers.
\end{proposition}

HTU~\cite{HTU-2017} also provides a way to compute for all extremal neighborhoods of type mk1A. Refer Urzúa~\cite[Proposition~2.12]{Urzua-2016-1} for a concise summary.

\begin{proposition}[{HTU~\cite[\S2.3, \S3.4]{HTU-2017}}]\label{proposition:mk1A->mk2A}
Let $[e_1,\dotsc,\overline{e_i},\dotsc,e_s]$ be an mk1A with $\frac{m^2}{ma-1}=[e_1,\dotsc,e_s]$. Let $\frac{m_2}{m_2-a_2}=[e_1,\dotsc,e_{i-1}]$ and $\frac{m_1}{m_1-a_1}=[e_{i+1},\dotsc,e_s]$ (if possible; that is, for the first $i > 1$ and for the second $i < s$). Then there are two mk2A's given by
\begin{equation*}
\text{$[f_{s_2},\dotsc,f_1]-1-[e_1,\dotsc,e_s]$ and $[e_1,\dotsc,e_s]-1-[g_1,\dotsc,g_{s_1}]$}
\end{equation*}
where $\frac{m_2^2}{m_2a_2-1}=[f_1,\dotsc,f_{s_2}]$ and $\frac{m_1^2}{m_1a_1-1}=[g_1,\dotsc,g_{s_1}]$ such that the corresponding cyclic quotient surface singularity $\frac{1}{\Omega}(1,\Delta)$ and $\delta$ are the same for the mk1A and the mk2A's. Moreover these three extremal neighborhoods are either (1) flipping with the same extremal P-resolution for the flip or (2) divisorial with the same $(Q \in Y)$.
\end{proposition}

The above propositions provide us the way to determine the type of the given extremal neighborhood: Find the initial extremal neighborhood from the given one by applying the Mori recursion in the opposite way and determine its type by calculating the value $\delta m_1 - m_2$ of the initial one.

\begin{example}[A flipping mk2A with $\delta > 1$]
Let $(Q \in Y)=\frac{1}{11}(1,3)$. Let $\mathcal{Z}$ be an mk2A $[(50,37)]-1-[(19,5)]$. Since $\delta=3$, the initial mk2A is $\mathcal{Z}_1=[(7,5)]-1-[(2,1)]$. We have $\delta m_1-m_2 < 0$. So the initial mk2A $\mathcal{Z}_1$ (and hence $\mathcal{Z}_3$ also) is of flipping type. The flip have $Z^+=[4]-3$ as the central fiber. Notice that the mk1A $[4,7,2,\overline{2},3,2,2]$ is also of flipping type with $Z^+=[4]-3$ as its flip.
\end{example}

\begin{example}[A flipping mk2A with $\delta=1$]
Let $(Q \in Y) = \frac{1}{13}(1,3)$ . Let $\mathcal{Z}$ be an mk2A $[2,5]-1-[2,2,6]$. Then $\delta=1$ and the initial mk2A $\mathcal{Z}_1$ is $[2,2,6]-1-\varnothing$. So it is of flipping type and the flip is $[5,2]-2$.
\end{example}

\begin{example}[A divisorial mk2A]
Let $(Q \in \mathcal{Y})$ be the Wahl singularity $\frac{1}{4}(1,1)$. Then the mk2A $[8,2,2,2,2]-1-[6,2,2]$ is of divisorial type.
\end{example}

\subsection{Degenerations}

We briefly recall how the degenerations occur during flips only for the situations that we will encounter often. For details and full generalities, refer Urzúa~\cite[\S4]{Urzua-2016-2}. Let $(E^{-1} \subset Z_0 \subset \mathcal{Z}) \to (Q \in T \subset \mathcal{Y})$ be a flipping $mk1A$. Let $\Gamma_0$ be an irreducible curve in $S_0$ such that $\Gamma_0 \cdot E^{-1} = 1$ but $\Gamma_0$ does not pass through any of s the singular points of $S_0$. Let $\Gamma_t$ be the deformation of $\Gamma_0$; that is, there is a divisor $\Gamma$ in $\mathcal{Z}$ such that $\Gamma|_{Z_t} = \Gamma_t$. Then $\Gamma_t$ is isomorphic to $\Gamma_0$. Let $(E^+ \subset Z^+_0 \subset \mathcal{Z}^+) \to (Q \in T \subset \mathcal{Y})$ be the flipping of $(E^- \subset Z_0 \subset \mathcal{Z}) \to (Q \in T \subset \mathcal{Y})$ along $E^-$. Let $\Gamma^+ \subset \mathcal{Z}^+$ be the proper transform of $\Gamma$. Then $\Gamma^+_t = \Gamma_t$ for $t \neq 0$. But:

\begin{proposition}[{Urzúa~\cite[\S4]{Urzua-2016-2}}]\label{proposition:degeneration}
We have
\begin{equation*}
\Gamma^+_0 = \Gamma_0' + \beta E^+
\end{equation*}
for some positive integer $\beta$, where $\Gamma_0'$ is the proper transform of $\Gamma_0$.
\end{proposition}

Since $\mathcal{Z}^+$ is $\mathbb{Q}$-factorial, we have the equality
\begin{equation}\label{equation:degeneration}
(\Gamma_0' + \beta E^+) \cdot K_{Z^+_0} = \Gamma^+_t \cdot K_{Z^+_t} = \Gamma_t \cdot K_{Z_t} = \Gamma_0 \cdot K_{Z_0}
\end{equation}
So one can calculate $\beta$ from the above equation.

\subsection{Flips encountered frequently}

The flipping mk1A, which appears frequently, is as follows.

\begin{proposition}[Usual flips; {Urzúa~\cite[Proposition~2.15]{Urzua-2016-1}}]\label{proposition:Usual-flips}
The mk1A
\begin{equation*}
\mathcal{Z}=[a_1,\dotsc,a_{s-1},\overline{a_s}]
\end{equation*}
is of flipping type. Suppose that $a_i \ge 3$ but $a_{i+1}=\dotsb=a_s=2$ for some $i$. Here if $a_s \ge 3$ then we set $i=s$. Then the image of $A_1$ in the extremal P-resolution $Z^+$ is the curve $C^+$ and there is a Wahl singularity $[a_2,\dotsc,a_i-1]$ on $C^+$ if $i \ge 2$. That is, the numerical data for $\mathcal{Z}^+$ is either $a_1-[a_2,\dotsc,a_i-1]$ for $i \ge 2$ or $a_1-1$ for $i=1$.
\end{proposition}

In this situation, the degeneration occurs in the following manner.

\begin{corollary}[{Urzúa~\cite[Proposition~4.1]{Urzua-2016-2}}]\label{Corollary:usual-degeneration}
Let $[e_1,\dotsc,e_{s-1},\overline{e_s}]$ be an mk1A. Then $\beta=1$; that is, we have $\Gamma^+_0 = \Gamma_0' + E_1$.
\end{corollary}

\section{From P-resolutions to picture deformations}\label{section:P-resolution->Picture-deformations}

Let $(X,p)$ be a sandwiched surface singularity and let $\mathcal{X} \to \Delta$ be a one-parameter smoothing of $(X,p)$ over a small open disk $D \subset \mathbb{C}^2$ centered at the origin. Assume that there is a M-resolution $U \to X$ that induces the smoothing $\mathcal{X}$. Then there is a M-resolution $Z \to Y$ of the compatible compactification $Y$ of $X$ such that its $\mathbb{Q}$-Gorenstein smoothing $\mathcal{Z} \to \Delta$ blows down to the smoothing $\mathcal{Y} \to \Delta$ of $Y$ that is an extension of $\mathcal{X} \to \Delta$.

In this section we will show that one can identify the picture deformation corresponding to the given smoothing $\mathcal{X} \to \Delta$ via running the semistable minimal model program to the  $\mathbb{Q}$-Gorenstein smoothing $\mathcal{Z} \to \Delta$.

\subsection{Reduction to smooth central fibers}

We first reduce the smoothing $\mathcal{Z} \to \Delta$ of the M-resolution $Z$ of $Y$ to a deformation with a smooth central fiber.

Urzúa~\cite{Urzua-2016-2} defines a so-called \emph{$W$-surface} as a normal projective surface $S$ together with a proper deformation $\mathcal{S} \to \Delta$ such that (1) $S$ has at most singularities of class $T_0$, (2) $\mathcal{S}$ is a normal complex 3-fold such that the canonical divisor $K_{\mathcal{S}}$ is $\mathbb{Q}$-Cartier, (3) The fiber $S_0$ is reduced and isomorphic to $S$, (4) The fiber $S_t$ is nonsingular for $t \neq 0$. Then:

\begin{proposition}[{Urzúa~\cite[Corollary~3.5]{Urzua-2016-2}}]\label{proposition:Urzua-Corollary-3.5}
If $S_0$ is birational to $S_t$ for $t \neq 0$, then the smoothing $\mathcal{S} \to \Delta$ can be reduced to a deformation $\mathcal{S}' \to \Delta$ whose central fiber $S_0'$ is smooth by applying a finite number of the divisorial contractions and the flips described in Section~\ref{section:semistable-MMP}.
\end{proposition}

Returning to our situation, notice that $Z=Z_0$ is a $W$-surface. Furthermore the central fiber $Z_0$ and a general fiber $Z_t$ are both rational surfaces because they contain a $(+1)$-curve that was the line at infinity in $\mathbb{CP}^2$. Therefore it follows by the above Proposition~\ref{proposition:Urzua-Corollary-3.5} that:

\begin{proposition}\label{proposition:smooth-central-fiber}
By applying the divisorial contractions and the flips described in Section~\ref{section:semistable-MMP} to $(-1)$-curves on the central fiber $Z_0$ of $\mathcal{Z} \to \Delta$, one can run MMP to $\mathcal{Z} \to \Delta$ until one obtains a deformation $\mathcal{Z}' \to \Delta$ whose central fiber $Z_0'$ is smooth.
\end{proposition}

Flips never change a general fibers. But a divisorial contraction is just a blowing down on each fibers. So:

\begin{corollary}
In Proposition~\ref{proposition:smooth-central-fiber}, a general fiber $Z_t$ is obtained by blowing up several times a general fiber $Z'_t$ of the smoothing $\mathcal{Z}' \to \Delta$ of $Z_0'$.
\end{corollary}

\begin{theorem}\label{theorem:From-P-resol-To-Pic-def}
One can run the semi-stable MMP to $\mathcal{Z} \to \Delta$ until one obtains the corresponding picture deformation $\mathcal{D} \to \Delta$ of the compactified decorated curve $(D,l)$.
\end{theorem}

\begin{proof}
Notice that the minimal resolution of the central fiber $Z_0'$ in the above Proposition~\ref{proposition:smooth-central-fiber} is obtained by blowing up $\mathbb{CP}^2$ at $[0,0,1]$. So after a sequence of divisorial contractions applied to $\mathcal{Z}' \to \Delta$, we have a deformation $\mathcal{Z}'' \to \Delta$ whose central fiber $Z_0''$ is just $\mathbb{CP}^2$. Similarly, its general fiber $Z_t''$ of $\mathcal{Z}'' \to \Delta$ is also $\mathbb{CP}^2$. Therefore the deformation $\mathcal{Z}'' \to \Delta$ is just a deformation of $\mathbb{CP}^2$ to itself. But notice that the maps $Z_0 \to Z_0''(=\mathbb{CP}^2)$ is decomposed into $Z_0 \to Y_0 \to \mathbb{CP}^2$. And the map $Y_0 \to \mathbb{CP}^2$ is just the contraction of all decorated $(-1)$-curves on $Z_0$. On the other hand, the map $Z_t \to Z_t''(=\mathbb{CP}^2)$ for $t \neq 0$ is just a composition of blowing downs. Therefore the image of the proper transform $(\widetilde{D},l)$ in $\mathcal{Z}'' \to \Delta$ is the compactified picture deformation of the compactified decorated curve $(D,l)$ of $Y$.
\end{proof}

\subsection{Identifying method}

A flip changes only the central fiber. So it is enough to pay attention to only divisorial contractions in order to track down how a general fiber is changed during the MMP process above. But divisorial contractions are just blow-downs of $(-1)$-curves on a general fiber. Since $\mathcal{Z}' \to \Delta$ is a deformation of a smooth central fiber $Z_0'$, a general fiber $Z_t'$ is diffeomorphic to $Z_0'$. Especially flips induce degenerations. That is, an irreducible curve in a general fiber $Z_t'$ may be degenerated into several curves in the central fiber $Z'_0$. By comparing $Z_0'$ and $Z_t'$ via degeneration, one can get the data of positions of $(-1)$-curves in $Z_t'$. Then by tracking the blow-downs $Z_t \to Z_t'$ one can get the picture deformation of $X$ corresponding to the smoothing $\mathcal{X} \to \Delta$.

\begin{theorem}\label{theorem:Phi_PI=Phi_IoPhi_P}
There is a map
\begin{equation*}
\phi_{PH} \colon \mathcal{P}(X) \to \mathcal{I}(X)
\end{equation*}
defined by the MMP such that $\phi_{PI} = \phi_I \circ \phi_P$.
\end{theorem}

\section{Illustrations of the correspondences}\label{section:illustrations}

We illustrate how to determine the picture deformations corresponding to given normal P-modifications using several examples.

\subsection{An introductory example}\label{section:[4]}

Let $(X,p)$ be a cyclic quotient surface singularity $\frac{1}{4}(1,1)$. It has a sandwiched structure given by
\begin{equation*}
\begin{tikzpicture}[scale=0.75]
\node[bullet] (00) at (0,0) [labelAbove={$-4$}] {};

\node[bullet] (-20) at (-2,0) [label=left:{$\widetilde{C}_1$}] {};
\node[bullet] (-10) at (-1,0) [labelAbove={$-1$}] {};

\node[bullet] (10) at (1,0) [labelAbove={$-1$}] {};
\node[bullet] (20) at (2,0) [label=right:{$\widetilde{C}_2$}] {};

\node[bullet] (0-1) at (0,-1) [label=left:{$-1$}] {};
\node[bullet] (0-2) at (0,-2) [label=below:{$\widetilde{C}_3$}] {};

\draw[-] (-20)--(-10)--(00);
\draw[-] (00)--(10)--(20);
\draw[-] (00)--(0-1)--(0-2);
\end{tikzpicture}
\end{equation*}

The decorated curve $(C,l)$ of $(X,p)$ consists of three different lines $C_1, C_2, C_3$ passing through the origin of $\mathbb{C}^2$ with the decorations $l_1=l_2=l_3=2$. There are two picture deformations of $(C,l)$:

\begin{center}
\includegraphics[scale=1.25]{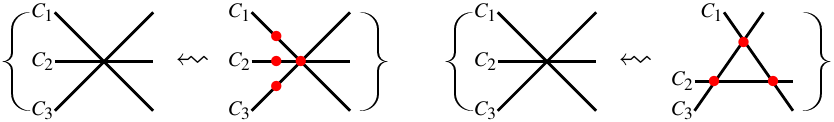} 
\end{center}

On the other hand, we have also two M-resolutions: The minimal resolution and the singularity itself. So we have two compactified M-resolutions

\begin{equation*}
\begin{tikzpicture}[scale=0.75]
\node[bullet] (00) at (0,0) [labelAbove={$-4$}] {};

\node[bullet] (-20) at (-2,0) [label=left:{$\widetilde{C}_1$}] {};
\node[bullet] (-10) at (-1,0) [labelAbove={$-1$}] {};

\node[bullet] (10) at (1,0) [labelAbove={$-1$}] {};
\node[bullet] (20) at (2,0) [label=right:{$\widetilde{C}_2$}] {};

\node[bullet] (0-1) at (0,-1) [label=left:{$-1$}] {};
\node[bullet] (0-2) at (0,-2) [label=below:{$\widetilde{C}_3$}] {};

\draw[-] (-20)--(-10)--(00);
\draw[-] (00)--(10)--(20);
\draw[-] (00)--(0-1)--(0-2);
\end{tikzpicture} \qquad
\begin{tikzpicture}[scale=0.75]
\node[rectangle] (00) at (0,0) [labelAbove={$-4$}] {};

\node[bullet] (-20) at (-2,0) [label=left:{$\widetilde{C}_1$}] {};
\node[bullet] (-10) at (-1,0) [labelAbove={$-1$}] {};

\node[bullet] (10) at (1,0) [labelAbove={$-1$}] {};
\node[bullet] (20) at (2,0) [label=right:{$\widetilde{C}_2$}] {};

\node[bullet] (0-1) at (0,-1) [label=left:{$-1$}] {};
\node[bullet] (0-2) at (0,-2) [label=below:{$\widetilde{C}_3$}] {};

\draw[-] (-20)--(-10)--(00);
\draw[-] (00)--(10)--(20);
\draw[-] (00)--(0-1)--(0-2);
\end{tikzpicture}
\end{equation*}

We show the procedure for identifying picture deformations from given M-resolutions. Let's start with the minimal resolution. Let $(Y,p)$ be the compatible compactification of $(X,p)$ and let $Z$ be the M-resolution of $Y$ corresponding to the M-resolution of $X$ that is the minimal resolution of $X$. Let $\mathcal{Z} \to \Delta$ be the corresponding $\mathbb{Q}$-Gorenstein smoothing of $Z$. So we have the following picture.

\begin{center}
\includegraphics[scale=1.25]{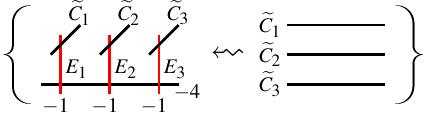} 
\end{center}

The $(-1)$-curve $E_i$ ($i=1,2,3$) in the central fiber deforms to a $(-1)$-curve in a general fiber that intersects only with $\widetilde{C}_i$, respectively. Therefore we can add $(-1)$-curves to the above picture.

\begin{center}
\includegraphics[scale=1.25]{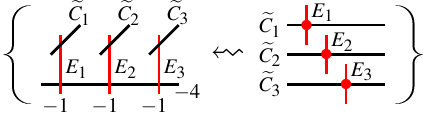} 
\end{center}

We apply the divisorial contraction to $E_1, E_2, E_3$. Then we have a new deformation, where we denote by $P_i$'s the points to which $E_i$'s are contracted.

\begin{center}
\includegraphics[scale=1.25]{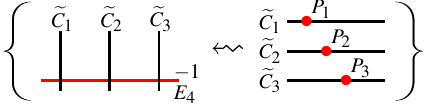} 
\end{center}

There is a new $(-1)$-curve $E_4$, which deforms to a $(-1)$-curve in a central fiber that intersects with $\widetilde{C}_1, \widetilde{C}_2, \widetilde{C}_3$. So we have one more $(-1)$-curve $E_4$ in a general fiber:

\begin{center}
\includegraphics[scale=1.25]{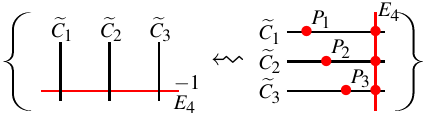} 
\end{center}

We finally apply the divisorial contraction to $E_4$. Then we obtain the first picture deformation

\begin{center}
\includegraphics[scale=1.25]{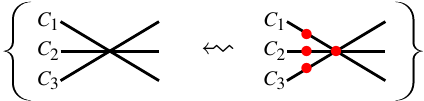} 
\end{center}

Now let's consider the second M-resolution, where three $(-1)$-curves pass through the singularity $[4]$:

\begin{center}
\includegraphics[scale=1.25]{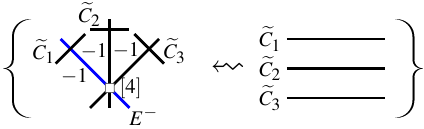} 
\end{center}

The curve $E^{-}$ is a flipping curve. So we apply the usual flip to $E^{-}$. So we have a new flipped deformation

\begin{center}
\includegraphics[scale=1.25]{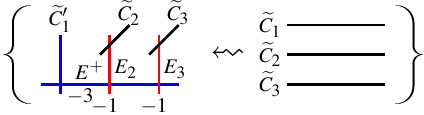} 
\end{center}

After the flip, the decorated curve $\widetilde{C}_1$ in a general fiber is degenerated to the sum of $\widetilde{C}_1'$ and the $(-3)$-curve $E^+$ in the central fiber. That is,
\begin{equation*}
\widetilde{C}_1' + E^+ \leftsquigarrow \widetilde{C}_1
\end{equation*}

Notice that $E_i \cdot (\widetilde{C}_1' + E^+)=1$ for $i=1,2$ in the central fiber. So $E_1$ and $E_2$ deform to $(-1)$-curves in a general fiber which intersect with $C_1, C_2$ and $C_1, C_3$, respectively:

\begin{center}
\includegraphics[scale=1.25]{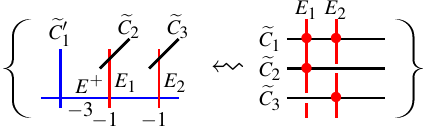} 
\end{center}

We apply the divisorial contractions to $E_1$ and $E_2$.

\begin{center}
\includegraphics[scale=1.25]{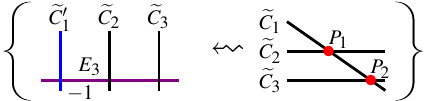} 
\end{center}

We then have a new $(-1)$-curve $E_3$, whose deformation in a general fiber intersects with $C_2$ and $C_3$.

\begin{center}
\includegraphics[scale=1.25]{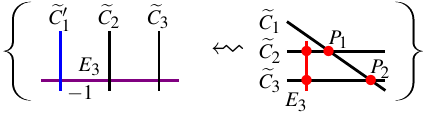} 
\end{center}

Finally, applying the divisorial contraction to $E_3$, we obtain the second picture deformation

\begin{center}
\includegraphics[scale=1.25]{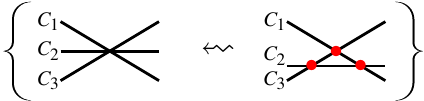} 
\end{center}

\subsection{A cyclic quotient surface singularity}\label{section:example-cyclic}

Let $(X,p)$ be a cyclic quotient surface singularity $\frac{1}{19}(1,7)$. Suppose that a decorated curve of $(X,p)$ is given as in Example~\ref{example:3-4-2-decorated-curve}.
\begin{center}
\includegraphics[scale=1.25]{figure-3-4-2-decorated-curve}
\end{center}
It has three picture deformations; Example~\ref{example:3-4-2-picture-deformation}.
\begin{center}
\includegraphics[scale=1.25]{figure-3-4-2-picture-deformation}
\end{center}
On the other hand, there are also three M-resolutions of $(X,p)$: $U_1=3-4-2$, $U_2=3-[4]-2$, $U_3=[4]-1-[5,2]$. We explain the procedure for identifying the picture deformation corresponding to $U_3$ step by step below. For $U_1$ and $U_2$ we present the procedures in Figure~\ref{figure:PtoI-U1-U2}.

\begin{figure}
\centering
\subfloat[$U_1=3-4-2$]{\includegraphics[scale=1.25]{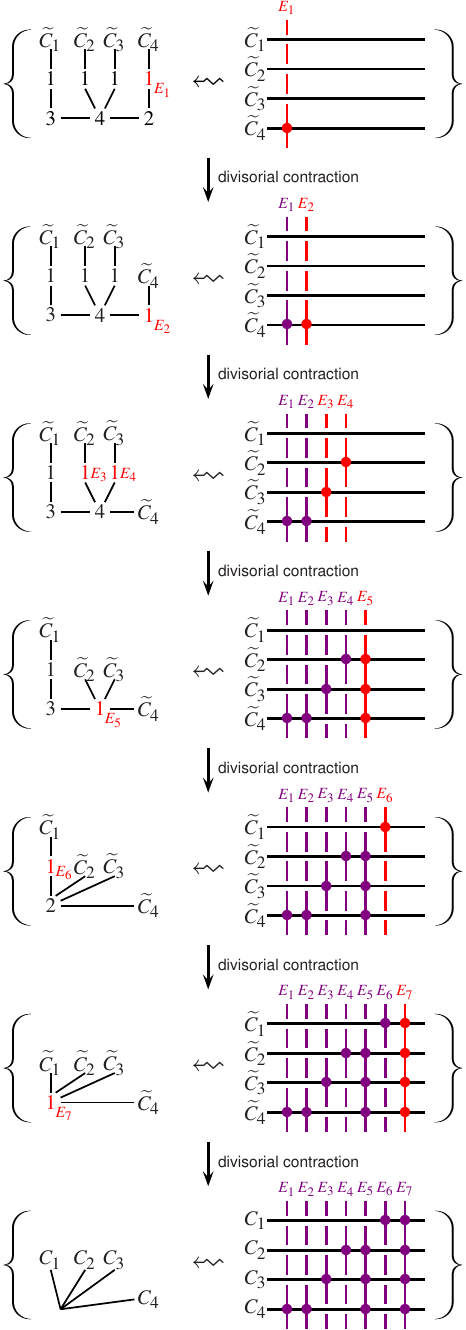}} 
\qquad
\subfloat[{$U_2=3-[4]-2$}]{\includegraphics[scale=1.25]{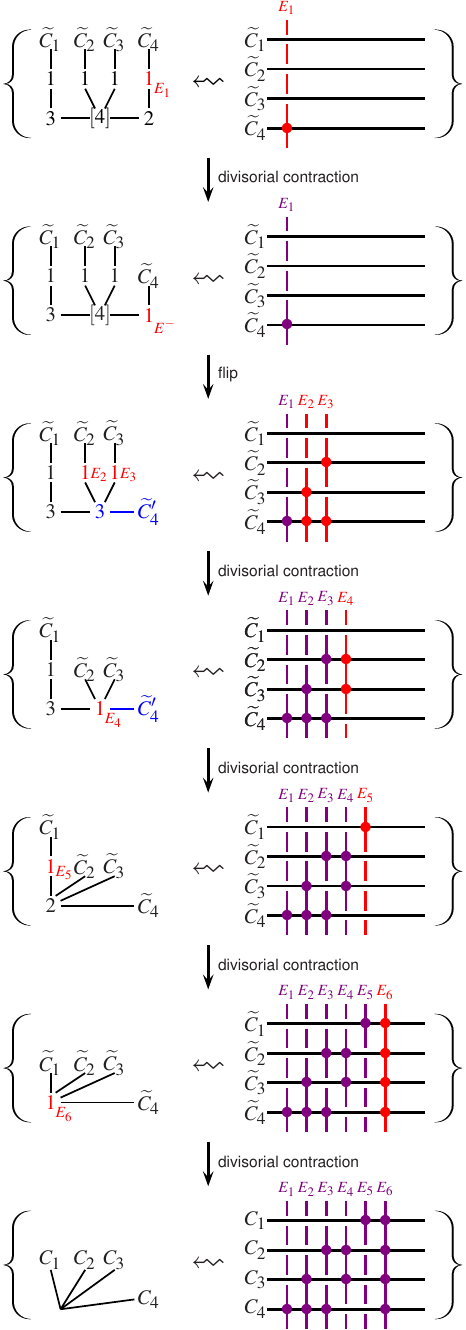}} 
\caption{Example in Section~\ref{section:example-cyclic}}
\label{figure:PtoI-U1-U2}
\end{figure}

Let $Z_3$ be the M-resolution of $Y$ corresponding to $U_3$ and let $\mathcal{Z}_3 \to \Delta$ be the smoothing of $Z_3$ induced from $\mathbb{Q}$-Gorenstein smoothings of T-singularities of $Z_3$. Since the compactified decorated curve $\widetilde{C}_1, \dotsc, \widetilde{C}_4$ do not disappear during the deformation, the central fiber and a general fiber of the deformation $\mathcal{Z}_3 \to \Delta$ look as follows:
\begin{center}
\includegraphics[scale=1.25]{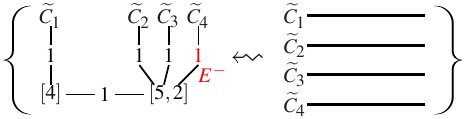} 
\end{center}

We first apply a usual flip along the $(-1)$-curve $E^-$ intersecting $\widetilde{C}_4$. Then we have a new deformation
\begin{center}
\includegraphics[scale=1.25]{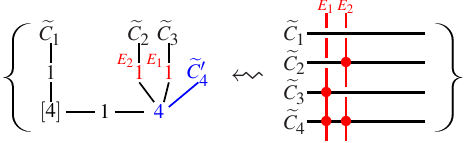} 
\end{center}
Here the curve $\widetilde{C}_4$ in a general fiber is degenerated to the curve $4+\widetilde{C}_4'$ in the central fiber by Corollary~\ref{Corollary:usual-degeneration}. We tie them together with a blue curve.

There are two $(-1)$-curves $E_1$ and $E_2$ in the central fiber. They do not disappear during the deformation. So there are also two $(-1)$-curves, say again $E_1$ and $E_2$ in a general fiber. In the central fiber, the $(-1)$-curve $E_1$ intersects transversally once with $\widetilde{C}_3$ and also with the degeneration $4+\widetilde{C}_4'$ of $\widetilde{C}_4$. So $E_1$ intersects with $\widetilde{C}_3$ and $\widetilde{C}_4$ in a genera fiber in the same way. Similarly, $E_2$ intersects with $\widetilde{C}_2$ and $\widetilde{C}_4$ in a genera fiber. In short, we can track how $(-1)$-curves in a general fiber intersect with the decorated curves by looking at how they intersects with the (degenerations of) decorated curves in the central fiber.

We next apply divisorial contractions along $E_1$ and $E_2$:
\begin{center}
\includegraphics[scale=1.25]{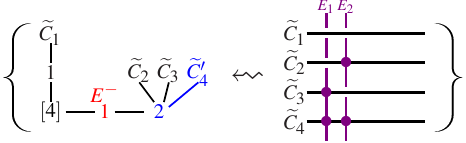} 
\end{center}
Here we painted in violet the contracted curves in a general fiber.

Applying the usual flip along the $(-1)$-curve $E^-$ passing through the singular point $[4]$, we get the deformation
\begin{center}
\includegraphics[scale=1.25]{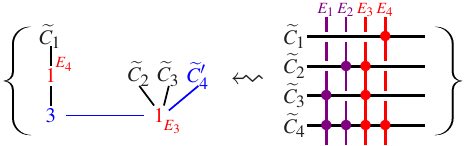} 
\end{center}
The decorated curve $\widetilde{C}_4$ degenerates to $3-1-\widetilde{C}_4'$ in the central fiber. Two $(-1)$-curves $E_3$ and $E_4$ appear after the flip. Note that $E_3 \cdot \widetilde{C}_i=1$ for $i=1,2,3$ and $E_4 \cdot \widetilde{C}_j=1$ for $j=1,4$. So we have two $(-1)$-curves $E_3$ and $E_4$ in a general fiber that intersect with the decorated curves in the same way.

We apply divisorial contractions along $E_3$ and $E_4$:
\begin{center}
\includegraphics[scale=1.25]{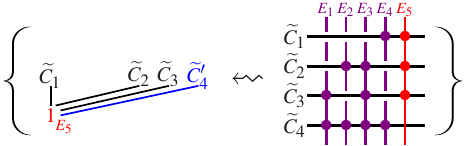} 
\end{center}
Then a new $(-1)$-curve $E_5$ appears so that $E_5 \cdot \widetilde{C}_i=1$ for $i=2,3,4$ in a general fiber.

Finally, we apply a divisorial contraction along $E_5$. Then we get the picture deformation corresponding to the P-resolution $U_3$:
\begin{center}
\includegraphics[scale=1.25]{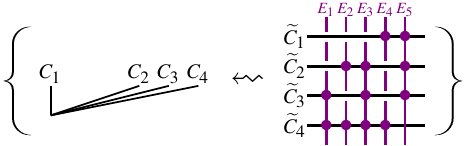} 
\end{center}
which is exactly same to the following graphical representation of the corresponding picture deformation:
\begin{center}
\includegraphics[scale=1.25]{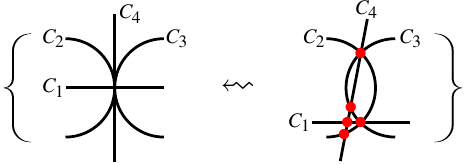} 
\end{center}

\section{Sandwiched structures on weighted homogeneous surface singularities}\label{section:WHSS-usual-sandwiched-structure}

A surface singularity $X$ is called \emph{weighted homogeneous} if it has a good $\mathbb{C}^{\ast}$-action; equivalently, $X=\Spec{B}$, where $B=\oplus_{i=0}^{\infty} B_i$ is a positively graded ring. Hence $X$ is defined by weighted homogeneous polynomials as an affine variety.

\subsection{Minimal resolutions}

We assume that $X$ is normal. Following Orlik--Wagreich~\cite{Orlik-Wagreich-1971}, the dual graph of the minimal resolution of $X$ is of \emph{star-shape}. That is, it is a connected tree such that at most one vertex is connected to more than two other vertices. The component of the exceptional divisor that intersects more than two components (if it exists) is called the \emph{central curve}, denoted by $E_0$. In the dual graph, the vertex representing the central curve is called the \emph{node}.

If there is no central curve, then $X$ is a cyclic quotient surface singularity. Suppose that $X$ has the central curve. The connected components of the dual graph after deleting the node are called the \emph{branches} of the graph. Such connected components are linear chains of vertices, representing dual graphs of cyclic quotient surface singularities. The exceptional curves of the $i$-th branch are denoted by $B_{ij}$ ($1 \le i \le r$, $1 \le j \le r_i$), where $B_{i1}$ intersects $E_0$ and $B_{ij}$ intersects $B_{i\,j+1}$. Let $d=-(E_0 \cdot E_0)$, and $b_{ij} = - (B_{ij} \cdot B_{ij})$. Set
\begin{equation*}
n_i/b_i = [b_{i1},\cdots,b_{ir_i}]
\end{equation*}
where $n_i$ and $b_i$ are positive integers with $n_i > b_i$ and $(n_i,b_i)=1$. The dual graph of the minimal resolution of $X$ is given as in Figure~\ref{figure:dual-graph-WHSS}.

\begin{lemma}[{Pinkham~\cite[Theorem 2.1]{Pinkham-1977}}]\label{lemma:cross-ratios}
Let $Q_i$ ($i=1,\dotsc,t$) be the intersection point $E_0 \cap B_{i1}$. Then the analytic type of $X$ is determined by the isomorphism class of the pair $(E_0, \{Q_1, \dotsc, Q_t\})$.
\end{lemma}

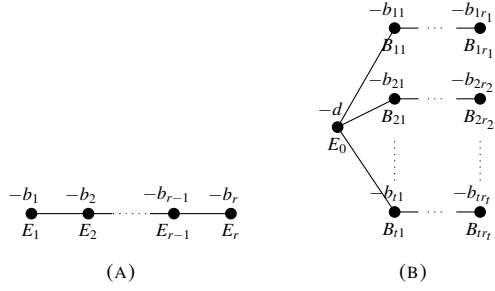
\begin{figure}
\subfloat[]{
\begin{tikzpicture}[scale=0.75]
\node[bullet] (10) at (1,0) [labelAbove={$-b_1$},label=below:{$E_1$}] {};
\node[bullet] (20) at (2,0) [labelAbove={$-b_2$},label=below:{$E_2$}] {};

\node[empty] (250) at (2.5,0) [] {};
\node[empty] (30) at (3,0) [] {};

\node[bullet] (350) at (3.5,0) [labelAbove={$-b_{r-1}$},label=below:{$E_{r-1}$}] {};
\node[bullet] (450) at (4.5,0) [labelAbove={$-b_r$},label=below:{$E_r$}] {};

\draw [-] (10)--(20);
\draw [-] (20)--(250);
\draw [dotted] (20)--(350);
\draw [-] (30)--(350);
\draw [-] (350)--(450);
\end{tikzpicture}}
\qquad
\subfloat[]{
\begin{tikzpicture}[scale=0.75]
\node[bullet] (10) at (1,0) [labelAbove={$-d$},label=below:{$E_0$}] {};

\node[bullet] (2-175) at (2,1.75) [labelAbove={$-b_{11}$},label=below:{$B_{11}$}] {};
\node[empty] (25-175) at (2.5,1.75) [] {};
\node[empty] (3-175) at (3,1.75) [] {};
\node[bullet] (35-175) at (3.5,1.75) [labelAbove={$-b_{1\smash[b]{r_1}}$},label=below:{$B_{1r_1}$}] {};

\node[bullet] (2-05) at (2,0.5) [labelAbove={$-b_{21}$},label=below:{$B_{21}$}] {};
\node[empty] (25-05) at (2.5,0.5) [] {};
\node[empty] (3-05) at (3,0.5) [] {};
\node[bullet] (35-05) at (3.5,0.5) [labelAbove={$-b_{2\smash[b]{r_2}}$},label=below:{$B_{2r_2}$}] {};

\node[bullet] (2--15) at (2,-1.5) [labelAbove={$-b_{t1}$},label=below:{$B_{t1}$}] {};
\node[empty] (25--15) at (2.5,-1.5) [] {};
\node[empty] (3--15) at (3,-1.5) [] {};
\node[bullet] (35--15) at (3.5,-1.5) [labelAbove={$-b_{t\smash[b]{r_t}}$},label=below:{$B_{tr_t}$}] {};

\draw[-] (10)--(2-175);
\draw[-] (2-175)--(25-175);
\draw[dotted] (25-175)--(3-175);
\draw[-] (3-175)--(35-175);

\draw[-] (10)--(2-05);
\draw[-] (2-05)--(25-05);
\draw[dotted] (25-05)--(3-05);
\draw[-] (3-05)--(35-05);

\draw[dotted] (2,-0.25)--(2,-1);
\draw[dotted] (3.5,-0.25)--(3.5,-1);

\draw[-] (10)--(2--15);
\draw[-] (2--15)--(25--15);
\draw[dotted] (25--15)--(3--15);
\draw[-] (3--15)--(35--15);
\end{tikzpicture}}

\caption{The dual graph of a weighted homogeneous surface singularity}
\label{figure:dual-graph-WHSS}
\end{figure}

\begin{definition}\label{definition:WHSS-big-node}
Let $X$ be a weighted homogeneous surface singularity. Suppose that $X$ is has the central curve $E_0$ and $t$ branches. We say that $X$ has a \emph{big node} if $d=-(E_0 \cdot E_0) \ge t+1$.
\end{definition}

Such a singularity $X$ is rational and it has a sandwiched structure. Indeed Pinkham~\cite[Corollary~5.7]{Pinkham-1977} proved that $X$ is a rational surface singularity if and only if the central curve $E_0$ is rational and $kd - \sum_{i=1}^{t} \lceil kb_i/n_i \rceil > -2$ for all $k > 0$, where $\lceil a \rceil$ is the least integer greater than or equal to $a$. So if $X$ has a big node, then it is rational. Furthermore, $X$ is a minimal surface singularity. Therefore it has a sandwiched structure.

We will introduce `usual' sandwiched structures of cyclic quotient surface singularities and weighted homogeneous surface singularities with the big nodes in the following sections.

\subsection{Cyclic quotient surface singularities}

Let $X$ be a cyclic quotient surface singularity $\frac{1}{n}(1,a)$ with $(n,a)=1$. Let $(V,E)$ be the minimal resolution of $(X,p)$, where $E=E_1 \cup \dotsb \cup E_r$ is the union of the exceptional divisors $E_i$ whose dual graph is given in Figure~\ref{figure:dual-graph-WHSS}(A), where $n/a = [b_1,\dotsc,b_r]$. Némethi and Popescu-Pampu~\cite[\S7.1]{NPP-2010-PLMS} introduced how to construct decorated curves for minimal surface singularities, that is, rational surface singularities with reduced fundamental cycles. Cyclic quotient surface singularities can also be considered minimal one. For instance, we add $(b_i-2)$ number of $(-1)$-vertices for each $(-b_i)$-vertices $E_i$ if $b_i \ge 3$. Finally we add one more $(-1)$-vertex to $E_r$. In Figure~\ref{figure:sandwiched-structure-for-cyclic}, we illustrate a sandwiched structure of $X$ in this manner. Then, for each $(-1)$-vertex of the graph $\Gamma'$, we choose a `curvetta' $\widetilde{C}_i$. Blowing down $\Gamma'$ to a smooth point, we get a germ of a plane curve singularity $C$ with all the components $C_i$ smooth. The decorations $l_i$ for the components $C_i$ are given by the following lemma.

\begin{lemma}[{de Jong--van Straten~\cite[Example~1.5(4)]{deJong-vanStraten-1998}} or {Némethi--Popescu-Pampu~\cite[Lemma~7.1.1]{NPP-2010-PLMS}}]\label{lemma:determine-l}
The weight $l_i$ equal to the distance from the vertex $E_1$ to the $(-1)$-vertex associated to the curvetta $C_i$.
\end{lemma}

\begin{figure}
\centering

\begin{tikzpicture}[scale=2]
\tikzset{font=\scriptsize}
\node[bullet] (10) at (1,0) [labelAbove={$-b_1$}] {};

\node[bullet] (10-1) at (0.8,-0.5) [label={[label distance=-0.35em]above left:{$-1$}}] {};
\node[smallbullet] at (0.9,-0.5) [] {};
\node[smallbullet] at (1,-0.5) [] {};
\node[smallbullet] at (1.1,-0.5) [] {};
\node[bullet] (10-2) at (1.2,-0.5) [label={[label distance=-0.35em]above right:{$-1$}}] {};

\draw [decorate, decoration = {calligraphic brace,mirror}] (0.8,-0.6)--(1.2,-0.6) node[pos=0.5,below=0.1em,black]{$b_1-2$};

\node[bullet] (20) at (2,0) [labelAbove={$-b_2$}] {};

\node[bullet] (20-1) at (1.8,-0.5) [label={[label distance=-0.35em]above left:{$-1$}}] {};
\node[smallbullet] at (1.9,-0.5) [] {};
\node[smallbullet] at (2,-0.5) [] {};
\node[smallbullet] at (2.1,-0.5) [] {};
\node[bullet] (20-2) at (2.2,-0.5) [label={[label distance=-0.35em]above right:{$-1$}}] {};

\draw [decorate, decoration = {calligraphic brace,mirror}] (1.8,-0.6)--(2.2,-0.6)  node[pos=0.5,below=0.1em,black]{$b_2-2$};

\node[empty] (250) at (2.5,0) [] {};
\node[empty] (30) at (3,0) [] {};

\node[bullet] (350) at (3.5,0) [labelAbove={$-b_{r-1}$}] {};

\node[bullet] (350-1) at (3.3,-0.5) [label={[label distance=-0.35em]above left:{$-1$}}] {};
\node[smallbullet] at (3.4,-0.5) [] {};
\node[smallbullet] at (3.5,-0.5) [] {};
\node[smallbullet] at (3.6,-0.5) [] {};
\node[bullet] (350-2) at (3.7,-0.5) [label={[label distance=-0.35em]above right:{$-1$}}] {};

\draw [decorate, decoration = {calligraphic brace,mirror}] (3.3,-0.6)--(3.7,-0.6)  node[pos=0.5,below=0.1em,black]{$b_{r-1}-2$};

\node[bullet] (450) at (4.5,0) [labelAbove={$-b_r$}] {};

\node[bullet] (450-1) at (4.3,-0.5) [label={[label distance=-0.35em]above left:{$-1$}}] {};
\node[smallbullet] at (4.4,-0.5) [] {};
\node[smallbullet] at (4.5,-0.5) [] {};
\node[smallbullet] at (4.6,-0.5) [] {};
\node[bullet] (450-2) at (4.7,-0.5) [label={[label distance=-0.35em]above right:{$-1$}}] {};

\draw [decorate, decoration = {calligraphic brace,mirror}] (4.3,-0.6)--(4.7,-0.6)  node[pos=0.5,below=0.1em,black]{$b_r-1$};

\draw [-] (10)--(20);
\draw [-] (20)--(250);
\draw [dotted] (20)--(350);
\draw [-] (30)--(350);
\draw [-] (350)--(450);

\draw [-] (10)--(10-1);
\draw [-] (10)--(10-2);

\draw [-] (20)--(20-1);
\draw [-] (20)--(20-2);

\draw [-] (350)--(350-1);
\draw [-] (350)--(350-2);

\draw [-] (450)--(450-1);
\draw [-] (450)--(450-2);
\end{tikzpicture}
\caption{An usual sandwiched structure $\Gamma'$}
\label{figure:sandwiched-structure-for-cyclic}
\end{figure}
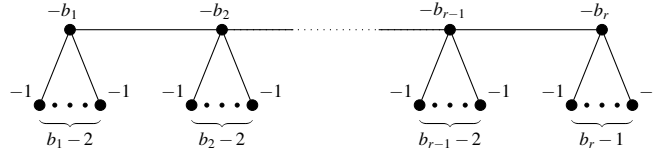

\subsection{Weighted homogeneous surface singularities with the big nodes}

A sandwiched structure for a weighted homogeneous surface singularity $X$ with the big node will be given as follows. For each arms, we add $(-1)$-vertices as we did for cyclic quotient surface singularities in the previous subsection. That is, we add ($b_{ij}-2$) $(-1)$-vertices to each exceptional vertex $B_{ij}$ with $j < r_i$ (only for $b_{ij} > 2$) and we add ($b_{ir_i}-1$) $(-1)$-vertices to $B_{ir_i}$ for all $i=1,\dotsc,t$. On the node $E_0$, we add $(d-t-1)$ $(-1)$-vertices to $E_0$. Let $\widetilde{C}_{ij}^{k}$ be a curvetta corresponding to the $(-1)$-curve attaching $B_{ij}$, where $k=1,\dotsc,b_{ij}-2$ for $j < r_i$ and $k=1,\dotsc,b_{ir_i}-1$ for $j=r_i$. We denote curvettas over $E_0$ by $\widetilde{C}_1,\dotsc,\widetilde{C}_{d-t-1}$. Notice that the corresponding decorated curves $C_{ij}^{k}$ and $C_l$ are all smooth at the origin $(0,0)$ and their decoration numbers are determined by Lemma~\ref{lemma:determine-l}.

\section{Compactifications of weighted homogeneous surface singularities}\label{section:compactification-WHSS}

Let $X$ be a cyclic quotient surface singularity or a weighted homogeneous surface singularity with the big node. We introduce various ways of compactifying $X$.

As a sandwiched surface singularity, we may consider the compatible compactification of $X$ described in Section~\ref{section:compactification}. We denote the compactification by $\overline{X}^{JS}$. On the other hand, Pinkham~\cite{Pinkham-1977} introduced a compactification of $X$ induced from its $\mathbb{C}^{\ast}$-action. Since $X=\Spec{B}$ for some positively graded ring $B$, one may compactify $X$ by taking $\Proj(B[t])$, where $t$ is given degree $1$ for the grading of $B[t]$. So we can define another compactification of $X$, denoted by $\overline{X}^{P}$, as the minimal resolution of the $\mathbb{C}^{\ast}$-compactification by Pinkham. Finally, we introduce an additional compactification of $X$, denoted by $\overline{X}^{PS}$, in case $X$ is a weighted homogeneous surface singularity with the big node. It will be a bridge between two compactifications $\overline{X}^{JS}$ and $\overline{X}^{P}$.

\subsection{The compactifications $\overline{X}^{JS}$}\label{section:X^JS}

We would like to describe the construction of the compatible compactification given in Section~\ref{section:compactification} a little more systematically in case of weighted homogeneous surface singularities.

At first, suppose that $X$ is a cyclic quotient surface singularity. We can construct a smooth surface $\widetilde{X}^{JS}$ by blowing up $\mathbb{CP}^2$ at $[0,0,1]$ (including its infinitely near points) as given inductively in Figure~\ref{figure:blowing-ups}. Explicitly, it is as follows. We blow up $\mathbb{CP}^2$ at $[0,0,1]$ so that we have a $(-1)$-curve $E_1$. We blow up at ($b_1-2$) different points on $E_1$. These ($b_1-2$) exceptional curves become the sandwiched $(-1)$-curve attached on $E_1$. We blow up at some point of $E_1$ so that we have the exceptional curve $E_1$ with $E_1 \cdot E_1 = -b_1$ and a new $(-1)$-curve $E_2$. Repeating this process to $E_2$ and so on, we obtain a smooth surface $\widetilde{X}^{JS}$ that contains the exceptional divisor of the singularity $X$.

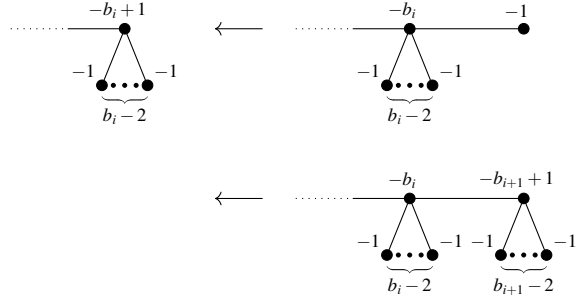
\begin{figure}
\centering
\begin{tikzpicture}[scale=1.5]
\tikzset{font=\scriptsize}
\begin{scope}[shift={(0,0)}]
\node[empty] (00) at (0,0) [] {};
\node[empty] (050) at (0.5,0) [] {};
\draw [dotted] (0,0)--(0.5, 0);
\draw [-] (0.5,0)--(1,0);

\node[bullet] (10) at (1,0) [labelAbove={$-b_i+1$}] {};

\node[bullet] (10-1) at (0.8,-0.5) [label={[label distance=-0.35em]above left:{$-1$}}] {};
\node[smallbullet] at (0.9,-0.5) [] {};
\node[smallbullet] at (1,-0.5) [] {};
\node[smallbullet] at (1.1,-0.5) [] {};
\node[bullet] (10-2) at (1.2,-0.5) [label={[label distance=-0.35em]above right:{$-1$}}] {};

\draw [decorate, decoration = {calligraphic brace,mirror}] (0.8,-0.6)--(1.2,-0.6) node[pos=0.5,below=0.1em,black]{$b_i-2$};

\draw [-] (10)--(10-1);
\draw [-] (10)--(10-2);
\end{scope}

\node[empty] (1750) at (1.75,0) [] {};
\node[empty] (2250) at (2.25,0) [] {};

\draw[<-] (1750) -- (2250);

\begin{scope}[shift={(2.5,0)}]
\node[empty] (00) at (0,0) [] {};
\node[empty] (050) at (0.5,0) [] {};
\draw [dotted] (0,0)--(0.5, 0);
\draw [-] (0.5,0)--(1,0);

\node[bullet] (10) at (1,0) [labelAbove={$-b_i$}] {};

\node[bullet] (10-1) at (0.8,-0.5) [label={[label distance=-0.35em]above left:{$-1$}}] {};
\node[smallbullet] at (0.9,-0.5) [] {};
\node[smallbullet] at (1,-0.5) [] {};
\node[smallbullet] at (1.1,-0.5) [] {};
\node[bullet] (10-2) at (1.2,-0.5) [label={[label distance=-0.35em]above right:{$-1$}}] {};

\draw [decorate, decoration = {calligraphic brace,mirror}] (0.8,-0.6)--(1.2,-0.6) node[pos=0.5,below=0.1em,black]{$b_i-2$};

\draw [-] (10)--(10-1);
\draw [-] (10)--(10-2);

\node[bullet] (20) at (2,0) [labelAbove={$-1$}] {};

\draw [-] (10) -- (20);

\end{scope}

\node[empty] (175-15) at (1.75,-1.5) [] {};
\node[empty] (225-15) at (2.25,-1.5) [] {};

\draw[<-] (175-15) -- (225-15);

\begin{scope}[shift={(2.5,-1.5)}]
\node[empty] (00) at (0,0) [] {};
\node[empty] (050) at (0.5,0) [] {};
\draw [dotted] (0,0)--(0.5, 0);
\draw [-] (0.5,0)--(1,0);

\node[bullet] (10) at (1,0) [labelAbove={$-b_i$}] {};

\node[bullet] (10-1) at (0.8,-0.5) [label={[label distance=-0.35em]above left:{$-1$}}] {};
\node[smallbullet] at (0.9,-0.5) [] {};
\node[smallbullet] at (1,-0.5) [] {};
\node[smallbullet] at (1.1,-0.5) [] {};
\node[bullet] (10-2) at (1.2,-0.5) [label={[label distance=-0.35em]above right:{$-1$}}] {};

\draw [decorate, decoration = {calligraphic brace,mirror}] (0.8,-0.6)--(1.2,-0.6) node[pos=0.5,below=0.1em,black]{$b_i-2$};

\draw [-] (10)--(10-1);
\draw [-] (10)--(10-2);

\node[bullet] (20) at (2,0) [labelAbove={$-b_{i+1}+1$}] {};

\node[bullet] (20-1) at (1.8,-0.5) [label={[label distance=-0.35em]above left:{$-1$}}] {};
\node[smallbullet] at (1.9,-0.5) [] {};
\node[smallbullet] at (2,-0.5) [] {};
\node[smallbullet] at (2.1,-0.5) [] {};
\node[bullet] (20-2) at (2.2,-0.5) [label={[label distance=-0.35em]above right:{$-1$}}] {};

\draw [decorate, decoration = {calligraphic brace,mirror}] (1.8,-0.6)--(2.2,-0.6)  node[pos=0.5,below=0.1em,black]{$b_{i+1}-2$};

\draw [-] (20)--(20-1);
\draw [-] (20)--(20-2);

\draw [-] (10) -- (20);
\end{scope}

\end{tikzpicture}
\caption{A sequence of blowing-ups}
\label{figure:blowing-ups}
\end{figure}

Suppose now that $X$ is a weighted homogeneous surface singularity with the big node. We can construct $\widetilde{X}^{JS}$ as follows. Let $L_{\infty}$ be the line at infinity in $\mathbb{CP}^2$ defined by $\{[x,y,z] \in \mathbb{CP}^2 \mid z=0\}$. Let $L_1, \dotsc, L_t$ and $M_1, \dotsc, M_{d-t-1}$ be different lines passing through $[0,0,1]$. We blow up $\mathbb{CP}^2$ at $[0,0,1]$. We have the Hirzebruch surface $\mathbb{F}_1$ with the negative section $E_0$, which is the exceptional curve of the blowing up, and with a positive section $E_{\infty}$ that is the image of $L_{\infty}$. We denote again by $L_i$ the proper images of the line $L_i$. Let $Q_i=E_{\infty} \cap L_i$ for each $i=1,\dotsc,t$. We may choose $L_i$'s so that the isomorphism class of $(E_0, P_1,\dotsc,P_t)$ is the same as that of the singularity $X$.

For each $i=1,\dotsc,t$, we can blow up at $P_i$ in the same way as above for a cyclic quotient surface singularity so that we construct the $i$th branch of $\widetilde{X}^{JS}$. We then blow up once at each $R_j$'s for $j=1,\dotsc,d-t-1$. The image of $E_0$ becomes the central curve of $\widetilde{X}$. Notice that the images of $M_j$ becomes $(-1)$-curves intersecting $E_{\infty}$.

\begin{definition}
The \emph{de Jong--van Straten compactification} $\overline{X}^{JS}$ of $X$ is the singular surface obtained by contracting the exceptional divisor for $X$ in the smooth surface $\widetilde{X}^{JS}$ constructed in the above.
\end{definition}

\begin{remark}\label{remark:decorated-curve-over-E0}
We may chooses the images of $M_j$ as the compactified decorated curves of $\overline{X}^{JS}$ over the central curve $E_0$.
\end{remark}

\subsection{The compactifications $\overline{X}^{P}$}\label{section:X^{P}}

We briefly recall the compactifications and their deformations introduced in Pinkham~\cite{Pinkham-1977, Pinkham-1978}. In the below, we will construct the minimal resolution $\widetilde{X}^{P}$ of the compactification $\overline{X}^{P}$.

At first, let $X$ be a cyclic quotient surface singularity $\frac{1}{n}(1,a)$ with $1 \le a < n$ and $(n,a)=1$. Suppose that the dual graph of its minimal resolution is given as in Figure~\ref{figure:dual-graph-WHSS}(A), where $n/a=[b_1,\dotsc,b_r]$. Let $L_{\infty}$ be the line at infinity defined by $\{[x,y,z] \in \mathbb{CP}^2 \mid z=0\}$ and let $L$ be another line through $[0,0,1]$. We can blow up $\mathbb{CP}^2$ at $[0,0,1]$ and its infinitely near points so that we have a complex surface $\widetilde{X}^{P}$ which contains a configuration of rational curves whose dual graph is given by Figure~\ref{figure:dual-graph-X^{P}}(A), where $n/(n-a) = [a_1,\dotsc,a_e]$. Such a sequence of blow-ups exists because the sequence $(b_1,\dotsc,b_r,1,a_e,\dotsc,a_1)$ represents the zero Hirzebruch-Jung continued fraction, that is, $[b_1,\dotsc,b_r,1,a_e,\dotsc,a_1]=0$. We denote the image of $L_{\infty}$ in $\widetilde{X}^{P}$ by $E_{\infty}$ for convenience.

Next, let $X$ be a weighted homogeneous surface singularity with the big node, whose dual graph of the minimal resolution $\widetilde{X}$ is given by Figure~\ref{figure:dual-graph-WHSS}. Let $P_i$ ($i=1,\dotsc,t$) be the point on the central curve $E_0 \subset \widetilde{X}$ where $B_{1j}$ intersect with $E_0$. Let $\mathbb{F}_d = \mathbb{P}(\sheaf{O} \oplus \sheaf{O}(-d))$ be the Hirzebruch surface. It has a unique negative section, denoted again by $E_0$, with $E_0 \cdot E_0 = -d$, and it also has a positive section $E_{\infty}$ with $E_{\infty} \cdot E_{\infty} = d$. Choose an isomorphism of $E_{\infty}$ with the central curve $E_0$ and denote by $Q_i (\in E_{\infty})$ ($i=1,\dotsc,t$) the images of the $P_i \in E_0$, where $P_i$'s are the intersection points defined in Lemma~\ref{lemma:cross-ratios}. We can blow up $\mathbb{F}_d$ at $Q_i \in E_{\infty}$ and then blow up suitable points on the exceptional divisors, not on the proper transform of $E_{\infty}$, so that we obtain a surface $\widetilde{X}^{P}$ that contains smooth rational curves $B_{ij}$ and $A_{ij}$ whose dual graph is given in Figure~\ref{figure:dual-graph-X^{P}}(B), where $(n_i-a_i)/n_i=[a_{i1},\dotsc,a_{ie_i}]$ for $i=1,\dotsc,t$.

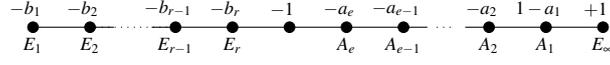
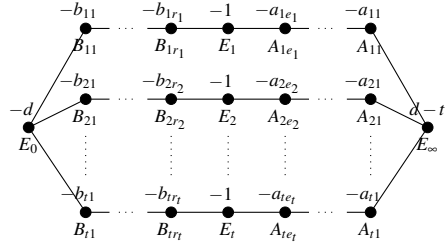
\begin{figure}
\subfloat[Cyclic quotient surface singularities]{\begin{tikzpicture}[scale=0.75]
\node[bullet] (10) at (1,0) [labelAbove={$-b_1$},label=below:{$E_1$}] {};
\node[bullet] (20) at (2,0) [labelAbove={$-b_2$},label=below:{$E_2$}] {};

\node[empty] (250) at (2.5,0) [] {};
\node[empty] (30) at (3,0) [] {};

\node[bullet] (350) at (3.5,0) [labelAbove={$-b_{r-1}$},label=below:{$E_{r-1}$}] {};
\node[bullet] (450) at (4.5,0) [labelAbove={$-b_r$},label=below:{$E_r$}] {};

\node[bullet] (550) at (5.5,0) [labelAbove={$-1$}]{};

\node[bullet] (650) at (6.5,0) [labelAbove={$-a_e$},label=below:{$A_e$}] {};
\node[bullet] (750) at (7.5,0) [labelAbove={$-a_{e-1}$},label=below:{$A_{e-1}$}] {};

\node[empty] (800) at (8,0) []{};
\node[empty] (850) at (8.5,0) []{};

\node[bullet] (900) at (9,0) [labelAbove={$-a_2$},label=below:{$A_2$}] {};
\node[bullet] (100) at (10,0) [labelAbove={$1-a_1$},label=below:{$A_1$}] {};
\node[bullet] (110) at (11,0) [labelAbove={$+1$},label=below:{$E_{\infty}$}] {};

\draw [-] (10)--(20);
\draw [-] (20)--(250);
\draw [dotted] (20)--(350);
\draw [-] (30)--(350);
\draw [-] (350)--(450);

\draw [-] (450)--(550);
\draw [-] (550)--(650);
\draw [-] (650)--(750);
\draw [-] (750)--(800);
\draw [dotted] (800)--(850);
\draw [-] (850)--(900);
\draw [-] (900)--(100);
\draw [-] (100)--(110);
\end{tikzpicture}}

\subfloat[Weighted homogeneous surface singularities]{
\begin{tikzpicture}[scale=0.75]
\node[bullet] (10) at (1,0) [labelAbove={$-d$},label=below:{$E_0$}] {};

\node[bullet] (80) at (8,0) [label=above:{$d-t$},label=below:{$E_{\infty}$}] {};

\node[bullet] (2-175) at (2,1.75) [labelAbove={$-b_{11}$},label=below:{$B_{11}$}] {};
\node[empty] (25-175) at (2.5,1.75) [] {};
\node[empty] (3-175) at (3,1.75) [] {};
\node[bullet] (35-175) at (3.5,1.75) [labelAbove={$-b_{1\smash[b]{r_1}}$},label=below:{$B_{1r_1}$}] {};

\node[bullet] (45-175) at (4.5,1.75) [labelAbove={$-1$},label=below:{$E_1$}] {};

\node[bullet] (55-175) at (5.5,1.75) [labelAbove={$-a_{1\smash[b]{e_1}}$},label=below:{$A_{1e_1}$}] {};

\node[empty] (60-175) at (6,1.75) []{};
\node[empty] (65-175) at (6.5,1.75) []{};

\node[bullet] (70-175) at (7,1.75) [labelAbove={$-a_{11}$},label=below:{$A_{11}$}] {};

\node[bullet] (2-05) at (2,0.5) [labelAbove={$-b_{21}$},label=below:{$B_{21}$}] {};
\node[empty] (25-05) at (2.5,0.5) [] {};
\node[empty] (3-05) at (3,0.5) [] {};
\node[bullet] (35-05) at (3.5,0.5) [labelAbove={$-b_{2\smash[b]{r_2}}$},label=below:{$B_{2r_2}$}] {};

\node[bullet] (45-05) at (4.5,0.5) [labelAbove={$-1$},label=below:{$E_2$}] {};

\node[bullet] (55-05) at (5.5,0.5) [labelAbove={$-a_{2\smash[b]{e_2}}$},label=below:{$A_{2e_2}$}] {};

\node[empty] (60-05) at (6,0.5) []{};
\node[empty] (65-05) at (6.5,0.5) []{};

\node[bullet] (70-05) at (7,0.5) [labelAbove={$-a_{21}$},label=below:{$A_{21}$}] {};

\node[bullet] (2--15) at (2,-1.5) [labelAbove={$-b_{t1}$},label=below:{$B_{t1}$}] {};
\node[empty] (25--15) at (2.5,-1.5) [] {};
\node[empty] (3--15) at (3,-1.5) [] {};
\node[bullet] (35--15) at (3.5,-1.5) [labelAbove={$-b_{t\smash[b]{r_t}}$},label=below:{$B_{tr_t}$}] {};

\node[bullet] (45--15) at (4.5,-1.5) [labelAbove={$-1$},label=below:{$E_t$}] {};

\node[bullet] (55--15) at (5.5,-1.5) [labelAbove={$-a_{t\smash[b]{e_t}}$},label=below:{$A_{te_t}$}] {};

\node[empty] (60--15) at (6,-1.5) []{};
\node[empty] (65--15) at (6.5,-1.5) []{};

\node[bullet] (70--15) at (7,-1.5) [labelAbove={$-a_{t1}$},label=below:{$A_{t1}$}] {};

\draw[-] (10)--(2-175);
\draw[-] (2-175)--(25-175);
\draw[dotted] (25-175)--(3-175);
\draw[-] (3-175)--(35-175);
\draw[-] (35-175)--(45-175);
\draw[-] (45-175)--(55-175);
\draw[-] (55-175)--(60-175);
\draw[dotted] (60-175)--(65-175);
\draw[-] (65-175)--(70-175);
\draw[-] (70-175)--(80);

\draw[-] (10)--(2-05);
\draw[-] (2-05)--(25-05);
\draw[dotted] (25-05)--(3-05);
\draw[-] (3-05)--(35-05);
\draw[-] (35-05)--(45-05);
\draw[-] (45-05)--(55-05);
\draw[-] (55-05)--(60-05);
\draw[dotted] (60-05)--(65-05);
\draw[-] (65-05)--(70-05);
\draw[-] (70-05)--(80);

\draw[dotted] (2,-0.15)--(2,-0.9);
\draw[dotted] (3.5,-0.15)--(3.5,-0.9);
\draw[dotted] (4.5,-0.15)--(4.5,-0.9);
\draw[dotted] (5.5,-0.15)--(5.5,-0.9);
\draw[dotted] (7,-0.15)--(7,-0.9);

\draw[-] (10)--(2--15);
\draw[-] (2--15)--(25--15);
\draw[dotted] (25--15)--(3--15);
\draw[-] (3--15)--(35--15);
\draw[-] (35--15)--(45--15);
\draw[-] (45--15)--(55--15);
\draw[-] (55--15)--(60--15);
\draw[dotted] (60--15)--(65--15);
\draw[-] (65--15)--(70--15);
\draw[-] (70--15)--(80);
\end{tikzpicture}}

\caption{The dual graph of the Pinkham compactification $\widetilde{X}^{P}$}
\label{figure:dual-graph-X^{P}}
\end{figure}

\begin{definition}\label{definition:P-compactification}
The \emph{Pinkham compactification} of $X$ is the singular surface $\overline{X}^{P}$ obtained by contracting the exceptional divisor corresponding to $X$ in the smooth surface $\widetilde{X}^{P}$ to the singular point $X$. The \emph{compactifying divisor} of $X$ is the union of the divisor $E_{\infty}$ and $A_j$'s (for cyclic cases) or $A_{ij}$'s (for weighted cases) in $\widetilde{X}^{P}$, which is denoted by $\widetilde{E}_{\infty}^{P}$.
\end{definition}

\subsection{The compactifications $\overline{X}^{PS}$}\label{section:X^{P}S}

Let $X$ be a weighted homogeneous surface singularity with the big node. We introduce another compactification of $X$.

Let $L_{\infty}$, $L_i$, and $M_j$ be the lines chosen in Section~\ref{section:X^JS}. For each $i=1,\dotsc,t$, we can blow up $P_i$ (including its infinitely near points) in the same way that we construct the $i$th linear chains of the Pinkham's compactification $\widetilde{X}^{P}$ in Section~\ref{section:X^{P}}. That is, we construct the $i$th linear chain of $\widetilde{X}^{PS}$ whose dual graph is given in Figure~\ref{figure:dual-graph-X^{P}S}. In addition, we blow up at each of the intersection points $R_j = E_0 \cap M_j$ ($j=1,\dotsc,d-t-1$). Then we construct a smooth surface $\widetilde{X}^{PS}$, whose dual graph is given by Figure~\ref{figure:dual-graph-X^{P}S}. Here $D_j$ ($j=1,\dotsc,d-t-1$) is the image of $M_j$ and $\widetilde{C}_j$ is the exceptional curve in $\widetilde{X}^{PS}$ of the blowing up at $R_j$.

Contracting $E_0$ and $B_{ij}$'s in $\widetilde{X}^{PS}$, we get a singular surface:

\begin{definition}\label{definition:PS-compactification}
The singular surface constructed in the above is called the \emph{postscript compactification} of $X$. We denote it by $\overline{X}^{PS}$. The \emph{compactifying divisor} of $\overline{X}^{PS}$ is the union of $E_{\infty}$, $A_{ij}$'s, and $\widetilde{C}_k$'s, which is denoted by $\widetilde{E}_{\infty}^{PS}$.
\end{definition}

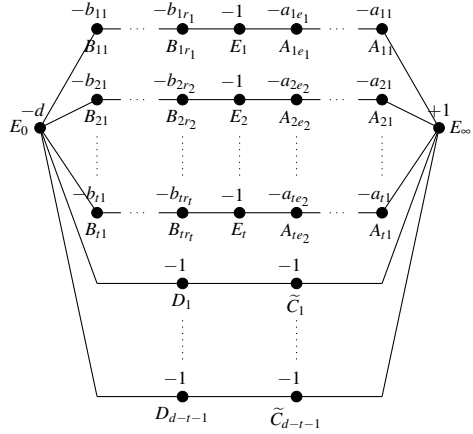
\begin{figure}
\begin{tikzpicture}[scale=0.75]
\node[bullet] (10) at (1,0) [labelAbove={$-d$},label=left:{$E_0$}] {};

\node[bullet] (80) at (8,0) [label=above:{$+1$},label=right:{$E_{\infty}$}] {};

\node[bullet] (2-175) at (2,1.75) [labelAbove={$-b_{11}$},label=below:{$B_{11}$}] {};
\node[empty] (25-175) at (2.5,1.75) [] {};
\node[empty] (3-175) at (3,1.75) [] {};
\node[bullet] (35-175) at (3.5,1.75) [labelAbove={$-b_{1\smash[b]{r_1}}$},label=below:{$B_{1r_1}$}] {};

\node[bullet] (45-175) at (4.5,1.75) [labelAbove={$-1$},label=below:{$E_1$}] {};

\node[bullet] (55-175) at (5.5,1.75) [labelAbove={$-a_{1\smash[b]{e_1}}$},label=below:{$A_{1e_1}$}] {};

\node[empty] (60-175) at (6,1.75) []{};
\node[empty] (65-175) at (6.5,1.75) []{};

\node[bullet] (70-175) at (7,1.75) [labelAbove={$-a_{11}$},label=below:{$A_{11}$}] {};

\node[bullet] (2-05) at (2,0.5) [labelAbove={$-b_{21}$},label=below:{$B_{21}$}] {};
\node[empty] (25-05) at (2.5,0.5) [] {};
\node[empty] (3-05) at (3,0.5) [] {};
\node[bullet] (35-05) at (3.5,0.5) [labelAbove={$-b_{2\smash[b]{r_2}}$},label=below:{$B_{2r_2}$}] {};

\node[bullet] (45-05) at (4.5,0.5) [labelAbove={$-1$},label=below:{$E_2$}] {};

\node[bullet] (55-05) at (5.5,0.5) [labelAbove={$-a_{2\smash[b]{e_2}}$},label=below:{$A_{2e_2}$}] {};

\node[empty] (60-05) at (6,0.5) []{};
\node[empty] (65-05) at (6.5,0.5) []{};

\node[bullet] (70-05) at (7,0.5) [labelAbove={$-a_{21}$},label=below:{$A_{21}$}] {};

\node[bullet] (2--15) at (2,-1.5) [labelAbove={$-b_{t1}$},label=below:{$B_{t1}$}] {};
\node[empty] (25--15) at (2.5,-1.5) [] {};
\node[empty] (3--15) at (3,-1.5) [] {};
\node[bullet] (35--15) at (3.5,-1.5) [labelAbove={$-b_{t\smash[b]{r_t}}$},label=below:{$B_{tr_t}$}] {};

\node[bullet] (45--15) at (4.5,-1.5) [labelAbove={$-1$},label=below:{$E_t$}] {};

\node[bullet] (55--15) at (5.5,-1.5) [labelAbove={$-a_{t\smash[b]{e_2}}$},label=below:{$A_{te_2}$}] {};

\node[empty] (60--15) at (6,-1.5) []{};
\node[empty] (65--15) at (6.5,-1.5) []{};

\node[bullet] (70--15) at (7,-1.5) [labelAbove={$-a_{t1}$},label=below:{$A_{t1}$}] {};

\node[bullet] (35--275) at (3.5,-2.75) [labelAbove={$-1$},label=below:{$D_1$}] {};

\node[bullet] (55--275) at (5.5,-2.75) [labelAbove={$-1$},label=below:{$\widetilde{C}_1$}] {};

\node[bullet] (35--475) at (3.5,-4.75) [labelAbove={$-1$},label=below:{$D_{d-t-1}$}] {};

\node[bullet] (55--475) at (5.5,-4.75) [labelAbove={$-1$},label=below:{$\widetilde{C}_{d-t-1}$}] {};

\draw[-] (10)--(2-175);
\draw[-] (2-175)--(25-175);
\draw[dotted] (25-175)--(3-175);
\draw[-] (3-175)--(35-175);
\draw[-] (35-175)--(45-175);
\draw[-] (45-175)--(55-175);
\draw[-] (55-175)--(60-175);
\draw[dotted] (60-175)--(65-175);
\draw[-] (65-175)--(70-175);
\draw[-] (70-175)--(80);

\draw[-] (10)--(2-05);
\draw[-] (2-05)--(25-05);
\draw[dotted] (25-05)--(3-05);
\draw[-] (3-05)--(35-05);
\draw[-] (35-05)--(45-05);
\draw[-] (45-05)--(55-05);
\draw[-] (55-05)--(60-05);
\draw[dotted] (60-05)--(65-05);
\draw[-] (65-05)--(70-05);
\draw[-] (70-05)--(80);

\draw[dotted] (2,-0.15)--(2,-0.9);
\draw[dotted] (3.5,-0.15)--(3.5,-0.9);
\draw[dotted] (4.5,-0.15)--(4.5,-0.9);
\draw[dotted] (5.5,-0.15)--(5.5,-0.9);
\draw[dotted] (7,-0.15)--(7,-0.9);

\draw[-] (10)--(2--15);
\draw[-] (2--15)--(25--15);
\draw[dotted] (25--15)--(3--15);
\draw[-] (3--15)--(35--15);
\draw[-] (35--15)--(45--15);
\draw[-] (45--15)--(55--15);
\draw[-] (55--15)--(60--15);
\draw[dotted] (60--15)--(65--15);
\draw[-] (65--15)--(70--15);
\draw[-] (70--15)--(80);

\draw[-] (10)--(2,-2.75);
\draw[-] (2,-2.75)--(35--275);
\draw[-] (35--275)--(55--275);
\draw[-] (55--275)--(7,-2.75);
\draw[-] (7,-2.75)--(80);

\draw[dotted] (3.5,-3.4)--(3.5,-4.15);
\draw[dotted] (5.5,-3.4)--(5.5,-4.15);

\draw[-] (10)--(2,-4.75);
\draw[-] (2,-4.75)--(35--475);
\draw[-] (35--475)--(55--475);
\draw[-] (55--475)--(7,-4.75);
\draw[-] (7,-4.75)--(80);
\end{tikzpicture}

\caption{The dual graph of the postscript compactification $\widetilde{X}^{PS}$}
\label{figure:dual-graph-X^{P}S}
\end{figure}

\begin{remark}\label{remark:X^{P}S-from-X^{P}}
We can construct $\overline{X}^{PS}$ from $\overline{X}^{P}$ as follows. Choose ($d-t-1$) distinct points $R_1, \dotsc, R_{d-t-1}$ on $E_{\infty}$ in $\overline{X}^{P}$ which are different from $P_1, \dotsc, P_t$. Blowing up $\overline{X}^{P}$ at $R_1, \dotsc, R_{d-t-1}$, we then obtain $\widetilde{X}^{PS}$. The reason why this method works is because contracting the images of fibers passing through $R_j$'s (which are $(-1)$-curves) is just the same as applying elementary transformations ($d-t-1$)-times to $\mathbb{F}_d$.
\end{remark}

\subsection{No local-to-global obstructions}

Let $\overline{X}^{\Box}$ be one of the compactifications $\overline{X}^{JS}$, $\overline{X}^{P}$, and $\overline{X}^{PS}$.

\begin{theorem}\label{theorem:no-local-to-global}
For each smoothing of $(X,p)$, there is a smoothing of $\overline{X}^{\Box}$ that is an extension of the given one of $(X,p)$ such that the compactified decorated curves for $\Box=JS$ or the compactifying divisor $\widetilde{E}_{\infty}^{P}$ for $\Box=P, JS$ does not disappear by the smoothing of $\overline{X}^{\Box}$.
\end{theorem}

\begin{proof}
In case of $\Box=JS$, we prove it in Section~\ref{section:compactification}. For $\Box=P$, refer Pinkham~[Theorem~2.9]\cite{Pinkham-1978}.

Suppose that $\Box=PS$. We can show that there is no local-to-global obstruction to $\overline{X}^{PS}$ by proving that $H^2(\overline{X}^{PS},\sheaf{T})=0$. The proof is identical to that presented in Theorem~\ref{theorem:extension-of-deformation}. So every deformation of singularities in $\overline{X}^{PS}$ can be extended to that of $\overline{X}^{PS}$. We contract the arms connected to $E_{\infty}$ that are dual to that of the minimal resolution of $(X,p)$ to cyclic quotient surface singularities so that we have a new singular surface $\widehat{X}^{PS}$ containing one weighted homogeneous surface singularities and several cyclic quotient surface singularities. One can show that there is also no local-to-global obstruction to $\widehat{X}^{PS}$.

For each smoothing of $(X,p)$, there is a deformation $\widehat{\mathcal{X}}^{PS} \to \Delta$ of $\widehat{X}^{PS}$ which is an extension of the given smoothing of $(X,p)$ but also which is an extension of locally trivial deformations of each cyclic quotient surface singularities. By taking simultaneous resolutions of each cyclic quotient surface singularities in each fiber of $\widehat{\mathcal{X}}^{PS} \to \Delta$, we have a smoothing of $\overline{X}^{PS} \to \Delta$ that is an extension of the given smoothing of $(X,p)$, where the compactifying divisor $\widetilde{E}_{\infty}$ does not disappear.
\end{proof}

\section{Topology of compactifications}\label{section:Topology-compactifications}

Let $X$ be a cyclic quotient surface singularity or a weighted homogeneous surface singularity with the big node. Let $\mathcal{X} \to \Delta$ be a smoothing of $X$ over a small disk $\Delta$. For each $\Box=JS, P$, let $\overline{\mathcal{X}}^{\Box} \to \Delta$ be a smoothing of the compactification $\overline{X}^{\Box}$, which is an extension of $\mathcal{X} \to \Delta$ given in Proposition~\ref{theorem:no-local-to-global}.

We first investigate the homologies of two smooth surfaces $\widetilde{X}^{\Box}$ and a general fiber $\overline{X}^{\Box}_t$ of $\overline{\mathcal{X}}^{\Box} \to \Delta$. Then we can compare general fibers of $\overline{\mathcal{X}}^{\Box} \to \Delta$.

\subsection{Topology of $\widetilde{X}^{JS}$: Cyclic cases} \label{section:topology-CQSS-JS}

Let $X$ be a cyclic quotient surface singularity. One can describe generators of $H_2(\widetilde{X}^{JS};\mathbb{Z})$ and homology classes of the exceptional curves $E_i$ using the \emph{extended} Riemenschneider dot diagram:

\begin{definition}
The \emph{extended Riemenschneider dot diagram} of the dual graph given in Figure~\ref{figure:dual-graph-WHSS}(A) consists of the dots of the Riemenschneider dot diagram of the dual graph and one additional dot lying diagonally below to the dot in the last row.
\end{definition}

For each $i=1,\dotsc,r$, we label each dots in the $i$-th row of the extended Riemenschneider dot diagram by $e_{i,0}, \dotsc, e_{i,b_i-2}$ from the left as below. For $i=r+1$, we label the additional dot by $e_{r,b_r-1}$
\begin{equation}\label{equation:labeling-row}
\begin{aligned}
\begin{tikzpicture}[scale=0.75]
\node[smallbullet] at (-0.5,1) []{};
\node[smallbullet] at (-0.75,1) []{};
\node[smallbullet] at (-1,1) []{};

\node[bullet] (01) at (0,1) [label=above:{$e_{{i-1},b_{i-1}-2}$}] {};

\node[bullet] (00) at (0,0) [label=above:{$e_{i,0}$}] {};
\node[bullet] (10) at (1,0) [label=above:{$e_{i,1}$}] {};

\node[smallbullet] at (1.75,0) []{};
\node[smallbullet] at (2,0) []{};
\node[smallbullet] at (2.25,0) []{};

\node[bullet] (30) at (3,0) [label=above:{$e_{i,b_i-2}$}] {};

\node[bullet] (3-1) at (3,-1) [label=above:{$e_{i+1,0}$}] {};

\node[smallbullet] at (3.5,-1) []{};
\node[smallbullet] at (3.75,-1) []{};
\node[smallbullet] at (4,-1) []{};
\end{tikzpicture}
\end{aligned}
\end{equation}

Then we may consider each $e_{i,j}$'s represent the homology classes of exceptional 2-spheres in $H_2(\widetilde{X}^{JS};\mathbb{Z})$. In particular, $e_{r,b_r-1}$ represents the exceptional divisor of the final blow-up for constructing $\widetilde{X}^{JS}$ among the sequence of blowing-ups given in Figure~\ref{figure:blowing-ups}.

\begin{lemma}\label{lemma:homology-E-CQSS-JS}
The homology classes of the exceptional divisors $E_i$ in $\widetilde{X}^{JS}$ given in Figure~\ref{figure:dual-graph-WHSS}(A) are given as follows: For $i<r$,
\begin{equation}\label{equation:homology-E-row}
\begin{aligned}
[E_i]&=e_{i,0}-e_{i,1}-\dotsb-e_{i,b_i-2}-e_{i+1,0} \\
[E_r]&=e_{r,0}-e_{r,1}-\dotsb-e_{r,b_r-2}-e_{r,b_r-1}
\end{aligned}
\end{equation}
\end{lemma}

Let's denote again by $(C,l)$ (for simplicity) the compactified decorated curve of $\overline{X}^{JS}$ and let $\widetilde{C}$ be the proper transform of $C$ in $\widetilde{X}^{JS}$. Each homology classes $e_{i,j}$ for $j \ge 1$ represent $(-1)$-curves over $E_i$ in $\widetilde{X}^{JS}$ given in Figure~\ref{figure:sandwiched-structure-for-cyclic}. Let $\widetilde{C}_{i,j}$ be the component of $\widetilde{C}$ corresponding to $e_{i,j}$. We denote the homology class of the line at infinity $L_{\infty}$ in $\widetilde{X}^{JS}$ by $l_{\infty}$.

\begin{lemma}\label{lemma:homology-C-CQSS-JS}
There is a positive integer $n_{i,j}$ such that the homology classes of $\widetilde{C}_{i,j}$ for $j \ge 1$ are given by
\begin{equation}\label{equation:homology-D}
[\widetilde{C}_{i,j}] = n_{i,j} l_{\infty} -e_{1,0}-e_{2,0}-\dotsb-e_{i,0}-e_{i,j}.
\end{equation}
\end{lemma}

\begin{example}\label{example:3-4-2-homology-hair}
Let $(X,p)$ be a cyclic quotient surface singularity $\frac{1}{19}(1,7)$. We label the extended Rimenschneider dot diagram for $19/7=[3,4,2]$ as follows.
\begin{equation*}
\begin{tikzpicture}[scale=0.75]
\node[bullet] at (0,3) [labelAbove={$e_1$}] {};
\node[bullet] at (1,3) [labelAbove={$e_2$}] {};
\node[bullet] at (1,2) [labelAbove={$e_3$}] {};
\node[bullet] at (2,2) [labelAbove={$e_4$}] {};
\node[bullet] at (3,2) [labelAbove={$e_5$}] {};
\node[bullet] at (3,1) [labelAbove={$e_6$}] {};
\node[circle] at (4,0) [labelAbove={$e_7$}] {};
\end{tikzpicture}
\end{equation*}
Then the homology classes of $E_i$ and $\widetilde{C}_j$ are given by
\begin{equation*}
\begin{aligned}
[E_1] &= e_1-e_2-e_3 \\
[E_2] &= e_3-e_4-e_5-e_6 \\
[E_3] &= e_6-e_7
\end{aligned}
\end{equation*}
and
\begin{equation*}
\begin{aligned}
[\widetilde{C}_1] &= 2l_{\infty}-e_1-e_2 \\
[\widetilde{C}_2] &= 2l_{\infty}-e_1-e_3-e_4 \\
[\widetilde{C}_3] &= 2l_{\infty}-e_1-e_3-e_5 \\
[\widetilde{C}_4] &= 2l_{\infty}-e_1-e_3-e_6-e_7
\end{aligned}
\end{equation*}
\end{example}

\subsection{Topology of $\widetilde{X}^{P}$: Cyclic cases}\label{section:topology-CQSS-P}

Suppose that the dual graph of $\widetilde{X}^{P}$ is given as in Figure~\ref{figure:dual-graph-X^{P}}(A). We would like to represent the homology classes of $E_i$'s and $A_j$'s in $H_2(\widetilde{X}^{P};\mathbb{Z})$. We use again the extended Riemenschneider dot diagram for $n/a$ as before. For $E_i$'s, it is convenient to label each dots as in Equation~\eqref{equation:labeling-row}.

\begin{lemma}\label{lemma:homology-E-CQSS-P}
The homology classes of $E_i$'s are presented as follows:
\begin{equation}\label{equation:homology-column}
\begin{aligned}
[E_i]&=e_{i,0}-e_{i,1}-\dotsb-e_{i,b_i-2}-e_{i+1,0} \\
[E_r]&=e_{r,0}-e_{r,1}-\dotsb-e_{r,b_r-2}-e_{r,b_r-1}
\end{aligned}
\end{equation}
\end{lemma}

\begin{proof}
This follows from the construction described in Section~\ref{section:X^{P}}.
\end{proof}

But, for $A_j$'s, it is convenient that we decorate each dots in the $j$th column of the original Riemenschneider dot diagram as belows:
\begin{equation*}
\begin{tikzpicture}[scale=0.75]
\node[smallbullet] at (-1.5,1.25) []{};
\node[smallbullet] at (-1.5,1) []{};
\node[smallbullet] at (-1.5,0.75) []{};

\node[bullet] at (-1.5,0) [labelAbove={$e_{a_{j-1}-2,j-1}$}] {};

\node[bullet] at (0,0) [labelAbove={$e_{0,j}$}] {};
\node[bullet] at (0,-1) [labelAbove={$e_{1,j}$}] {};

\node[smallbullet] at (0,-1.75) []{};
\node[smallbullet] at (0,-2) []{};
\node[smallbullet] at (0,-2.25) []{};

\node[bullet] at (0,-3) [labelAbove={$e_{a_j-2,j}$}] {};

\node[bullet] at (1.5,-3) [labelAbove={$e_{0,j+1}$}] {};

\node[smallbullet] at (1.5,-3.25) []{};
\node[smallbullet] at (1.5,-3.5) []{};
\node[smallbullet] at (1.5,-3.75) []{};
\end{tikzpicture}
\end{equation*}
And we label the additional dot by $e_{a_s-1,s}$, which represents the homology class of the exceptional divisor of the final blow-up in the sequence of blowing ups for constructing $\widetilde{X}^{P}$. Then

\begin{lemma}\label{lemma:homology-A-CQSS-P}
The homology classes of $A_j$'s are given as follows:
\begin{equation}\label{equation:homology-A}
\begin{aligned}
[A_1] &= l_{\infty}-e_{0,1}-\dotsc-e_{a_1-2,0}-e_{0,2},\\
[A_j] &= e_{0,j}-e_{1,j}-\dotsc-e_{a_j-2,j}-e_{0,j+1},\\
[A_s] &= e_{0,s}-e_{1,s}-\dotsc-e_{a_s-2,s}-e_{a_s-1,s}.
\end{aligned}
\end{equation}
where $1 < j < s$.
\end{lemma}

\begin{example}[Continued from Example~\ref{example:3-4-2-homology-hair}]\label{example:3-4-2-homology-tail}
Let $(X,p)$ be a cyclic quotient surface singularity $\frac{1}{19}(1,7)$. Using the same labeling as in Example~\ref{example:3-4-2-homology-hair}, the homology classes of $A_j$'s are given by
\begin{equation*}
\begin{aligned}
[A_1] &= l_{\infty}-e_1-e_2, \\
[A_2] &= e_2-e_3-e_4, \\
[A_3] &= e_4-e_5, \\
[A_4] &= e_5-e_6-e_7.
\end{aligned}
\end{equation*}
\end{example}

\subsection{Comparing compactifications and their deformations: Cyclic cases}

We will compare two minimal resolutions, $\widetilde{X}^{JS}$ and $\widetilde{X}^{P}$, and two general fibers, $\overline{X}^{JS}_t$ and $\overline{X}^{P}_t$.

In the previous Section~\ref{section:topology-CQSS-JS} and Section~\ref{section:topology-CQSS-P}, we observed that $\widetilde{X}^{JS}$ and $\widetilde{X}^{P}$ are obtained from $\mathbb{CP}^2$ by a sequence of blowing-ups in a such a way that each blow-up can be labeled by the dots in the extended Riemenschneider dot diagram. Let $e_{i,j}$'s denote the homology classes of the $(-1)$-curves corresponding to the dots in the extended Riemenschneider dot diagram. Then we have $H_2(\widetilde{X}^{JS};\mathbb{Z}) = \langle l_{\infty}, e_{i,j} \rangle$ and $H_2(\widetilde{X}^{P};\mathbb{Z}) = \langle l_{\infty}, e_{i,j} \rangle$.

Let $(C=\cup_{j=1}^{s} C_j,l)$ be the compactified decorated curve of $\overline{X}^{JS}$ and let $\widetilde{C}_j$ be the proper transform of $C_j$ in $\widetilde{X}^{JS}$. Let $A_1,\dotsc,A_s$ be the curves on $\widetilde{X}^{P}$ whose dual graph is given by Figure~\ref{figure:dual-graph-X^{P}}(A). Notice that the number of $C_j$'s and $A_j$'s is the same.

\begin{proposition}\label{proposition:diffeomorphism-center-CQSS}
There is a diffeomorphism $\phi \colon \widetilde{X}^{JS} \to \widetilde{X}^{P}$ such that $\phi(L_{\infty})=L_{\infty}$ and $\phi(E_i)=E_i$ for all $i=1,\dotsc,r$. Let $\phi_{\ast} \colon H_2(\widetilde{X}^{JS};\mathbb{Z}) \to H_2(\widetilde{X}^{P};\mathbb{Z})$ is the induced isomorphism. Then $\phi_{\ast}(e_{i,j})=e_{i,j}$ for all $i,j$. Furthermore, after relabeling $C_j$'s properly, we have
\begin{equation}\label{equation:Dj=A1+...+Aj}
\phi_{\ast}[\widetilde{C}_j] = n_j l_{\infty} + [A_1] + \dotsb + [A_j]
\end{equation}
for some non-negative integer $n_j$.
\end{proposition}

\begin{proof}
Suppose that $\widetilde{X}^{JS}$ and $\widetilde{X}^{P}$ are constructed from $\mathbb{CP}^2$ by the method described in Section~\ref{section:topology-CQSS-JS} and Section~\ref{section:topology-CQSS-P}, respectively. They are obtained from $\mathbb{CP}^2$ under the same number of blowing-ups. In particular, the homology classes of $E_i$ ($i=1,\dotsc,r$) are the same for both $H_2(\widetilde{X}^{JS};\mathbb{Z})$ and $H_2(\widetilde{X}^{P};\mathbb{Z})$. Therefore the identity map $id \colon \mathbb{CP}^2 \to \mathbb{CP}^2$ can be extended to a diffeomorphism $\phi \colon \widetilde{X}^{JS} \to \widetilde{X}^{P}$ such that $\phi(L_{\infty})=L_{\infty}$ and $\phi(E_i)=E_i$ for all $i=1,\dotsc,r$. Furthermore, by Lemma~\ref{lemma:homology-C-CQSS-JS} and Lemma~\ref{lemma:homology-A-CQSS-P}, if we relabeling $C_j$'s properly, then we have $\phi_{\ast}[\widetilde{C}_j] = n_j l_{\infty} + [A_1] + \dotsb + [A_j]$ as desired.
\end{proof}

\begin{example}[Continued from Examples~\ref{example:3-4-2-homology-hair} and \ref{example:3-4-2-homology-tail}]
One can check that
\begin{equation*}
\begin{aligned}
\phi_{\ast}[\widetilde{C}_1] &= l_{\infty} + [A_1], \\
\phi_{\ast}[\widetilde{C}_2] &= l_{\infty} + [A_1]+[A_2], \\
\phi_{\ast}[\widetilde{C}_3] &= l_{\infty} + [A_1]+[A_2]+[A_3], \\
\phi_{\ast}[\widetilde{C}_4] &= l_{\infty} + [A_1]+[A_2]+[A_3]+[A_4]
\end{aligned}
\end{equation*}
\end{example}

Let $\mathcal{X} \to \Delta$ be a smoothing of $X$. Let $\overline{\mathcal{X}}^{JS} \to \Delta$ and $\overline{\mathcal{X}}^{P} \to \Delta$ be the extensions of $\mathcal{X} \to \Delta$ to $\overline{X}^{JS}$ and $\overline{X}^{P}$, respectively, which are locally trivial outside the singular point $X$.

\begin{theorem}\label{theorem:diffeomorphism-fiberwise-CQSS}
There is a diffeomorphism $\phi_t \colon \overline{X}^{JS}_t \to \overline{X}^{P}_t$ between general fibers over $t \neq 0$ such that $\phi_t(L_{\infty})=L_{\infty}$. Furthermore, we still have
\begin{equation}\label{equation:Dj=A1+...+Aj-fiber}
(\phi_t)_{\ast}[\widetilde{C}_j] = n_j l_{\infty} + [A_1] + \dotsb + [A_j]
\end{equation}
\end{theorem}

\begin{proof}
By Proposition~\ref{proposition:diffeomorphism-center-CQSS}, there is a diffeomorphism $\phi \colon \overline{X}^{JS} \to \overline{X}^{P}$ such that $\phi(L_{\infty})=L_{\infty}$ and $\phi(E_i)=E_i$ for all $i$. Notice that $\overline{X}^{JS}_t$ (or $\overline{X}^{P}_t$) can be obtained as a differentiable 4-manifold from $\overline{X}^{JS}$ (or $\overline{X}^{P}$) by replacing a tubular neighborhood of $\cup_{i=1}^{r} E_i$ with the Milnor fiber of the smoothing $\mathcal{X} \to \Delta$ along the link of the singularity $X$. Therefore the diffeomorphism $\phi \colon \overline{X}^{JS} \to \overline{X}^{P}$ can be extended to a diffeomorphism $\phi_t \colon \overline{X}^{JS}_t \to \overline{X}^{P}_t$ as desired.
\end{proof}

\subsection{Topology of compactifications: Weighted cases}

Let $X$ be a weighted homogeneous surface singularity with the big node. We would like compare three compactifications: $\widetilde{X}^{JS}$, $\widetilde{X}^{PS}$, and $\widetilde{X}^{P}$.

Notice that we construct the $i$th branch of $\widetilde{X}^{JS}$ and the $i$th linear chain of $\widetilde{X}^{PS}$ in the exactly same way as in the case of cyclic quotient surface singularities. On the other hand, according to Remark~\ref{remark:decorated-curve-over-E0}, we may choose the compactified decorated curves of $\widetilde{X}^{JS}$ over the central curve $E_0$ by the images of the lines $M_j$, which are just the curves $\widetilde{C}_j$ in $\widetilde{X}^{PS}$ given in Figure~\ref{figure:dual-graph-X^{P}S}.

\begin{proposition}\label{proposition:diffeomorphism-center-WHSS}.
There is a diffeomorphism $\phi \colon \widetilde{X}^{JS} \to \widetilde{X}^{PS}$ such that $\phi(E_0)=E_0$, $\phi(E_{\infty})=E_{\infty}$, $\phi(B_{ij})=B_{ij}$, and $\phi(\widetilde{C}_j) = \widetilde{C}_j$ for all $i,j$. Furthermore, if $\widetilde{C}_{ij}$ denote the compactified decorated curves in $\widetilde{X}^{JS}$, then, after relabeling properly, there are positive integers $n_{ij}$ so that we have
\begin{equation*}
\phi_{\ast}([\widetilde{C}_{ij}]) = n_{ij} l_{\infty} + [A_{i1}] + \dotsb + [A_{ij}]
\end{equation*}
for all $i,j$.
\end{proposition}

\begin{proof}
This follows from Proposition~\ref{proposition:diffeomorphism-center-CQSS}.
\end{proof}

Let $\mathcal{X} \to \Delta$ be a smoothing of $X$. Let $\overline{\mathcal{X}}^{JS} \to \Delta$ and $\overline{\mathcal{X}}^{PS} \to \Delta$ be the compatible extension of $\mathcal{X} \to \Delta$ given in Theorem~\ref{theorem:no-local-to-global}.

\begin{theorem}\label{theorem:diffeomorphism-fiberwise-WHSS}
There is a diffeomorphism $\phi_t \colon \overline{X}^{JS}_t \to \overline{X}^{PS}_t$ between general fibers $\overline{X}^{JS}_t$ and $\overline{X}^{PS}_t$ such that $\phi_t(\widetilde{C}_j)=\widetilde{C}_j$ for all $j$ and
\begin{equation*}
(\phi_t)_{\ast}([\widetilde{C}_{ij}]) = n_{ij} l_{\infty} + [A_{i1}] + \dotsb + [A_{ij}]
\end{equation*}
for all $i,j$.
\end{theorem}

\begin{theorem}\label{theorem:diffeomorphism-fiberwise-WHSS}
There is a blowing down $\overline{\mathcal{X}}^{PS}_t \to \overline{X}^{P}_t$ between general fibers.
\end{theorem}

\begin{proof}
According to Remark~\ref{remark:X^{P}S-from-X^{P}}, $\widetilde{X}^{PS}$ can be obtained from $\widetilde{X}^{P}$ by blowing up along $E_{\infty}$, where the exceptional curves are exactly $\widetilde{C}_j$'s. Hence the assertion follows.
\end{proof}

In summary,

\begin{corollary}\label{corollary:JS->PS->P}
We have maps between general fibers
\begin{equation*}
\overline{X}^{JS}_t \xrightarrow{\phi} \overline{X}^{PS}_t \xrightarrow{\psi} \overline{X}^{P}_t
\end{equation*}
where $\phi$ is a diffeomorphism and $\psi$ is a blowing down.
\end{corollary}

\section{Smoothings of negative weights by Pinkham}
\label{section:smoothings-of-negative-weights}

Let $X$ be a cyclic quotient surface singularity or a weighted homogeneous surface singularity with the big node. We summarize the theory of smoothings of negative weights by Pinkham~\cite{Pinkham-1978} according to the singularities that we are interested in. And we introduce matrices that encode combinatorial descriptions of $\mathcal{C}(X)$.

\subsection{Smoothings of negative weights}

Let $\mathcal{X} \to \Delta$ be a smoothing of $X$ and let $\overline{\mathcal{X}}^{P} \to \Delta$ be its extension given in Theorem~\ref{theorem:extension-of-deformation}. The extension $\overline{\mathcal{X}}^{P} \to \Delta$ is a \emph{smoothing of negative weight} of $\overline{X}^{P}$; that is it is a smoothing obtained by adding terms of lower weight to the defining equations of $\overline{X}^{P}$. See Pinkham~\cite{Pinkham-1978} for details. In particular, the compactifying divisor $\widetilde{E}_{\infty}^{P}$ (cf.~Definition~\ref{definition:P-compactification}) of $X$ is embedded in a general fiber $\overline{X}^{P}_t$ of the extension $\overline{\mathcal{X}}^{P} \to \Delta$ and it supports an ample divisor on $\overline{X}^{P}_t$.

Conversely, under certain cohomological conditions, Pinkham~\cite[Theorem~6.7]{Pinkham-1978} presented a way to construct a negative weight smoothing of $\overline{X}^{P}$ from a curve $E$ in a smooth surface $S$ whose dual graph is the same as that of $\widetilde{E}_{\infty}^{P}$. Here we present a special case of Pinkham's result.

\begin{proposition}\label{proposition:negative-weight-smoothing}
Let $S$ be a smooth rational surface. Suppose that there is a reduced curve $E$ in $S$ that supports an ample divisor, whose dual graph is the same as that of $\widetilde{E}_{\infty}^{P}$. Then there is a unique one-parameter smoothing $\overline{\mathcal{X}}^{P} \to \Delta$ of $\overline{X}^{P}$ such that the interior of its Milnor fiber is diffeomorphic to $S \setminus E$.
\end{proposition}

By ``unique'' in the proposition, we mean that it is unique up to isomorphism of deformations.

\begin{proof}
The cohomological conditions given in Pinkham~\cite[Theorem~6.7]{Pinkham-1978} are automatically satisfied if $S$ is a smooth rational surface. For checking this fact, one may refer the proof of Pinkham~\cite[Theorem~6.10]{Pinkham-1978} or the proof of Stipsicz--Szabó--Wahl~\cite[Theorem~8.1]{SSW-2008}. Therefore there is a one-parameter smoothing $\overline{\mathcal{X}}^{P} \to \Delta$ of $\overline{X}^{P}$ such that $Z \cong \overline{X}^{P}_t$ and $E \cong \widetilde{E}^{P}$.

On the other hand, Fowler~\cite[Theorem 2.2.3.]{Fowler-2013} gave a proof of the uniqueness using Looijenga~\cite[A.6~Proposition]{Looijenga-1984}. We briefly recall the proof in Fowler~\cite[Theorem 2.2.3.]{Fowler-2013}. There exists a formal versal $\mathbb{C}^{\ast}$-equivariant  deformation $p \colon \mathcal{V} \to B$ of $X$; Pinkham~\cite[2.2~Theorem]{Pinkham-1978}. Let $p^{-} \colon \mathcal{V}^{-} \to B^{-}$ be the part of negative weight of the versal deformation of $X$. Then Pinkham~\cite[2.9~Theorem]{Pinkham-1978} showed that Then $p^{-} \colon \mathcal{V}^{-} \to B^{-}$ extends to a proper fiat family $\overline{p}^{-} \colon \overline{\mathcal{V}}^{-} \to B^{-}$ of deformaions of $\overline{X}^{P}$. Looijenga~\cite[A.6~Proposition]{Looijenga-1984} interpreted the family $\overline{p}^{-} \colon \overline{\mathcal{V}}^{-} \to B^{-}$ as a fine moduli space consisting of the triples $(Z; F; \phi)$, where (i) $Z$ is a projective surface with cyclic quotient singularities along a curve $F \subset Z$, (ii) $F$ is an ample Weil divisor on $Y$ which is isomorphic to the singular compactifying divisor $\overline{E}_{\infty}$, and (iii) if $Y=\Proj(R)$ and $X=\Spec(A)$ then $\phi \colon R/(t) \cong A$ for an element $t$ representing $1$. By contracting the branches on $E_{\infty}$ on $\overline{X}^{P}_t$ to cyclic quotient singularities, we obtain the Looijenga's triple from $(Z,E)$. So $(Z,E)$ determines a unique negative weight smoothings of $X$.
\end{proof}

\subsection{Homology matrices}

A general fiber $\overline{X}^P_t$ ($t \neq 0$) is a smooth rational surface because it contains a smooth rational curve $E_{\infty}$ with positive self-intersection number.

\begin{lemma}\label{lemma:blowndown-to-lines-cyclic}
Let $X$ be a cyclic quotient surface singularity. Then there is a sequence of blow-downs from a general fiber $\overline{X}_t^{P}$ ($t \neq 0$) to $\mathbb{CP}^2$ such that (i) no blowing down occur on $E_{\infty}$, and (ii) the image of the compactifying divisor $\widetilde{E}_{\infty}^{P}$ under the sequence of blow-downs consists of two lines in $\mathbb{CP}^2$: One from $E_{\infty}$ and the other from $A_1$.
\end{lemma}

\begin{proof}
By Theorem~\ref{theorem:diffeomorphism-fiberwise-CQSS}, there is a diffeomorphism $\phi \colon \overline{X}^{JS}_t \to \overline{X}^{P}_t$ between general fibers such that $\phi(E_{\infty})=E_{\infty}$ and
\begin{equation*}
\phi_{\ast}([\widetilde{C}_j]) = n_j l_{\infty} + [A_1] + \dotsb + [A_j]
\end{equation*}
A general fiber $\overline{X}^{JS}_t$ is obtained by blowing up several times $\mathbb{CP}^2$ at some points on the decorated curves in a general fiber of the compactified picture deformations $(\mathcal{C}, \mathcal{L})$; Theorem~\ref{theorem:stablility}. So each $(-1)$-curves in $\overline{X}^{JS}_t$ intersect only $\widetilde{C}_j$'s. Therefore each $(-1)$-curves in $\overline{X}^{P}_t$ intersect only $A_j$'s, not $E_{\infty}$. Hence there is a sequence of blow-downs $\psi \colon \overline{X}^{P}_t \to \mathbb{CP}^2$ such that no blowing down occur on $E_{\infty}$. Furthermore $\psi(A_1)$ must be a line because it intersects with $E_{\infty}$ at one point and the other $A_j$'s ($j \ge 2$) must be contracted to the same point on $\psi(A_1)$ because they do not intersect with $E_{\infty}$.
\end{proof}

We have a similar result for weighted cases.

\begin{lemma}\label{lemma:blowndown-to-lines-weighted}
Let $X$ be a weighted homogeneous surface singularity with the big node. There is a sequence of blow-downs from a general fiber $\overline{X}_t^{PS}$ ($t \neq 0$) to $\mathbb{CP}^2$ such that (i) no blowing down occur on $E_{\infty}$, and (ii) the image of the compactifying divisor $\widetilde{E}_{\infty}^{PS}$ under the sequence of blow-downs consists of lines in $\mathbb{CP}^2$.
\end{lemma}

\begin{proof}
By Theorem~\ref{theorem:diffeomorphism-fiberwise-WHSS}, there is a diffeomorphism $\phi \colon \overline{X}^{JS}_t \to \overline{X}^{PS}_t$ such that $\phi(\widetilde{C}_k)=\widetilde{C}_k$ for all $k$ and
\begin{equation*}
(\phi_t)_{\ast}([\widetilde{C}_{ij}]) = n_{ij} l_{\infty} + [A_{i1}] + \dotsb + [A_{ij}]
\end{equation*}
for all $i,j$. Notice that $\overline{X}^{JS}_t$ is obtained by blowing up $\mathbb{CP}^2$ at some points on the images of the decorated curves $\widetilde{C_{ij}}$ and $\widetilde{C}_k$ in a general fiber of the corresponding picture deformation to $\overline{\mathcal{X}}^{JS} \to \Delta$. Therefore there is also a sequence of blow-downs from $\overline{X}^{PS}_t$ to $\mathbb{CP}^2$.

On the other hand, any $(-1)$-curves in $\overline{X}^{JS}_t$ intersect  only $\widetilde{C}_{ij}$ and $\widetilde{C}_k$ and they do not intersect the image of the line at infinity $L_{\infty}$, which is just $E_{\infty}$ in $\overline{X}^{JS}_t$. So any $(-1)$-curves in $\overline{X}^{PS}_t$ intersect only $A_{ij}$'s and $\widetilde{C}_k$'s. Therefore the image of $E_{\infty}$ in $\mathbb{CP}^2$ is a line. Furthermore, each $A_{i1}$ ($i=1,\dotsc,t$) and each $\widetilde{C}_k$ ($k=1,\dotsc,d-t-1$) intersect with $E_{\infty}$ at one point, respectively. So their images in $\mathbb{CP}^2$ must be lines.
\end{proof}

We define a matrix that encodes combinatorial and homological properties of the ample embedding of the compactifying divisor $\widetilde{E}_{\infty}$ into a general fiber $\overline{X}^P_t$:

\begin{definition}\label{definition:homology-matrix}
The \emph{homology matrix} of the smoothing $\overline{\mathcal{X}}^P \to \Delta$ of $\overline{X}^P$ is the matrix that have one row for every irreducible component $A_i$'s (for the cyclic case) or $A_{ij}$'s (for the weighted case) of the compactifying divisor $\widetilde{E}_{\infty}$ and one column for every $(-1)$-curve in $\overline{X}^P_t$ such that the $(i,j)$-entry is defined by the intersection number of the component of $\widetilde{E}_{\infty}$ corresponding to the $i$th row and the $(-1)$-curves corresponding to $j$th column. We denote by $\mathcal{H}(X)$ by the set of all homology matrices of $X$.
\end{definition}

Every component of $\Def(X)$ is a smoothing component and every smoothing of $X$ can be extended to that of $\overline{X}^P$. So we have a well-defined map:

\begin{definition}
The \emph{homology matrix map} of $X$ is a map
\begin{equation*}
\phi_H \colon \mathcal{C}(X) \to \mathcal{H}(X)
\end{equation*}
defined by assigning the corresponding homology matrix in $\mathcal{H}(X)$ to each irreducible component in $\mathcal{C}(X)$.
\end{definition}

This map may be another combinatorial description of $\mathcal{C}(X)$.

\begin{problem}\label{problem:phi_H}
Determine when $\phi_H$ is injective.
\end{problem}

We prove that $\phi_H$ is injective for certain weighted homogeneous surface singularities in Theorem~\ref{theorem:phi_H-injective-WHSS}.

\subsection{Isotopies}\label{section:isotopy}

We prove that $\phi_H$ is injective for cyclic quotient surface singularities and weighted homogeneous surface singularities with the `enough' big nodes.

\begin{theorem}\label{theorem:phi_H-injective-cyclic}
For a cyclic quotient surface singularity $X$, the map $\phi_H \colon \mathcal{C}(X) \to \mathcal{H}(X)$ is injective; hence, it is bijective.
\end{theorem}

\begin{proof}
Suppose that there are two smoothings $\mathcal{X}_1 \to \Delta$ and $\mathcal{X}_2 \to \Delta$ of $X$ corresponding to the same homology matrix. According to Proposition~\ref{proposition:negative-weight-smoothing}, each smoothing $\mathcal{X}_i \to \Delta$ is completely determined by the pair $(S_i, \widetilde{E}_{\infty})$, where $S_i=\overline{X_i}^P_t$ is a smooth rational surface where the divisor $\widetilde{E}_{\infty}$ is amply embedded. Notice that there is an isotopy between $(\mathbb{CP}^2, L_1 \cup M_1)$ and $(\mathbb{CP}^2, L_2 \cup M_2)$. Since $\mathcal{X}_i$'s have the same homology matrix, each pairs $(S_i, \widetilde{E}_{\infty})$ are constructed from the pairs $(\mathbb{CP}^2, L_i \cup M_i)$ via the same sequence of blowing ups. Therefore there is also an isotopy between $(S_1, \widetilde{E}_{\infty})$ and $(S_2, \widetilde{E}_{\infty})$. Therefore there is a path in $\Def(X)$ which connects two points corresponding to each smoothings $\mathcal{X}_1$ and $\mathcal{X}_2$. Therefore they lie on the same component in $\Def(X)$.
\end{proof}

We now assume that $X$ is a weighted homogeneous surface singularity with the big node.

\begin{definition}\label{definition:line-arrangement}
Let $\mathcal{X} \to \Delta$ be a smoothing of $X$. The \emph{line arrangement} of $\mathcal{X} \to \Delta$ is the set of lines in $\mathbb{CP}^2$ that is obtained from $\overline{X}^{PS}_t$ by blowing down as in the above Lemma~\ref{lemma:blowndown-to-lines-weighted}
\end{definition}

\begin{proposition}
If $X$ is a weighted homogeneous surface singularity with $d \ge t+2$, then the line arrangement of each smoothing $\mathcal{X} \to \Delta$ is one of the following arrangements given in Figure~\ref{figure:line-arrangement}. Explicitly, let $L_{\infty}$ be the line that is the image of $E_{\infty}$. Then there is one point outside $L_{\infty}$ such that (A) all the other lines passe through the point or such that (B) all the other lines except only one line pass through the point, or (C) there are two points outside $L_{\infty}$ such that the rest of the lines must pass through one of the two points and at least three lines pass through each point.
\end{proposition}

\begin{figure}
\subfloat[]{
\begin{tikzpicture}[scale=0.75]
\draw[-] (0,0)--(3,0);

\draw[-] (0,0.5)--(2,-2);
\draw[-] (0.5,0.5)--(1.833,-2);
\draw[-] (1,0.5)--(1.666,-2);
\draw[-] (1.5,0.5)--(1.5,-2);
\draw[-] (3,0.5)--(1,-2);

\draw[dotted] (2,0.25)--(2.5,0.25);
\end{tikzpicture}} \qquad
\subfloat[]{
\begin{tikzpicture}[scale=0.75]
\draw[-] (0,0)--(3,0);

\draw[-] (0,0.5)--(3,-2);
\draw[-] (0.5,0.5)--(1.833,-2);
\draw[-] (1,0.5)--(1.666,-2);
\draw[-] (1.5,0.5)--(1.5,-2);
\draw[-] (3,0.5)--(1,-2);

\draw[dotted] (2,0.25)--(2.5,0.25);
\end{tikzpicture}
} \qquad
\subfloat[]{
\begin{tikzpicture}[scale=0.75]
\draw[-] (0,0)--(3,0);

\draw[-] (0.25,0.5)--(1,-2);
\draw[-] (0.5,0.5)--(0.875,-2);
\draw[-] (1.25,0.5)--(0.5,-2);

\draw[-] (1.75,0.5)--(2.5,-2);
\draw[-] (2,0.5)--(2.375,-2);
\draw[-] (2.75,0.5)--(2,-2);

\draw[dotted] (0.75,0.25)--(1,0.25);
\draw[dotted] (2.25,0.25)--(2.5,0.25);
\end{tikzpicture}
}
\caption{Line arrangements}
\label{figure:line-arrangement}
\end{figure}
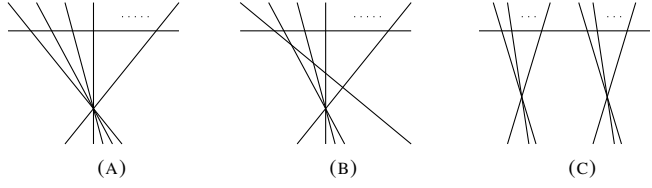

\begin{proof}
Let $L_{\infty}$ be the image of $E_{\infty}$. Let $M_j$ be the line that is the image of $\widetilde{C}_j$ for each $j=1,\dotsc,d-t-1$. Notice that the image of $M_1$ in $\overline{X}^{PS}_t$ is a $(-1)$-curve. So there are two points on $M_1$ outside $L_{\infty}$ where blowing-ups occur. So if $d=t+2$, then we have three possibilities (A), (B), (C). On the other hand, if $d \ge t+3$, then only (A) and (B) can happen.
\end{proof}

\begin{theorem}\label{theorem:phi_H-injective-WHSS}
Let $X$ be a weighted homogeneous surface singularity with $d \ge t+2$. The map $\phi_H \colon \mathcal{C}(X) \to \mathcal{H}(X)$ is injective; hence, it is bijective.
\end{theorem}

\begin{proof}
A smoothing $\mathcal{X} \to \Delta$ is determined by the pair $(S, \widetilde{E}_{\infty})$ of a smooth rational surface $S = \overline{X}^P_t$ and an amply embedded divisor $\widetilde{E}_{\infty}$ by Proposition~\ref{proposition:negative-weight-smoothing}. By Proposition~\ref{proposition:diffeomorphism-center-WHSS}, we blow up ($d-t-1$) points on $E_{\infty}$ properly so that we get the pair $(S', \widetilde{E}_{\infty}^{PS})$ where $S' = \overline{X}^{PS}_t$ and $\widetilde{E}_{\infty}^{PS}$ is the compactifying divisor for $\overline{X}^{PS}$ defined in Definition~\ref{definition:PS-compactification}.

Suppose that there are two smoothings $\mathcal{X}_1 \to \Delta$ and $\mathcal{X}_2 \to \Delta$ of $X$ such that their homology matrices are the same. Then two pairs $(S_1', \widetilde{E}_{\infty}^{PS})$ and $(S_2', \widetilde{E}_{\infty}^{PS})$ are isotopic to each other. This is because they are obtained from $\mathbb{CP}^2$ with the line arrangement of the same type (A), (B), or (C) (given in Figure~\ref{figure:line-arrangement}) via the same sequence of blowing-ups for they have the same homology matrices. But the line arrangements of the same type in Figure~\ref{figure:line-arrangement} are isotopic to each other by Starkston~\cite[Lemma~2.8]{Starkston-2015}. Therefore there is a isotopy between $\mathcal{X}_1 \to \Delta$ and $\mathcal{X}_2 \to \Delta$, which implies that they lie on the same component of $\Def(X)$.
\end{proof}

\section{From picture deformations to smoothings of negative weight}
\label{section:Pic-Def-to-SNW}

Let $X$ be a cyclic quotient surface singularity or a weighted homogeneous surface singularity with the big node with its usual sandwiched structure. We construct a map from $\mathcal{I}(X)$ to $\mathcal{H}(X)$ that commutes with $\phi_I \colon \mathcal{C}(X) \to \mathcal{I}(X)$ and $\phi_H \colon \mathcal{C}(X) \to \mathcal{H}(X)$; Theorem~\ref{theorem:phi_IH}.

\begin{definition}
Let $M$ be a matrix with the rows $M(1), \dotsc, M(e)$.

\begin{enumerate}
\item The \emph{difference matrix} $\Delta M$ is a matrix with $e$ rows $\Delta M(1), \dotsc, \Delta M(e)$ defined by
\begin{equation*}
\text{$\Delta M(1) = M(1)$ and $\Delta M(i) = M(i)-M(i-1)$ for $i \ge 2$}.
\end{equation*}

\item The \emph{positive matrix} $M^+$ is the matrix obtained from $M$ by deleting columns that contain entries smaller than zero.
\end{enumerate}
\end{definition}

\begin{theorem}\label{theorem:M-to-Delta^+M-cyclic}
Let $X$ be a cyclic quotient surface singularity and let $\mathcal{X} \to \Delta$ be a smoothing of $X$. If $M$ is the incidence matrix of $\mathcal{X} \to \Delta$, then the positive difference matrix
\begin{equation}\label{equation:Delta^+M-cyclic}
(\Delta M)^+
\end{equation}
is the homology matrix of the smoothing.
\end{theorem}

\begin{proof}
Any projective $(-1)$-curves in $\overline{X}^{JS}_t$ must intersect with one of $\widetilde{C}_i$'s, but not with $E_{\infty}$. On the other hand, according to Equation~\ref{equation:Dj=A1+...+Aj}, we have the following relations on their homologies:
\begin{align}\label{equation:AvsC-cyclic}
[A_1] &= m_1 l_{\infty} + \phi_{\ast}([\widetilde{C}_1]) \\
[A_i] &= m_i l_{\infty} + \phi_{\ast}([\widetilde{C}_i]) - \phi_{\ast}([\widetilde{C}_{i-1}])
\end{align}
where $i \ge 2$ and $m_i \in \mathbb{Z}$. If $E$ is a projective $(-1)$-curve in $\overline{X}^{P}_t$, then its inverse image $\phi^{-1}(E)$ is also a (differentiable) $(-1)$-curve on $\overline{X}^{JS}_t$. So its homology is equal to one of the homologies of the projective $(-1)$-curves in  $\overline{X}^{JS}_t$. Hence, the intersection of $\phi^{-1}(E)$ with $\widetilde{C}_i$'s must be equal to one of the column of the incidence matrix $M$. Since any projective curves intersect positively with another projective curve, the intersection of $E$ with $A_i$'s is equal to one of the columns of the positive difference matrix $(\Delta M)^+$ by Equation~\eqref{equation:AvsC-cyclic}.
\end{proof}

Let $X$ be a weighted homogeneous surface singularity with the big node. The procedure for obtaining homology matrices of $X$ from incidence matrices of $X$ is essentially the same for cyclic cases. Let $M$ be the incidence matrix for a smoothing $\mathcal{X} \to \Delta$ of $X$. For each nonzero entries in the rows corresponding to the decorated curves $\widetilde{C}_1, \dotsc, \widetilde{C}_{d-t-1}$, we delete the columns of $M$ that contain the nonzero entries. After that, we delete all rows corresponding to $\widetilde{C}_k$'s. We denote the submatrix by $M \setminus C$. For each $i=1,\dotsc,t$, we denote by $(M \setminus C)(i)$ the submatrix of $M \setminus C$ consisting of the rows of $M \setminus C$ that corresponds to the decorated curves on the $i$th branch.

\begin{theorem}\label{theorem:M-to-Delta^+M-weighted}
Let $X$ be a weighted homogeneous surface singularity with the big node. For the incidence matrix $M$ of a smoothing $\mathcal{X} \to \Delta$ of $X$, the corresponding homology matrix to $\mathcal{X} \to \Delta$ of $X$ is given by
\begin{equation}\label{equation:Delta^+M-weighted}
\begin{bmatrix}
\Delta (M \setminus C)(1) \\
\vdots\\
\Delta (M \setminus C)(t)
\end{bmatrix}^+
\end{equation}
\end{theorem}

\begin{proof}
By Corollary~\ref{corollary:JS->PS->P}, $\overline{X}^{JS}_t$ is diffeomorphic to $\overline{X}^{PS}_t$. Hence the homology matrix for $\overline{X}^{PS}_t$ is given by
\begin{equation*}
\begin{bmatrix}
\Delta M(1) \\
\vdots\\
\Delta M(t)
\end{bmatrix}^+
\end{equation*}
by Theorem~\ref{theorem:M-to-Delta^+M-cyclic}. Since $\overline{X}^{PS}_t$ is blown down to $\overline{X}^{P}_t$ along the $(-1)$-curves $\widetilde{C}_k$ for $k=1,\dotsc,d-t-1$, the images of any $(-1)$-curves intersecting $\widetilde{C}_k$'s are not $(-1)$-curves anymore. Therefore, by deleting the columns corresponding to the $(-1)$-curves intersecting any $\widetilde{C}_k$'s and the rows of all $\widetilde{C}_k$'s, we obtain the homology matrix for $\overline{X}^{P}_t$.
\end{proof}

\begin{definition}
Let $X$ be a cyclic quotient surface singularity or a weighted homogeneous surface singularity with the big node. Then for any incidence matrix $M$ of $X$, we define the \emph{homology matrix corresponding to $M$} (denoted by $\Delta^+ M$) by the corresponding homology matrix defined in Equation~\eqref{equation:Delta^+M-cyclic} for cyclic cases or Equation~\eqref{equation:Delta^+M-weighted} for weighted cases.
\end{definition}

\begin{theorem}\label{theorem:phi_IH}
Let $X$ be a cyclic quotient surface singularity or a weighted homogeneous surface singularity with the big node. Then we have a map
\begin{equation*}
\phi_{IH} \colon \mathcal{I}(X) \to \mathcal{H}(X)
\end{equation*}
defined by $\phi_{IH}(M) = \Delta^+M$ so that $\phi_H = \phi_{IH} \circ \phi_{I}$.
\end{theorem}

\begin{theorem}\label{theorem:phi_I-injective}
If $X$ is a cyclic quotient surface singularity or a weighted homogeneous surface singularity with $d \ge t+2$, then the map $\phi_I \colon \mathcal{C}(X) \to \mathcal{I}(X)$ is injective; hence, it is bijective.
\end{theorem}

\begin{proof}
This follows from Theorem~\ref{theorem:phi_H-injective-cyclic} and Theorem~\ref{theorem:phi_H-injective-WHSS}.
\end{proof}

As an easy application of Némethi--Popescu-Pampu~\cite{NPP-2010-PLMS}:

\begin{corollary}\label{corollary:not-strongly-diffeomorphic}
If $X$ is a cyclic quotient surface singularity or a weighted homogeneous surface singularity with $d \ge t+2$, then Minor fibers corresponding to different components of the reduced miniversal deformation space of $X$ are not `strongly' diffeomorphic to each other.
\end{corollary}

Here we follow the definition of `strong diffeomorphism' presented in Némethi and Popescu-Pampu~\cite{NPP-2010-IMRN}: If $M_1$ and $M_2$ are two oriented manifolds with boundary, endowed with a fixed isotopy class of orientation-preserving diffeomorphisms $\partial M_1 \to \partial M_2$, then $M_1$ and $M_2$ are \emph{strongly diffeomorphic} (with respect to that class) if there exists an orientation-preserving diffeomorphism $M_1 \to M_2$ whose restriction to the boundary belongs to the given class.

\begin{proof}[Proof of Corollary~\ref{corollary:not-strongly-diffeomorphic}]
Némethi--Popescu-Pampu~\cite[Theorem~4.3.4]{NPP-2010-PLMS} showed if the incidence matrices of two picture deformations are different up to permutation of columns, then their associated Milnor fibers are not strongly diffeomorphic. Since $\phi_I$ is injective, we can conclude that there is no pair of smoothing components such that their corresponding Milnor fibers are strongly diffeomorphic to each other.
\end{proof}

\section{From P-resolutions to smoothings of negative weights}\label{section:P-resolution->smoothings-of-negative-weights}

Let $X$ be a cyclic quotient surface singularity or a weighted homogeneous surface singularity with the big node. We define a map from $\mathcal{P}(X)$ to $\mathcal{H}(X)$. For cyclic cases, it is done in PPSU~\cite[\S10]{PPSU-2018}. The procedure for weighted cases is identical to that described in PPSU~\cite[\S9]{PPSU-2018}. We briefly recall the method given in PPSU~\cite[\S9]{PPSU-2018}.

Let $\mathcal{X} \to \Delta$ be a smoothing of $X$. Suppose that there is a P-resolution $Y \to X$ of $X$ that corresponds to $\mathcal{X} \to \Delta$. Since the minimal resolution of $Y$ dominates the minimal resolution of $X$, we have a compactification $\overline{Y}^P$ of $Y$ that dominates $\overline{X}^P$.

\begin{proposition}\label{proposition:no-obstruction-PtoH}
There is no local-to-global obstruction of $\overline{Y}^P$. In particular, every smoothing of $X$ can be lifted to that of $\overline{Y}^P$ that is locally trivial along $\overline{Y}^P \setminus X$.
\end{proposition}

\begin{proof}
The proof is verbatim to that of Theorem~\ref{theorem:no-local-to-global}.
\end{proof}

In particular, we have a map $\overline{\mathcal{Y}}^P \to \overline{\mathcal{X}}^P$ over $\Delta$, where $\overline{\mathcal{Y}}^P$ is an extension of $\mathbb{Q}$-Gorenstein smoothings of T-singularities on $Y$ that is locally trivial outside the singularities.

\begin{theorem}\label{theorem:P->H}
By applying only Iitaka–Kodaira divisorial contractions and usual flips to curves coming from the arms of $\overline{Y}^P$, we can run the semi-stable MMP to $\overline{\mathcal{Y}}^P \to \Delta$ until we obtain a deformation $\mathcal{W} \to \Delta$ whose central fiber $W_0$ is smooth.
\end{theorem}

\begin{proof}
The proof is identical to that of PPSU~\cite[Theorem~9.4]{PPSU-2018}.
\end{proof}

A flip changes only central fibers. But divisorial contractions are just blow-downs of $(-1)$-curves on each general fibers. Hence:

\begin{corollary}
In Proposition~\ref{theorem:P->H}, a general fiber $\overline{Y}^P_t$ ($t \neq 0$) of the smoothing $\overline{\mathcal{Y}}^P \to \Delta$ is obtained by blowing up several times a general fiber $W_t$ of the smoothing $\mathcal{W} \to \Delta$.
\end{corollary}

By comparing the central fiber $W_0$ and a general fiber $W_t$, one can get the data of positions of $(-1)$-curves in $W_t$. Then, one can get the data of intersections of $(-1)$-curves with the compactifying divisor $\widetilde{E}_{\infty}$ in $\overline{Y}^P_t$ by tracking the blow-downs $\overline{Y}^P_t \to W_t$. The data of $(-1)$-curves can be encoded as a homology matrix of $\overline{\mathcal{X}}^P$. So we have a map:

\begin{theorem}\label{theorem:phi_PH}
Let $X$ be a cyclic quotient surface singularity of a weighted homogeneous surface singularity with the big node. Via the procedure discussed in the above, we can define a map
\begin{equation*}
\phi_{PH} \colon \mathcal{P}(X) \to \mathcal{H}(X)
\end{equation*}
such that $\phi_{PH} = \phi_H \circ \phi_P$.
\end{theorem}

\section{Deformations of cyclic quotient surface singularities}\label{section:CQSS}

Let $X$ be a cyclic quotient surface singularity of type $\frac{1}{n}(1,a)$. Christophersen~\cite{Christophersen-1991} and Stevens~\cite{Stevens-1991} described the equations of each components in $\mathcal{C}(X)$. In this section we study the correspondence between their results and other theories of deformations.

\subsection{The equations of the reduced miniversal base space}

We recall briefly Christophersen~\cite{Christophersen-1991} and Stevens~\cite{Stevens-1991}. We first define a set $\mathcal{K}(X)$. For any sequence $\underline{x}=(x_1,\dotsc,x_e)$ of positive integers, we define an $s \times s$ matrix $M(\underline{x})$ by $M_{i,i}=x_i$, $M_{i,j}=-1$ if $\abs{i-j}=1$, and $M_{i,j}=0$ otherwise.

\begin{definition}[Orlik--Wagreich~\cite{Orlik-Wagreich-1977}]
A sequence $\underline{k}=(k_1,\dotsc,k_e) \in \mathbb{N}^e$ is called \emph{admissible} if the matrix $M(\underline{k})$ is positive semi-definite of rank at least $e-1$.
\end{definition}

\begin{definition}{Christophersen~\cite{Christophersen-1991}}
For $e \ge 1$, we define
\begin{equation*}
K_e = \{ (k_1,\dotsc,k_e) \in \mathbb{N}^e \mid \text{$(k_1,\dotsc,k_e)$ is admissible and $[k_1,\dotsc,k_e]=0$}\}.
\end{equation*}
that is, the set of all admissible $e$-tuples which represent a zero Hirzebruch--Jung continued fraction.
\end{definition}

Suppose that $n/(n-a)=[a_1,\dotsc,a_e]$.

\begin{definition}
We define
\begin{equation*}
\mathcal{K}(X) = \{(k_1,\dotsc,k_e) \in K_e \mid \text{$k_i \le a_i$ for all $i$}\}.
\end{equation*}
\end{definition}

As a toric variety, one can find a system of equations that define the singularity $(X,p)$. The system includes the equations
\begin{equation*}
z_{i-1}z_{i+1}-z_i^{a_i}=0
\end{equation*}
for $i=1,\dotsc,e$. Then Christophersen~\cite{Christophersen-1991} and Stevens~\cite{Stevens-1991} showed that the deformation component parameterized by $\underline{k} \in \mathcal{K}(X)$ contains a 1-parameter deformation which has a set of equations that contains
\begin{equation*}
z_{i-1}z_{i+1}-z_i^{a_i}=tz_i^{k_i}
\end{equation*}
for all $i=1,\dotsc,e$.

So we may regard that the set $\mathcal{K}(X)$ parameterizes $\mathcal{C}(X)$ via equations. Therefore we can define a bijective map
\begin{equation*}
\phi_K \colon \mathcal{C}(X) \to \mathcal{K}(X).
\end{equation*}

\subsection{From smoothings of negative weights to the equations}

Let $H$ be a homology matrix of $X$. For each $i$th row (corresponding to $A_i$) of $H$, we define a sequence $\underline{k}$ for $H$ as follows:

\begin{definition}
For each $i$, we define a positive integer $d_i$ by the number of nonzero entries on the $i$th row of $H$. Let $k_i = a_i - d_i$. We denote the sequence $(k_1,\dotsc,k_e)$ by $\underline{k}(H)$.
\end{definition}

\begin{lemma}\label{lemma:H->K}
We have $\underline{k}(H) \in \mathcal{K}(X)$.
\end{lemma}

\begin{proof}
Notice that $d_i$ is equal to the number of $(-1)$-curves in $\overline{X}^P_t$ that intersect with $A_i$. But every $(-1)$-curve intersects only one $A_i$ by Lemma~\ref{lemma:blowndown-to-lines-cyclic}. Furthermore there is a sequence of blow-downs from $\overline{X}^P_t$ to $\mathbb{CP}^2$, where (i) no blow-down occur on $E_{\infty}$, (ii) the image of $A_1$ is a line, and (iii) all the other $A_j$'s are contracted to a point on $A_1$. Therefore $d_i \le a_i$. If we blow down all $(-1)$-curves intersecting $A_i$'s at once, then the images of $A_i$'s becomes a linear chain $k_1-\dotsb-k_e$. But the chain should be contracted to a point on $A_1$. Therefore we have $\underline{k} \in \mathcal{K}(X)$.
\end{proof}

We have a map
\begin{equation*}
\phi_{HK} \colon \mathcal{H}(X) \to \mathcal{K}(X), \quad H \mapsto \underline{k}(H).
\end{equation*}

\subsection{From picture deformations to the equations}

Let $I$ be an incidence matrix of $X$.

\begin{definition}
We define a number $d_i$ by
\begin{equation}\label{equation:di}
d_i = \#\{j \mid \text{$I_{pj}=0$ for $p < i$ and $I_{pj}=1$ for $p \ge i$}\}.
\end{equation}
Let $k_i=a_i-d_i$ for $i=1,\dotsc,e$. We denote the sequence $(k_1,\dotsc,k_e)$ by $\underline{k}(I)$.
\end{definition}

\begin{lemma}
We have $\underline{k}(I) \in \mathcal{K}(X)$.
\end{lemma}

\begin{proof}
This follows immediately from Theorem~\ref{theorem:M-to-Delta^+M-cyclic}, that is, $\phi_{IH}(I)=\Delta^+I$.
\end{proof}

\begin{remark}
The number $d_i$ is defined also in Némethi--Popescu-Pampu~\cite[Proposition~10.1.10]{NPP-2010-PLMS} in a similar way.
\end{remark}

\subsection{From P-resolutions to the equations}

This topic is one of the main results in PPSU~\cite{PPSU-2018}. They showed that one can obtain the corresponding sequence $\underline{k}$ from a given P-resolution via the semi-stable MMP.

The idea can be explained using the terminology in this paper. As in Section~\ref{section:P-resolution->smoothings-of-negative-weights}, for each P-resolution $Y$ of $X$, we can construct the corresponding smoothings $\overline{\mathcal{X}}^P$ of negative weights of $X$. Then the sequence of $\underline{k}$ for the P-resolution $Y$ is just the image of the homology matrix of $\overline{\mathcal{X}}^P$. In short, $\underline{k} = \phi_{HK}(\phi_{PH}(Y))$.

\subsection{Correspondence between four deformation theories}

For every $\underline{k} \in \mathcal{K}(X)$ the picture deformation, the P-resolution, and the smoothing of negative weight corresponding to $\underline{k}$ parameterize the same component in $\mathcal{C}(X)$. But in order to show that the equation associated to $\underline{k}$ also parameterize the same component, we need the result of Némethi--Popescu-Pampu~\cite{NPP-2010-PLMS}.

In the paper, they showed that the picture deformation and the equation corresponding to the same $\underline{k}$ parameterize the same component in $\mathcal{C}(X)$. For this, they applied the classification of minimal symplectic fillings of cyclic quotient surface singularities by Lisca~\cite{Lisca-2008}.

Therefore we have a complete picture of the correspondence between all known deformation theories of cyclic quotient surface singularities, which is given in Figure~\ref{figure:summary-diagram-CQSS}.

\begin{example}
Let $X$ be a cyclic quotient surface singularities $\frac{1}{19}(1,7)$. Then the correspondence are given as follows.
\begin{center}
\begin{tabulary}{\textwidth}{p{10em}p{8em}p{6em}p{6em}}
\toprule
Incidence matrices & Homology matrices & P-resolutions & Equations \\ \midrule
$\begin{bmatrix}
0 & 0 & 0 & 0 & 0 & 1 & 1\\
0 & 0 & 0 & 1 & 1 & 0 & 1 \\
0 & 0 & 1 & 0 & 1 & 0 & 1 \\
1 & 1 & 0 & 0 & 1 & 0 & 1 \\
\end{bmatrix}$
& $\begin{bmatrix}
0 & 0 & 0 & 1 \\
0 & 0 & 1 & 0 \\
0 & 0 & 0 & 0 \\
1 & 1 & 0 & 0
\end{bmatrix}$
& $3-4-[2]$
& $(1,2,2,1)$ \\ \midrule
$\begin{bmatrix}
0 & 0 & 0 & 0 & 1 & 1 \\
0 & 0 & 1 & 1 & 0 & 1 \\
0 & 1 & 0 & 1 & 0 & 1 \\
1 & 1 & 1 & 0 & 0 & 1 \\
\end{bmatrix}$
& $\begin{bmatrix}
0 & 0 & 1 \\
0 & 0 & 0 \\
0 & 1 & 0 \\
1 & 0 & 0
\end{bmatrix}$
& $3-[4]-2$
& $(1,3,1,2)$ \\ \midrule
$\begin{bmatrix}
0 & 0 & 0 & 1 & 1 \\
0 & 1 & 1 & 0 & 1 \\
1 & 0 & 1 & 0 & 1 \\
1 & 1 & 1 & 1 & 0 \\
\end{bmatrix}$
& $\begin{bmatrix}
0 & 0 \\
0 & 1 \\
1 & 0 \\
0 & 0
\end{bmatrix}$
& $[4]-1-[5,2]$
& $(2,2,1,3)$ \\ \bottomrule
\end{tabulary}
\end{center}
\end{example}

\section{A weighted homogeneous surface singularity admitting QHDS}
\label{section:Wpqr}

The singularity with which we are concerned in this section is known as $W(p,q,r)$ (after J. Wahl), whose resolution graph is given by in Figure~\ref{figure:Wpqr}. In this section, we prove that Kollár conjecture holds for $W(p,q,r)$. We follow the strategy described in Introduction.

\begin{figure}
\centering
\begin{tikzpicture}[scale=0.75]
\tikzset{font=\scriptsize}

\node[bullet] (00) at (0,0) [labelAbove={$-4$}] {};


\node[bullet] (-350) at (-3.5,0) [labelBelow={$-(p+3)$}] {};
\node[bullet] (-250) at (-2.5,0) [labelAbove={$-2$}] {};
\node[empty] (-20) at (-2,0) [] {};
\node[empty] (-150) at (-1.5,0) [] {};
\node[bullet] (-10) at (-1,0) [labelAbove={$-2$}] {};

\draw[-] (-350)--(-250);
\draw[-] (-250)--(-20);
\draw[dotted] (-20)--(-150);
\draw[-] (-150)--(-10);
\draw[-] (-10)--(00);
\draw [decorate, decoration = {calligraphic brace,mirror}] (-2.5,-0.15)--(-1,-0.15)  node[pos=0.5,below=0.1em,black]{$q$};


\node[bullet] (350) at (3.5,0) [labelBelow={$-(q+3)$}] {};
\node[bullet] (250) at (2.5,0) [labelAbove={$-2$}] {};
\node[empty] (20) at (2,0) [] {};
\node[empty] (150) at (1.5,0) [] {};
\node[bullet] (10) at (1,0) [labelAbove={$-2$}] {};

\draw[-] (350)--(250);
\draw[-] (250)--(20);
\draw[dotted] (20)--(150);
\draw[-] (150)--(10);
\draw[-] (10)--(00);
\draw [decorate, decoration = {calligraphic brace,mirror}] (1,-0.15)--(2.5,-0.15)  node[pos=0.5,below=0.1em,black]{$r$};


\node[bullet] (0-35) at (0,-3.5) [label=right:{$-(r+3)$}] {};
\node[bullet] (0-25) at (0,-2.5) [label=left:{$-2$}] {};
\node[empty] (0-2) at (0,-2) [] {};
\node[empty] (0-15) at (0,-1.5) [] {};
\node[bullet] (0-1) at (0,-1) [label=left:{$-2$}] {};

\draw[-] (0-35)--(0-25);
\draw[-] (0-25)--(0-2);
\draw[dotted] (0-2)--(0-15);
\draw[-] (0-15)--(0-1);
\draw[-] (0-1)--(00);
\draw [decorate, decoration = {calligraphic brace,mirror}] (0.15,-2.5)--(0.15,-1)  node[pos=0.5,right=0.1em,black]{$p$};

\end{tikzpicture}
\caption{$W(p,q,r)$, where $p, q, r \ge 0$}
\label{figure:Wpqr}
\end{figure}
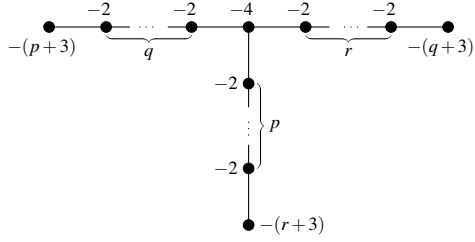

\subsection{A sandwiched structure}

Let $X=W(p,q,r)$. It is sandwiched for having the big node. We add ($p+2$) $(-1)$-vertices to the $-(p+3)$-vertex, ($q+2$) $(-1)$-vertices to the $-(q+3)$-vertex, and ($r+2$) $(-1)$-vertices to the $-(r+3)$-vertex in the resolution graph, making the resolution graph a sandwiched graph. We then consider small curvetta attached to each $(-1)$-vertices, which we denote by $\widetilde{A}_1, \dotsc, \widetilde{A}_{p+2}$, $\widetilde{B}_1, \dotsc, \widetilde{B}_{q+2}$, and $\widetilde{C}_1, \dotsc, \widetilde{C}_{r+2}$, respectively; Figure~\ref{figure:Wpqr-sandwiched}.

\begin{proposition}
The map $\phi_I \colon \mathcal{C}(X) \to \mathcal{I}(X)$ is injective; hence it is bijective.
\end{proposition}

\begin{proof}
Since it has three branches and it has the big node, every line arrangement (Definition~\ref{definition:line-arrangement}) corresponding to $X$ is one of them given in Figure~\ref{figure:line-arrangement}. Therefore we can prove that $\phi_H \colon \mathcal{C}(X) \to \mathcal{H}(X)$ is injective as we did in the proof of Theorem~\ref{theorem:phi_H-injective-WHSS}. Then it follows that $\phi_I \colon \mathcal{C}(X) \to \mathcal{I}(X)$ is also injective.
\end{proof}

\begin{figure}
\centering
\begin{tikzpicture}[scale=0.75]
\tikzset{font=\scriptsize}
\node[bullet] (00) at (0,0) [labelAbove={$-4$}] {};


\node[bullet] (-505) at (-5,0.5) [label=left:{$\widetilde{A}_1$}] {};
\node[bullet] (-5-05) at (-5,-0.5) [label=left:{$\widetilde{A}_{p+2}$}] {};
\draw[dotted] (-5,0.25)--(-5,-0.25);

\node[bullet] (-4505) at (-4.5,0.5) [labelAbove={$-1$}] {};
\node[bullet] (-45-05) at (-4.5,-0.5) [labelBelow={$-1$}] {};
\draw[dotted] (-4.5,0.25)--(-4.5,-0.25);

\draw[-] (-4505)--(-505);
\draw[-] (-45-05)--(-5-05);

\draw[-] (-4505)--(-350);
\draw[-] (-45-05)--(-350);

\node[bullet] (-350) at (-3.5,0) [labelBelow={$-(p+3)$}] {};
\node[bullet] (-250) at (-2.5,0) [labelAbove={$-2$}] {};
\node[empty] (-20) at (-2,0) [] {};
\node[empty] (-150) at (-1.5,0) [] {};
\node[bullet] (-10) at (-1,0) [labelAbove={$-2$}] {};

\draw[-] (-350)--(-250);
\draw[-] (-250)--(-20);
\draw[dotted] (-20)--(-150);
\draw[-] (-150)--(-10);
\draw[-] (-10)--(00);
\draw [decorate, decoration = {calligraphic brace,mirror}] (-2.5,-0.15)--(-1,-0.15)  node[pos=0.5,below=0.1em,black]{$q$};


\node[bullet] (505) at (5,0.5) [label=right:{$\widetilde{B}_1$}] {};
\node[bullet] (5-05) at (5,-0.5) [label=right:{$\widetilde{B}_{q+2}$}] {};
\draw[dotted] (5,0.25)--(5,-0.25);

\node[bullet] (4505) at (4.5,0.5) [labelAbove={$-1$}] {};
\node[bullet] (45-05) at (4.5,-0.5) [labelBelow={$-1$}] {};
\draw[dotted] (4.5,0.25)--(4.5,-0.25);

\draw[-] (4505)--(505);
\draw[-] (45-05)--(5-05);

\draw[-] (4505)--(350);
\draw[-] (45-05)--(350);

\node[bullet] (350) at (3.5,0) [labelBelow={$-(q+3)$}] {};
\node[bullet] (250) at (2.5,0) [labelAbove={$-2$}] {};
\node[empty] (20) at (2,0) [] {};
\node[empty] (150) at (1.5,0) [] {};
\node[bullet] (10) at (1,0) [labelAbove={$-2$}] {};

\draw[-] (350)--(250);
\draw[-] (250)--(20);
\draw[dotted] (20)--(150);
\draw[-] (150)--(10);
\draw[-] (10)--(00);
\draw [decorate, decoration = {calligraphic brace,mirror}] (1,-0.15)--(2.5,-0.15)  node[pos=0.5,below=0.1em,black]{$r$};


\node[bullet] (-05-5) at (-0.5,-5) [label=below:{$\widetilde{C}_1$}] {};
\node[bullet] (05-5) at (0.5,-5) [label=below:{$\widetilde{C}_{r+2}$}] {};
\draw[dotted] (-0.25,-5)--(0.25,-5);

\node[bullet] (-05-45) at (-0.5,-4.5) [label=left:{$-1$}] {};
\node[bullet] (05-45) at (0.5,-4.5) [label=right:{$-1$}] {};
\draw[dotted] (-0.25,-4.5)--(0.25,-4.5);

\draw[-] (-05-5)--(-05-45);
\draw[-] (05-5)--(05-45);

\draw[-] (-05-45)--(0-35);
\draw[-] (05-45)--(0-35);

\node[bullet] (0-35) at (0,-3.5) [label=right:{$-(r+3)$}] {};
\node[bullet] (0-25) at (0,-2.5) [label=left:{$-2$}] {};
\node[empty] (0-2) at (0,-2) [] {};
\node[empty] (0-15) at (0,-1.5) [] {};
\node[bullet] (0-1) at (0,-1) [label=left:{$-2$}] {};

\draw[-] (0-35)--(0-25);
\draw[-] (0-25)--(0-2);
\draw[dotted] (0-2)--(0-15);
\draw[-] (0-15)--(0-1);
\draw[-] (0-1)--(00);
\draw [decorate, decoration = {calligraphic brace,mirror}] (0.15,-2.5)--(0.15,-1)  node[pos=0.5,right=0.1em,black]{$p$};

\end{tikzpicture}
\caption{A sandwiched structure for $W(p,q,r)$}
\label{figure:Wpqr-sandwiched}
\end{figure}
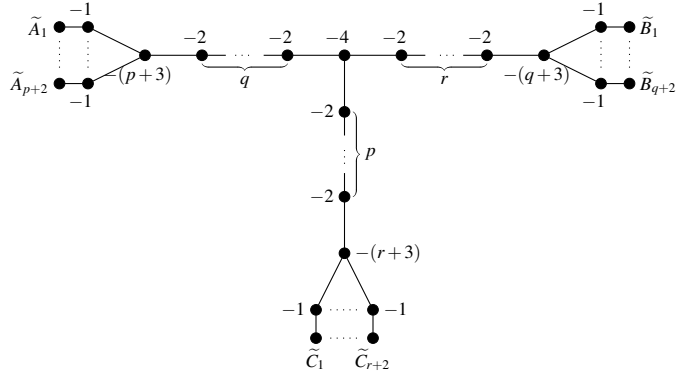

\subsection{The decorated curve and the incidence relations}

We denote by $A_i, B_j, C_k$ the decorated curves of $W(p,q,r)$ in $\mathbb{CP}^2$ that are induced from the sandwiched structure given in Figure~\ref{figure:Wpqr-sandwiched}.

Notice that all curves $A_i, B_j, C_k$ are smooth. Indeed $A_i$'s are plane curves given by $y=a_i x^{q+2}$, and $B_j$'s and $C_k$'s are those given by $y=b_jx^{r+2}$ and $y=c_kx^{p+2}$, respectively. So we have
\begin{equation}\label{equation:Wpqr-matrix-delta-l}
\begin{gathered}
\delta(A_i)=\delta(b_j)=\delta(C_k)=1 \\
l(A_i)=q+3, \quad l(B_j)=r+3, \quad l(C_k)=p+3
\end{gathered}
\end{equation}
for all $i,j,k$. Furthermore the incidence relations between $A_i, B_j, C_k$ are given as follows.
\begin{equation}\label{equation:Wpqr-matrix-intersection}
\begin{gathered}
A_i \cdot A_{i'}= q+2, \quad B_j \cdot B_{j'} = r+2, \quad C_k \cdot C_{k'} = p+2 \\
A_i \cdot B_j = 1, \quad B_j \cdot C_k = 1, \quad C_k \cdot A_i = 1
\end{gathered}
\end{equation}
for all $i, i', j, j', k, k'$.

\subsection{Submatrices of combinatorial incidence matrices}

We find all combinatorial incidence matrices for the decorated curve of $W(p,q,r)$.

First of all, since $\delta(A_i)=\delta(B_j)=\delta(C_k)=1$, all entries of any incidence matrices are $1$. One the other hand, $l(A_i)=q+3$ and $A_i \cdot A_{i'}=q+2$. So the submatrix consisting of the rows corresponding to $A_i$'s are one of the followings:

\begin{center}
\includegraphics{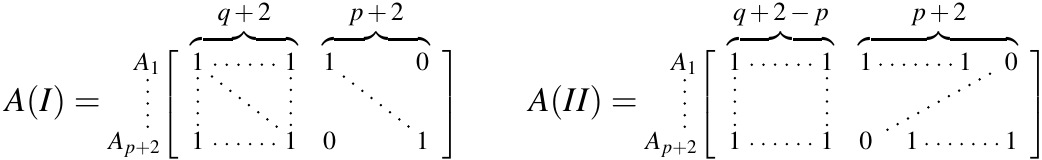}
\end{center}


Notice that $A(II)$ can occur only if $q+2-p \ge 0$. Furthermore, if $p=0$, then $A(I)=A(II)$. So $A(II)$ is considered only when $p \ge 1$.

Similarly, those corresponding to $B_j$'s and $C_k's$ are as follows:

\begin{center}
\includegraphics{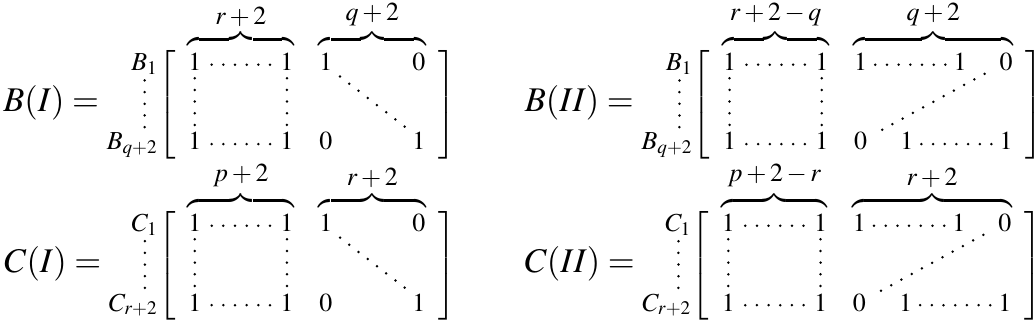}
\end{center}

%

\noindent where $B(II)$ and $C(II)$ can occur only if $r+2-q \ge 0$, $q \ge 1$, and $p+2-r \ge 0$, $r \ge 1$, respectively.

\subsection{Combinations of the submatrices}

We will explain how the columns of the above submatrices are combined to produce combinatorial incidence matrices. Notice that we should have $A_i \cdot B_j = 1$, $B_j \cdot C_k = 1$, $C_k \cdot A_i = 1$ from Equation~\eqref{equation:Wpqr-matrix-intersection}. That is, the inner product of any pair of the $A_i$-rows ($B_j$-rows, or $C_k$-rows) and the $B_j$-rows($C_k$-rows, or $A_i$-rows, respectively) is equal to $1$.

We begin with analyzing all possible combinations of the columns of the submatrices $A$ and $B$ such that $A_i \cdot B_j = 1$.

At first, suppose that $A(I)$ appears in a combinatorial incidence matrix $M$. We divide cases according to which columns of $B(I)$ or $B(II)$ can appear under the first column of $A(I)$.

\textit{Case~1}: A column of $B(I)$ or $B(II)$ consisting of only $1$ under the first column of $A(I)$.

The submatrix corresponding to $A$-rows and $B$-rows looks like

\begin{center}
\includegraphics{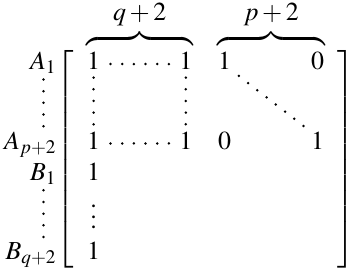}
\end{center}


\noindent Then the inner products `$A_i \cdot B_j =1$' occur on the first column of the incidence matrix $M$ for any $i, j$. So there must be no $1$ below the other columns of $A(I)$. Therefore the combinatorial incidence matrix $M$ contains the following submatrix

\begin{center}
\includegraphics{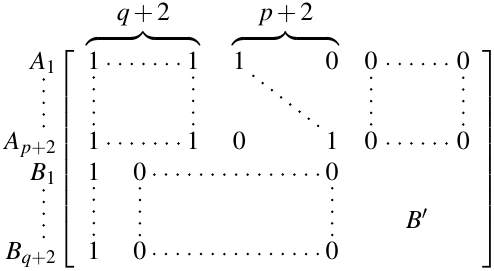}
\end{center}


\noindent where $B'$ is the submatrix obtained from $B(I)$ or $B(II)$ by deleting one column consisting of only $1$'s.

\textit{Case~2}: A column of $B(I)$ with only one $1$ under the first column of $A(I)$.

We may assume that the submatrix corresponding to $A$-rows and $B$-rows looks like

\begin{center}
\includegraphics{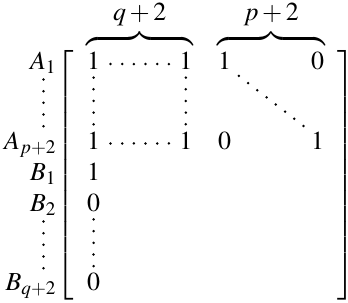}
\end{center}


\noindent Then the inner product `$A_i \cdot B_1=1$' occur on the first entry $1$ of $B_1$ for all $i$. Therefore there is no $1$ on the $B_1$-row until the last column of $A$. So we have

\begin{center}
\includegraphics{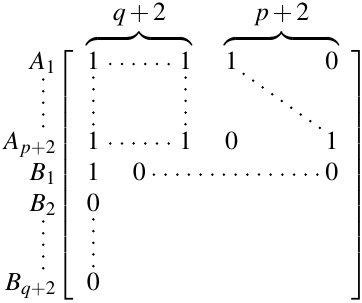}
\end{center}


\noindent Hence any columns of $B(I)$ consisting of only $1$'s cannot appear under any columns of $A(I)$. So in order to satisfy the conditions `$A_i \cdot B_j=1$', the submatrix corresponding to $A_i$-rows and $B_j$-rows must be of the form

\smallskip

\begin{equation}\label{equation:Wpqr-AB-2}
\includegraphics{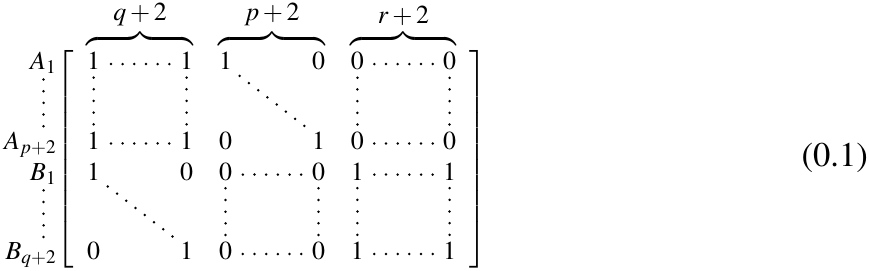}
\end{equation}

\textit{Case~3}. A column of $B(II)$ with only one $0$ under the first column of $A(I)$.

We may assume that the submatrix corresponding to $A$-rows and $B$-rows looks like

\begin{center}
\includegraphics{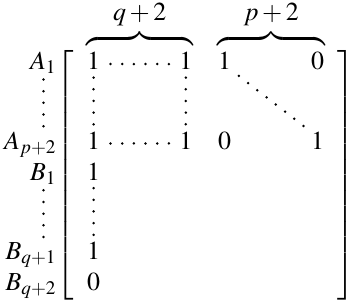}
\end{center}


\noindent As before, because of the inner product condition `$A_i \cdot B_j=1$', the entries of $B_j$-rows for $j=1,\dots,q+1$ must be $0$ until the last column of $A(I)$. So we have

\begin{center}
\includegraphics{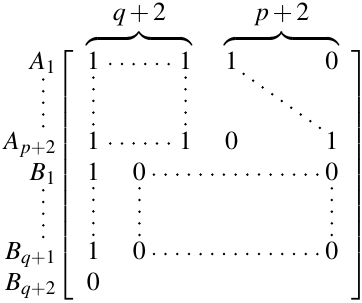}
\end{center}


\noindent But, since $q \ge 1$, there is no way to complete the $B_{q+2}$-row so that the inner product condition $A_i \cdot B_{q+2}=1$ satisfies for any $i$. Therefore Case~3 cannot occur.

\textit{Case~4}. No columns of $B(I)$ and $B(II)$ under the first column of $A(I)$.

That is, we may assume that

\begin{center}
\includegraphics{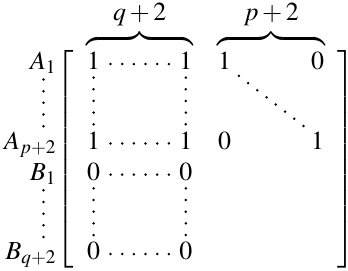}
\end{center}


\noindent Then there must be the columns of $B(I)$ or $B(II)$ consisting of only $1$'s under the $(p+2)$-column of $A(I)$ in order to satisfy the inner product condition `$A_i \cdot B_j=1$. That is, we have


\begin{equation}\label{equation:Wpqr-AB-3}
\begin{aligned}
\includegraphics{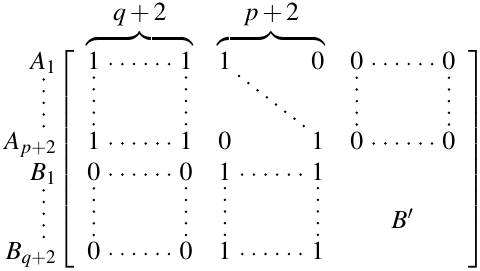}
\end{aligned}
\end{equation}
where $B'$ is the submatrix obtained from $B(I)$ or $B(II)$ by deleting $(p+2)$ columns consisting of only $1$'s. Notice that we must assume that
\begin{equation}\label{equation:Wpqr-AB-inequality}
\text{$r+2 \ge p+2$ for $B(I)$ or $r+2-q \ge p+2$ for $B(II)$}
\end{equation}
because we need at most $(p+2)$ columns consisting of only $1$'s in $B(I)$ or $B(II)$.

Next, assume that the submatrix $A(II)$ appears in the incidence matrix $M$. As before, we divide the cases as follows.

\textit{Case~5}. A column of $B(I)$ or $B(II)$ consisting of only $1$'s under the first column of $A(II)$.

By the similar arguments as in Case~1, we have

\begin{center}
\includegraphics{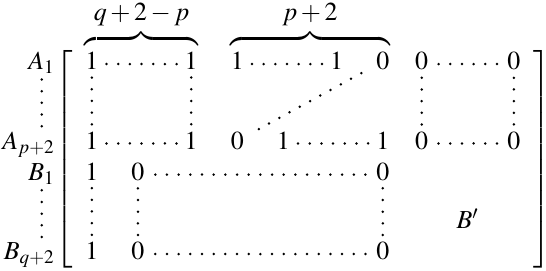}
\end{center}


\noindent where $B'$ is the submatrix obtained from $B(I)$ or $B(II)$ by deleting one column consisting of only $1$'s.

\textit{Case~6}. A column of $B(I)$ with only one $1$ under the first column of $A(II)$.

As in Case~2, the entries of the $B_1$-row except the first one must be $0$. So we have

\begin{center}
\includegraphics{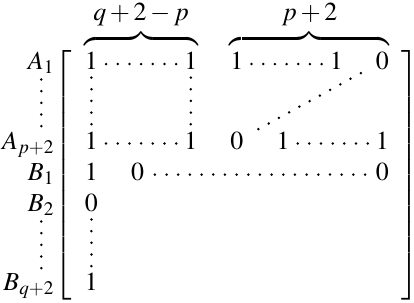}
\end{center}


\noindent Since we assume that $p \ge 1$, there must be a $B_j$-row whose entries under the $(q+2-p)$-columns of $A(II)$ are all $0$. We may assume that the $B_{q+2}$-row is of this type. Then there must be only one $1$ in the entries of the $B_{q+2}$-row under the $(p+2)$-columns of $A(II)$ in order to satisfy the inner product condition `$A_i \cdot B_{q+2}=1$'. However no matter where $1$ is placed, there must be a row of $A(II)$ that does not satisfy the inner product condition. So Case~6 cannot occur.

\textit{Case~7}. A column of $B(II)$ with only one $0$ under the first column of $A(II)$.

Because of the exactly same reason in Case~3, this case cannot happen.

\textit{Case~8}. No columns of $B(I)$ or $B(II)$ under the first column of $A(II)$.

We will show that this last Case~8 cannot occur as well. We may assume that all entries under the $(q+2-p)$-columns of $A(II)$ are $0$. That is,

\begin{center}
\includegraphics{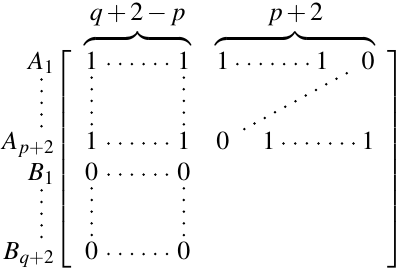}
\end{center}


Notice that, under the $(p+2)$-columns of $A(II)$, there must be a non-zero column of $B(I)$ or $B(II)$. So we may assume that the first entry in the $B_1$-row under the first column of the $(p+2)$-columns of $A(II)$ is $1$. Then the other entries of the $B_1$-row should be zero because $A_i\cdot B_1 =1$ and $p \ge 1$. Then we have

\begin{center}
\includegraphics{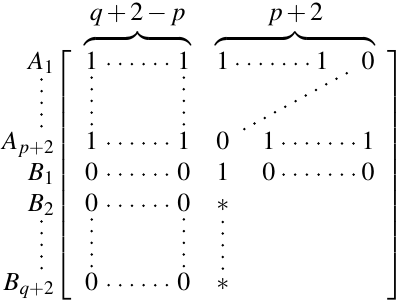}
\end{center}

%
\noindent But, $A_{p+2} \cdot B_1 = 0$, which violates the inner product conditions. Therefore Case~8 cannot occur.

One can analyze combinations of the submatrices $B, C$ and $C, A$ in a similar way. In summary,

\begin{lemma}\label{lemma:Wpqr}
The inner products occur either on the columns of the submatrices $A, B, C$ consisting of only $1$'s or on the identity submatrix of $A(I), B(I), C(I)$ and the submatrix consisting of only $1$'s of $A, B, C$ under the suitable conditions on the number of columns as in Equation~\eqref{equation:Wpqr-AB-inequality}.
\end{lemma}

\subsection{All possible combinatorial incidence matrices}

Suppose that the inner product occurs only on the columns of the submatrices whose entries consisting only $1$. Then there are only two possibilities:

\begin{center}
\includegraphics{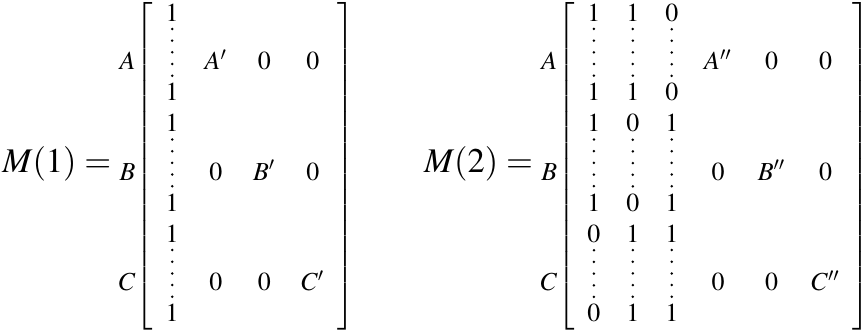}
\end{center}


Notice that $A', A''$, $B', B''$, $C', C''$ in the figure are the submatrices obtained from $A, B, C$ by deleting one or two columns consisting of only $1$'s, respectively. So if any of the submatrices $A(II)$, $B(II)$, or $C(II)$ are encountered in the above incidence matrices, then $M(1)$ can occur under the conditions
\begin{equation}\label{equation:condition-M(1)-with-II}
\text{$p \ge 1$, $q+2-p \ge 1$, or $q \ge 1$, $r+2-q \ge 1$, or $r \ge 1$, $p+2-r \ge 1$,}
\end{equation}
respectively, and $M(2)$ can occur under the conditions
\begin{equation}\label{equation:condition-M(2)-with-II}
\text{$p \ge 1$, $q+2-p \ge 2$, or $q \ge 1$, $r+2-q \ge 2$, or $r \ge 1$, $p+2-r \ge 2$,}
\end{equation}
respectively.

Suppose that the inner product happens on the identity submatrix of $A(I)$ and the submatrix of $B$ consisting of only $1$'s as in Equation~\eqref{equation:Wpqr-AB-3}. Then the inner products on $A, C$ and $B, C$ must occur on the columns consisting of only $1$'s by Lemma~\ref{lemma:Wpqr}. Therefore we have

\begin{center}
\includegraphics{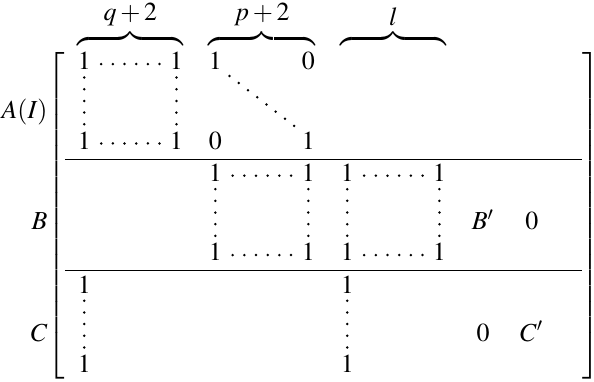}
\end{center}

%
\noindent where $B'$ and $C'$ are the remained parts of $B$ and $C$, respectively. Notice that the above configuration can happen under the condition
\begin{equation*}
l := \begin{cases}
r-p \ge 1 & \text{for $B(I)$} \\
r-p-q \ge 1 & \text{for $B(II)$}
\end{cases}
\end{equation*}

We claim that the above configuration can happen only for $C(I)$: The submatrix $C(II)$ must have at least two columns consisting of only $1$'s. So $p+2-r \ge 2$, which contradicts, however, to the above condition $l \ge 1$. Therefore we have the following incidence matrix

\begin{center}
\includegraphics{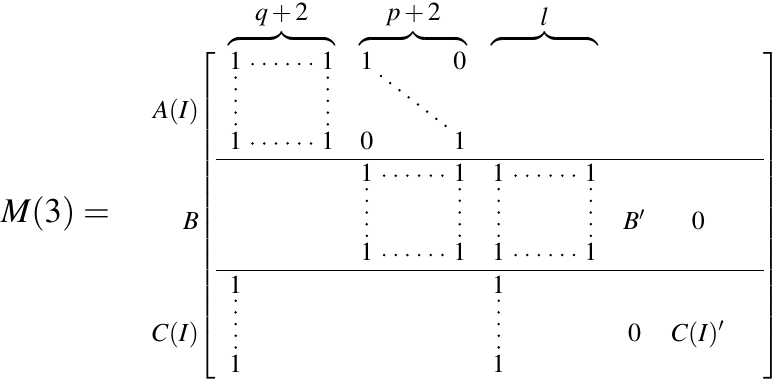}
\end{center}

%
\noindent under the condition
\begin{equation}\label{equation:condition-M(3)}
l := \begin{cases}
r-p \ge 1 & \text{if $B=B(I)$} \\
r-p-q \ge 1, q \ge 1 & \text{if $B=B(II)$}
\end{cases}
\end{equation}

Similarly, we have more incidence matrices

\begin{center}
\includegraphics{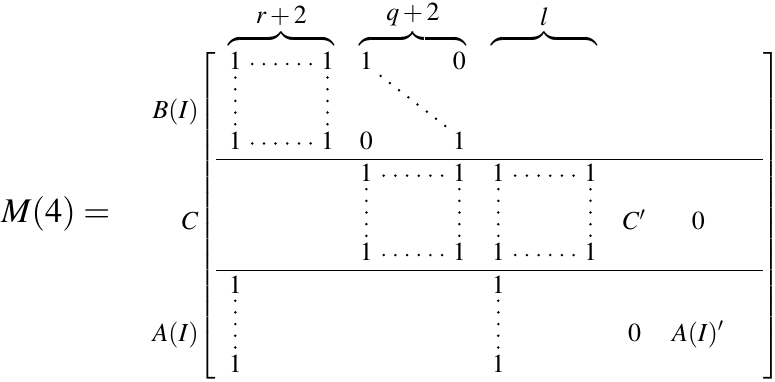}
\end{center}

%
\noindent under the condition
\begin{equation}\label{equation:condition-M(4)}
l := \begin{cases}
p-q \ge 1 & \text{if $C=C(I)$} \\
p-q-r \ge 1, r \ge 1 & \text{if $C=C(II)$}
\end{cases}
\end{equation}

\begin{center}
\includegraphics{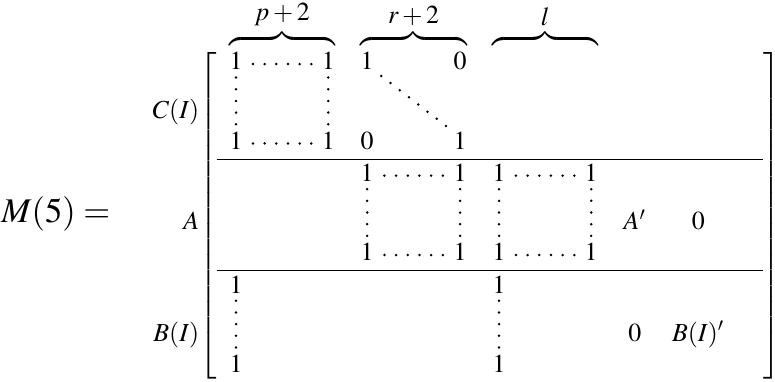}
\end{center}

%
\noindent under the condition
\begin{equation}\label{equation:condition-M(5)}
l := \begin{cases}
q-r \ge 1 & \text{if $A=A(I)$} \\
q-r-p \ge 1, p \ge 1 & \text{if $A=A(II)$}
\end{cases}
\end{equation}

Finally, suppose that the inner product happens on the identity submatrix of $A(I)$ and the submatrix of $B$ consisting of only $1$'s as in Equation~\eqref{equation:Wpqr-AB-2}. Then there is only one possible combinatorial incidence matrix

\begin{center}
\includegraphics{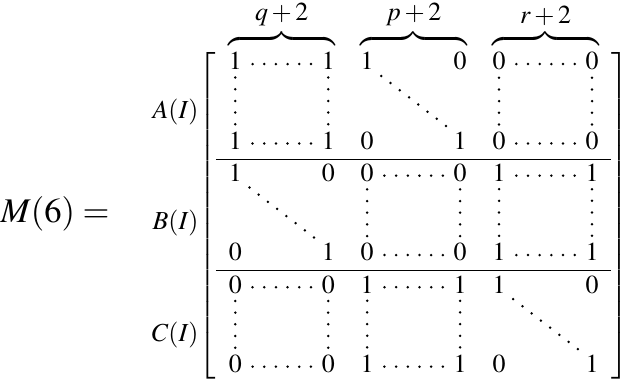}
\end{center}


In addition, if $p=q=r$, then we have one more combinatorial matrix (which Jonathan Wahl informed us):

\begin{center}
\includegraphics{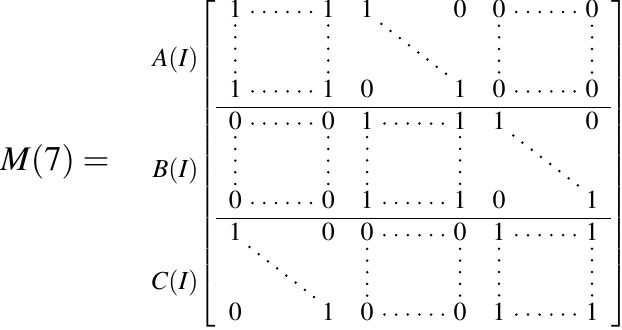}
\end{center}


\begin{proposition}\label{proposition:Wpqr-all-incidence-matrices}
There are seven types of combinatorial incidence matrices $M(1), \dotsc, M(6)$ for $W(p,q,r)$ under the suitable conditions on $p,q,r$ described in Equation~\eqref{equation:condition-M(1)-with-II} for $M(1)$, \eqref{equation:condition-M(2)-with-II} for $M(2)$, \eqref{equation:condition-M(3)} for $M(3)$, \eqref{equation:condition-M(4)} for $M(4)$, \eqref{equation:condition-M(5)} for $M(5)$.
\end{proposition}

\section{Kollár conjecture for $W(p,q,r)$}\label{section:Wpqr-K-Conjecture}

In this section we prove:

\begin{theorem}\label{theorem:Wpqr-K-conjecture}
Kollár conjecture holds for the singularities $W(p,q,r)$.
\end{theorem}

We divide the proof according to the types of the combinatorial incidence matrices of $W(p,q,r).$

\subsection{P-resolutions for $M(1)$ and $M(2)$}

We first construct P-resolutions corresponding to each of the combinatorial incidence matrices of type $M(1)$ or $M(2)$.

\begin{proposition}\label{proposition:Wpqr-P-modifications-M(1)orM(2)}
Each combinatorial incidence matrices of type $M(1)$ or $M(2)$ correspond to one of the following M-resolution of $W(p,q,r)$.

\begin{enumerate}[(1)]
\item The minimal resolution itself. It corresponds to the (combinatorial) incidence matrix $M(1)$ with $A(I)$, $B(I)$, and $C(I)$.

\item A M-resolution obtained by contracting the central $(-4)$-curve. It corresponds to the (combinatorial) incidence matrix $M(2)$ with $A(I)$, $B(I)$, and $C(I)$.

\item A M-resolution obtained by contracting the $(-p-3)$-curve and the $(p-1)$ $(-2)$-curves that follow it (that is, the Wahl configuration $[p+3,2,\dotsc,2]$) from the $A$-arm (assuming $p \ge 1$ and $q \ge p-1$). It corresponds to the (combinatorial) incidence matrix $M(1)$ with $A(II)$, but $B(I)$ and $C(I)$.

\item A M-resolution obtained by contracting the Wahl configuration $[q+3,2,\dotsc,2]$ from the $B$-arm (assuming $q \ge 1$ and $r \ge q-1$). It corresponds to the (combinatorial) incidence matrix $M(1)$ with $A(I), B(II), C(I)$.

\item A M-resolution obtained by contracting the Wahl configuration $[r+3,2,\dotsc,2]$ from the $C$-arm (assuming $r \ge 1$ and $p \ge r-1$). It corresponds to the (combinatorial) incidence matrix $M(1)$ with $A(I), B(I), C(II)$.

\item Any M-resolutions obtained by combining (3), (4), or (5) under the necessary hypotheses on $p, q, r$. They correspond to the (combinatorial) incidence matrix $M(1)$ whose submatrices $A, B, C$ are determined by combining as in (3), (4), (5).

\item Any M-resolutions obtained by combining (3), (4), or (5) with (2) under the stronger conditions  $q \ge p$, $p \ge r$, and $r \ge q$, respectively (in order to ensure that the contractions can be performed independently). They correspond to the (combinatorial) incidence matrix $M(2)$ whose submatrices $A, B, C$ are determined by combining as in (3), (4), (5).
\end{enumerate}
\end{proposition}

\begin{proof}
For P-resolutions of type $M(1)$, there are no T-singularities. We need only divisorial contractions. As we see in Section~\ref{section:[4]}, it is easy to track how $(-1)$-curves intersect the decorated curves in a general fiber. For P-resolutions of type $M(2)$, we've already covered the case where the singularity $[4]$ is located in the central node in Section~\ref{section:[4]}. Also we have dealt with cases in which P-resolutions possess only Wahl singularities of type $[a+4,2,\dotsc,2]$ in Section~\ref{section:example-cyclic}. Therefore, we can prove the proposition in a similar way.
\end{proof}

\subsection{P-resolutions for $M(3)$, $M(4)$, and $M(5)$}

If $r-p-1 \ge 0$ and $q \ge 1$, then we can construct a P-resolution of $W(p,q,r)$ whose dual graph is given in Figure~\ref{figure:Wpqr-P-for-M(3)-I-q>=1}. It has two T-singularities:
\begin{equation*}
\text{$[p+3,\underbrace{2,\cdots,2}_{q-1}, 3, \underbrace{2,\dotsc,2}_{p}]$ and $[p+5,\underbrace{2,\dotsc,2}_{p+1}]$}
\end{equation*}

If $r-p-1 \ge 0$ but $q=0$, then we can construct a P-resolution of $W(p,q,r)$ as in Figure~\ref{figure:Wpqr-P-for-M(3)-I-q=0} which has two T-singularities:
\begin{equation*}
\text{$[p+4,\underbrace{2,\dotsc,2}_{p}]$ and $[p+5,\underbrace{2,\dotsc,2}_{p+1}]$}
\end{equation*}

Moreover, if $r-p-1 \ge q$ and $q \ge 1$, we can contract the Wahl configuration $2-\dotsb-2-(q+3)$ in the $B$-arm to obtain another P-resolution of $W(p,q,r)$ as shown in Figure~\ref{figure:Wpqr-P-for-M(3)-II-q>=1} which has one more T-singularity  in addition to the two singularities:
\begin{equation*}
[\underbrace{2,\dotsc,2}_{q-1},q+3]
\end{equation*}

Observe that it is possible to verify that the canonical divisors of the aforementioned P-resolutions are ample by applying the algorithm for calculating discrepancies described, for instance, in Urzúa--Vilches~\cite[\S2.3]{Urzua-Vilches-2021}.

\begin{figure}
\centering
\begin{tikzpicture}[scale=0.75]
\tikzset{font=\scriptsize}
\node[rectangle] (00) at (0,0) [labelAbove={$-p-5$}] {};


\node[rectangle] (-80) at (-8,0) [labelAbove={$-(p+3)$}] {};

\node[rectangle] (-70) at (-7,0) [labelAbove={$-2$}] {};  
\node[empty] (-650) at (-6.5,0) [] {};
\node[empty] (-60) at (-6,0) [] {};
\node[rectangle] (-550) at (-5.5,0) [labelAbove={$-2$}] {}; 

\draw [decorate, decoration = {calligraphic brace,mirror}] (-7,-0.15)--(-5.5,-0.15)  node[pos=0.5,below=0.1em,black]{$q-1$};

\node[rectangle] (-450) at (-4.5,0) [labelAbove={$-3$}] {};

\node[rectangle] (-350) at (-3.5,0) [labelAbove={$-2$}] {}; 
\node[empty] (-30) at (-3,0) [] {};
\node[empty] (-250) at (-2.5,0) [] {};
\node[rectangle] (-20) at (-2,0) [labelAbove={$-2$}] {}; 

\draw [decorate, decoration = {calligraphic brace,mirror}] (-3.5,-0.15)--(-2,-0.15)  node[pos=0.5,below=0.1em,black]{$p$};

\node[bullet] (-10) at (-1,0) [labelAbove={$-1$}] {};

\draw[-] (-80)--(-70);

\draw[-] (-70)--(-650);
\draw[dotted] (-650)--(-60);
\draw[-] (-60)--(-550);
\draw[-] (-550)--(-450);

\draw[-] (-450)--(-350);

\draw[-] (-350)--(-30);
\draw[dotted] (-30)--(-250);
\draw[-] (-250)--(-20);

\draw[-] (-10)--(-20);
\draw[-] (-10)--(00);


\node[rectangle] (10) at (1,0) [labelAbove={$-2$}] {}; 
\node[empty] (150) at (1.5,0) [] {};
\node[empty] (20) at (2,0) [] {};
\node[rectangle] (250) at (2.5,0) [labelAbove={$-2$}] {}; 

\draw [decorate, decoration = {calligraphic brace,mirror}] (1,-0.15)--(2.5,-0.15)  node[pos=0.5,below=0.1em,black]{$p+1$};

%
%

\node[bullet] (350) at (3.5,0) [label=right:{$-2$}] {}; 
\node[empty] (35-05) at (3.5,-0.5) [] {};
\node[empty] (35-1) at (3.5,-1) [] {};
\node[bullet] (35-15) at (3.5,-1.5) [label=right:{$-2$}] {}; 

\draw [decorate, decoration = {calligraphic brace, mirror}] (3.35,0)--(3.35,-1.5)  node[pos=0.5,left=0.1em,black]{$r-p+1$};

\node[bullet] (35-25) at (3.5,-2.5) [label=right:{$-(q+3)$}] {};

\draw[-] (00)--(10);

\draw[-] (10)--(150);
\draw[dotted] (150)--(20);
\draw[-] (20)--(250);

\draw[-] (250)--(350);
\draw[-] (350)--(35-05);
\draw[dotted] (35-05)--(35-1);
\draw[-] (35-1)--(35-15);

\draw[-] (35-15)--(35-25);


\node[bullet] (0-35) at (0,-3.5) [label=left:{$-(r+3)$}] {};

\node[bullet] (0-25) at (0,-2.5) [label=left:{$-2$}] {}; 
\node[empty] (0-2) at (0,-2) [] {};
\node[empty] (0-15) at (0,-1.5) [] {};
\node[bullet] (0-1) at (0,-1) [label=left:{$-2$}] {}; 

\draw [decorate, decoration = {calligraphic brace,mirror}] (0.15,-2.5)--(0.15,-1)  node[pos=0.5,right=0.1em,black]{$p$};

\draw[-] (0-35)--(0-25);
\draw[-] (0-25)--(0-2);
\draw[dotted] (0-2)--(0-15);
\draw[-] (0-15)--(0-1);
\draw[-] (0-1)--(00);
\end{tikzpicture}
\caption{The P-resolution for $M(3)$ with $B(I)$ for $r-p \ge 1$, $q \ge 1$}
\label{figure:Wpqr-P-for-M(3)-I-q>=1}
\end{figure}
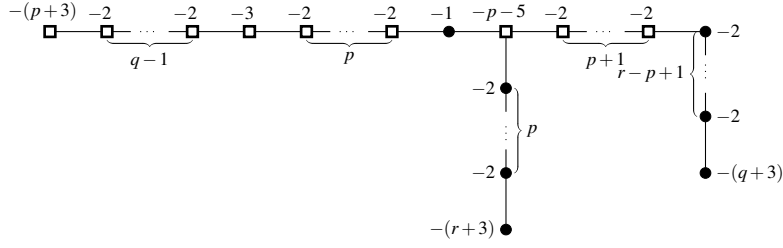

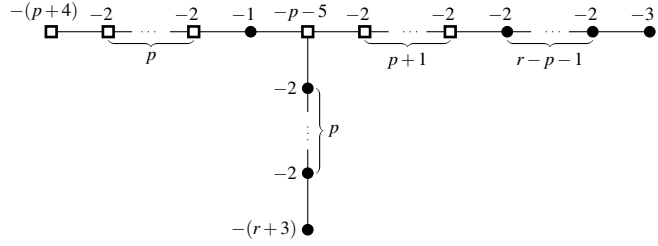
\begin{figure}
\centering
\begin{tikzpicture}[scale=0.75]
\tikzset{font=\scriptsize}


\node[rectangle] (00) at (0,0) [labelAbove={$-p-5$}] {};


\node[rectangle] (-450) at (-4.5,0) [labelAbove={$-(p+4)$}] {};

\node[rectangle] (-350) at (-3.5,0) [labelAbove={$-2$}] {};
\node[empty] (-30) at (-3,0) [] {};
\node[empty] (-250) at (-2.5,0) [] {};
\node[rectangle] (-20) at (-2,0) [labelAbove={$-2$}] {};

\node[bullet] (-10) at (-10) [labelAbove={$-1$}] {};

\draw[-] (-450)--(-350);
\draw[-] (-350)--(-30);
\draw[dotted] (-30)--(-250);
\draw[-] (-250)--(-20);
\draw[-] (-20)--(-10);
\draw[-] (-10)--(00);
\draw [decorate, decoration = {calligraphic brace,mirror}] (-3.5,-0.15)--(-2,-0.15)  node[pos=0.5,below=0.1em,black]{$p$};


\node[bullet] (60) at (6,0) [labelAbove={$-3$}] {};

\node[bullet] (50) at (5,0) [labelAbove={$-2$}] {};
\node[empty] (450) at (4.5,0) [] {};
\node[empty] (40) at (4,0) [] {};
\node[bullet] (350) at (3.5,0) [labelAbove={$-2$}] {};

\node[rectangle] (250) at (2.5,0) [labelAbove={$-2$}] {};
\node[empty] (20) at (2,0) [] {};
\node[empty] (150) at (1.5,0) [] {};
\node[rectangle] (10) at (1,0) [labelAbove={$-2$}] {};

\draw[-] (60)--(50);
\draw[-] (50)--(450);
\draw[dotted] (450)--(40);
\draw[-] (40)--(350);

\draw [decorate, decoration = {calligraphic brace,mirror}] (3.5,-0.15)--(5,-0.15)  node[pos=0.5,below=0.1em,black]{$r-p-1$};

\draw[-] (350)--(250);
\draw[-] (250)--(20);
\draw[dotted] (20)--(150);
\draw[-] (150)--(10);
\draw[-] (10)--(00);
\draw [decorate, decoration = {calligraphic brace,mirror}] (1,-0.15)--(2.5,-0.15)  node[pos=0.5,below=0.1em,black]{$p+1$};


\node[bullet] (0-35) at (0,-3.5) [label=left:{$-(r+3)$}] {};
\node[bullet] (0-25) at (0,-2.5) [label=left:{$-2$}] {};
\node[empty] (0-2) at (0,-2) [] {};
\node[empty] (0-15) at (0,-1.5) [] {};
\node[bullet] (0-1) at (0,-1) [label=left:{$-2$}] {};

\draw[-] (0-35)--(0-25);
\draw[-] (0-25)--(0-2);
\draw[dotted] (0-2)--(0-15);
\draw[-] (0-15)--(0-1);
\draw[-] (0-1)--(00);
\draw [decorate, decoration = {calligraphic brace,mirror}] (0.15,-2.5)--(0.15,-1)  node[pos=0.5,right=0.1em,black]{$p$};

\end{tikzpicture}
\caption{The P-resolution for $M(3)$ with $B(I)$ for $r-p \ge 1$, $q=0$}
\label{figure:Wpqr-P-for-M(3)-I-q=0}
\end{figure}

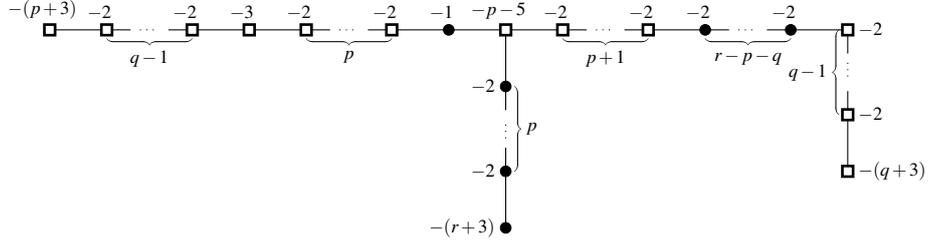
\begin{figure}
\centering
\begin{tikzpicture}[scale=0.75]
\tikzset{font=\scriptsize}
\node[rectangle] (00) at (0,0) [labelAbove={$-p-5$}] {};


\node[rectangle] (-80) at (-8,0) [labelAbove={$-(p+3)$}] {};

\node[rectangle] (-70) at (-7,0) [labelAbove={$-2$}] {};  
\node[empty] (-650) at (-6.5,0) [] {};
\node[empty] (-60) at (-6,0) [] {};
\node[rectangle] (-550) at (-5.5,0) [labelAbove={$-2$}] {}; 

\draw [decorate, decoration = {calligraphic brace,mirror}] (-7,-0.15)--(-5.5,-0.15)  node[pos=0.5,below=0.1em,black]{$q-1$};

\node[rectangle] (-450) at (-4.5,0) [labelAbove={$-3$}] {};

\node[rectangle] (-350) at (-3.5,0) [labelAbove={$-2$}] {}; 
\node[empty] (-30) at (-3,0) [] {};
\node[empty] (-250) at (-2.5,0) [] {};
\node[rectangle] (-20) at (-2,0) [labelAbove={$-2$}] {}; 

\draw [decorate, decoration = {calligraphic brace,mirror}] (-3.5,-0.15)--(-2,-0.15)  node[pos=0.5,below=0.1em,black]{$p$};

\node[bullet] (-10) at (-1,0) [labelAbove={$-1$}] {};

\draw[-] (-80)--(-70);

\draw[-] (-70)--(-650);
\draw[dotted] (-650)--(-60);
\draw[-] (-60)--(-550);
\draw[-] (-550)--(-450);

\draw[-] (-450)--(-350);

\draw[-] (-350)--(-30);
\draw[dotted] (-30)--(-250);
\draw[-] (-250)--(-20);

\draw[-] (-10)--(-20);
\draw[-] (-10)--(00);


\node[rectangle] (10) at (1,0) [labelAbove={$-2$}] {}; 
\node[empty] (150) at (1.5,0) [] {};
\node[empty] (20) at (2,0) [] {};
\node[rectangle] (250) at (2.5,0) [labelAbove={$-2$}] {}; 

\draw [decorate, decoration = {calligraphic brace,mirror}] (1,-0.15)--(2.5,-0.15)  node[pos=0.5,below=0.1em,black]{$p+1$};

\node[bullet] (350) at (3.5,0) [labelAbove={$-2$}] {};  
\node[empty] (40) at (4,0) [] {};
\node[empty] (450) at (4.5,0) [] {};
\node[bullet] (50) at (5,0) [labelAbove={$-2$}] {}; 

\draw [decorate, decoration = {calligraphic brace, mirror}] (3.5,-0.15)--(5,-0.15)  node[pos=0.5,below=0.1em,black]{$r-p-q$};

\node[rectangle] (60) at (6,0) [label=right:{$-2$}] {}; 
\node[empty] (6-05) at (6,-0.5) [] {};
\node[empty] (6-1) at (6,-1) [] {};
\node[rectangle] (6-15) at (6,-1.5) [label=right:{$-2$}] {}; 

\draw [decorate, decoration = {calligraphic brace, mirror}] (5.85,0)--(5.85,-1.5)  node[pos=0.5,left=0.1em,black]{$q-1$};

\node[rectangle] (6-25) at (6,-2.5) [label=right:{$-(q+3)$}] {};

\draw[-] (00)--(10);

\draw[-] (10)--(150);
\draw[dotted] (150)--(20);
\draw[-] (20)--(250);

\draw[-] (250)--(350);
\draw[-] (350)--(40);
\draw[dotted] (40)--(450);
\draw[-] (450)--(50);

\draw[-] (50)--(60);

\draw[-] (60)--(6-05);
\draw[dotted] (6-05)--(6-1);
\draw[-] (6-1)--(6-15);

\draw[-] (6-15)--(6-25);


\node[bullet] (0-35) at (0,-3.5) [label=left:{$-(r+3)$}] {};

\node[bullet] (0-25) at (0,-2.5) [label=left:{$-2$}] {}; 
\node[empty] (0-2) at (0,-2) [] {};
\node[empty] (0-15) at (0,-1.5) [] {};
\node[bullet] (0-1) at (0,-1) [label=left:{$-2$}] {}; 

\draw [decorate, decoration = {calligraphic brace,mirror}] (0.15,-2.5)--(0.15,-1)  node[pos=0.5,right=0.1em,black]{$p$};

\draw[-] (0-35)--(0-25);
\draw[-] (0-25)--(0-2);
\draw[dotted] (0-2)--(0-15);
\draw[-] (0-15)--(0-1);
\draw[-] (0-1)--(00);
\end{tikzpicture}
\caption{The P-resolution for $M(3)$ with $B(II)$ for $r-p-q \ge 1$, $q \ge 1$}
\label{figure:Wpqr-P-for-M(3)-II-q>=1}
\end{figure}

\begin{proposition}\label{proposition:Wpqr-P-modifications-M(3)M(4)orM(5)}
Every (combinatorial) incidence matrix of type $M(3)$ corresponds to one of the P-resolutions given in Figure~\ref{figure:Wpqr-P-for-M(3)-I-q>=1}, \ref{figure:Wpqr-P-for-M(3)-I-q=0}, and \ref{figure:Wpqr-P-for-M(3)-II-q>=1} according to the choice of $B$. Similarly, any (combinatorial) incidence matrices of type $M(4)$ or $M(5)$ correspond to the P-resolutions in the same figures; however, their subscripts are modified as follows: $(p,q,r)\rightarrow (q,r,p)$ for $M(4)$ and $(p,q,r) \rightarrow (r,p,q)$ for $M(5)$.
\end{proposition}

\begin{proof}
Let $M_3$ be the (combinatorial) incidence matrix in Figure~\ref{figure:Wpqr-P-for-M(3)-I-q>=1}. We will prove the proposition for $M_3$. The other cases are proven in a similar fashion.

We take the crepant M-resolution of Figure~\ref{figure:Wpqr-P-for-M(3)-I-q>=1}. Notice that the crepant M-resolution of the T-singularity $[p+3,2,\dotsc,2,3,2,\dotsc,2]$ in Figure~\ref{figure:Wpqr-P-for-M(3)-I-q>=1} is given by
\begin{equation*}
[p+4,2,\dotsc,2]-1-[p+4,2,\dotsc,2]-1-\dotsb-[p+4,2,\dotsc,2]-1-[p+4,2,\dotsc,2]
\end{equation*}
There are $q$ $(-1)$-curves in the above crepant M-resolution.

Let $Z$ be the compactified crepant M-resolution of Figure~\ref{figure:Wpqr-P-for-M(3)-I-q>=1} and let $\mathcal{Z} \to \Delta$ be a $\mathbb{Q}$-Gorenstein smoothing of the central fiber $Z_0=Z$.

At first, applying the MMP to the B-arm and the C-arm, we obtain the following gray part of the incidence matrix $M_3$:

\begin{center}
\includegraphics{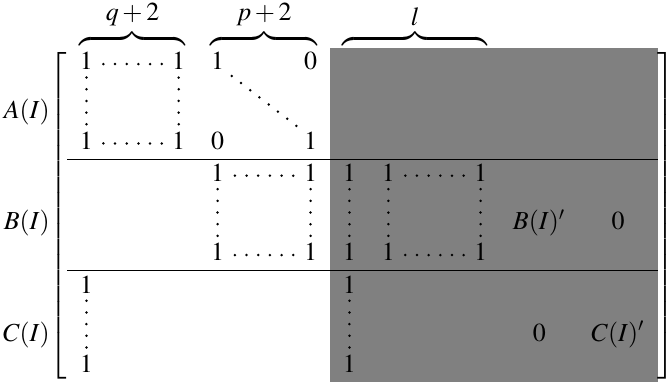}
\end{center}


After the MMP, we have a new deformation, denoted again by $\mathcal{Z} \to \Delta$ for simplicity, whose central fiber contains the following configuration
\begin{equation}\label{equation:Wpqr-M(3)-after-BC-arms}
\begin{aligned}
\begin{tikzpicture}[scale=0.85]

\node at (-1.75,0) [] {(A-arm)};
\node[empty] (-10) at (-1,0) [] {};

\node[bullet] (00) at (0,0) [labelAbove={$-p-3$},label=below left:{$F$}] {};

\node[bullet] (105) at (1,0.5) [label=right:{$\widetilde{B}_1'$}] {};
\node[bullet] (1-05) at (1,-0.5) [label=right:{$\widetilde{B}_{q+2}'$}] {};

\node[bullet] (-05-1) at (-0.5,-1) [label=below:{$\widetilde{C}_1$}] {};
\node[bullet] (05-1) at (0.5,-1) [label=below:{$\widetilde{C}_{r+2}$}] {};

\draw[-] (-10)--(00);
\draw[-] (00)--(105);
\draw[-] (00)--(1-05);
\draw[-] (00)--(-05-1);
\draw[-] (00)--(05-1);

\draw[dotted] (1,0.35)--(1,-0.35);
\draw[dotted] (0.35,-1)--(-0.35,-1);
\end{tikzpicture}
\end{aligned}
\end{equation}

Here we have degenerations of $\widetilde{B}_j$ induced by the usual flip:
\begin{equation*}
\widetilde{B}_j \rightsquigarrow \widetilde{B}_j' + F
\end{equation*}
for all $j=1,\dotsc,q+2$.

Next, we apply the usual flips to $\mathcal{Z}$ along with the $(-1)$-curves that intersect with $\widetilde{A}_1, \dotsc, \widetilde{A}_{p+1}$ in this order. Then the leftmost Wahl singularity $[p+4,2,\dotsc,2]$ of $Z$ is resolved so that we have a new deformation, denoted again by $\mathcal{Z} \to \Delta$ for simplicity, whose central fiber contains
\begin{equation*}
\begin{tikzpicture}[scale=0.75]
\node[bullet] (01) at (0,1) [label=left:{$\widetilde{A}_1'$}] {};
\node[bullet] (0-05) at (0,-0.5) [label=left:{$\widetilde{A}_{p+1}'$}] {};
\node[bullet] (0-1) at (0,-1) [label=left:{$\widetilde{A}_{p+2}$}] {};

\node[bullet] (05-05) at (0.5,-0.5) [labelBelow={\color{red}{$-1$}},label=right:{\color{red}{$E$}}] {};

\node[bullet] (10) at (1,0) [labelAbove={$-3$},label=below:{$G$}] {};

\node[bullet] (20) at (2,0) [labelAbove={$-2$},label=below:{$G_{p}$}] {};

\node[empty] (250) at (2.5,0) []{};
\node[empty] (30) at (3,0) []{};

\node[bullet] (350) at (3.5,0) [labelAbove={$-2$},label=below:{$G_1$}] {};

\node[bullet] (450) at (4.5,0) [labelAbove={$-1$},label=below:{$E'$}] {};

\node[empty] (50) at (5,0) []{};
\node[empty] (550) at (5.5,0) []{};

\draw[dotted] (0,0.85)--(0,-0.35);

\draw[-] (01)--(10);
\draw[-] (0-05)--(10);
\draw[-] (0-1)--(05-05)--(10);

\draw[-] (10)--(20);

\draw[-] (20)--(250);
\draw[dotted] (250)--(30);
\draw[-] (30)--(350);

\draw[-] (350)--(450);

\draw[-] (450)--(50);
\draw[dotted] (50)--(550);
\end{tikzpicture}
\end{equation*}

The degenerations occur as follows:
\begin{align*}
\widetilde{A}_1 &\rightsquigarrow \widetilde{A}_1' + G + G_p + \dotsb + G_1 \\
\widetilde{A}_2 &\rightsquigarrow \widetilde{A}_2' + G + G_p + \dotsb + G_2 \\
\vdots \\
\widetilde{A}_p &\rightsquigarrow \widetilde{A}_1' + G + G_p \\
\widetilde{A}_{p+1} &\rightsquigarrow \widetilde{A}_1' + G
\end{align*}
Therefore the red $(-1)$-curve $E$ induces the following gray column of the incidence matrix $M_3$:

\begin{center}
\includegraphics{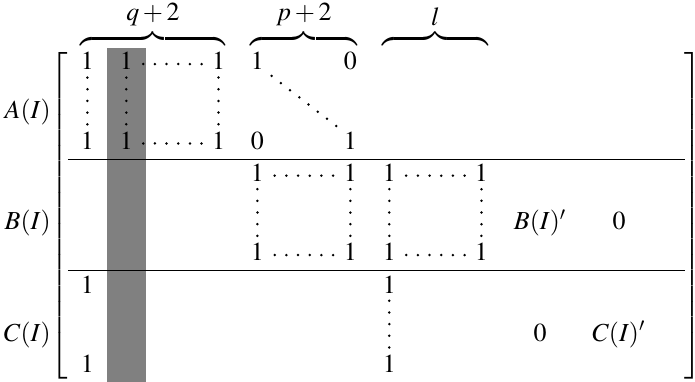}
\end{center}


We then apply a divisorial contraction to $E$. If there is another Wahl singularity $[p+4,2,\dotsc,2]$ on the $(-1)$-curve $E'$, we apply the usual flips beginning with $E'$ to the flipped images of $G_1, \dotsc, G_p$, and $G$ in succession. The sequence of flips then resolves the Wahl singularity, resulting in a new configuration:
\begin{equation*}
\begin{tikzpicture}[scale=0.75]
\node[bullet] (01) at (0,1) [label=left:{$\widetilde{A}_1'$}] {};
\node[bullet] (0-05) at (0,-0.5) [label=left:{$\widetilde{A}_{p+1}'$}] {};
\node[bullet] (0-1) at (0,-1) [label=left:{$\widetilde{A}_{p+2}$}] {};

\node[bullet] (10) at (1,0) [labelAbove={\color{red}{$-1$}},label=below:{\color{red}{$E$}}] {};

\node[bullet] (20) at (2,0) [labelAbove={$-3$},label=below:{$H$}] {};

\node[bullet] (30) at (3,0) [labelAbove={$-2$},label=below:{$H_{p}$}] {};

\node[empty] (350) at (3.5,0) []{};
\node[empty] (40) at (4,0) []{};

\node[bullet] (450) at (4.5,0) [labelAbove={$-2$},label=below:{$H_1$}] {};

\node[bullet] (550) at (5.5,0) [labelAbove={$-1$},label=below:{$E'$}] {};

\node[empty] (60) at (6,0) []{};
\node[empty] (650) at (6.5,0) []{};

\draw[dotted] (0,0.85)--(0,-0.35);

\draw[-] (01)--(10);
\draw[-] (0-05)--(10);
\draw[-] (0-1)--(10);

\draw[-] (10)--(20);
\draw[-] (20)--(30);

\draw[-] (30)--(350);
\draw[dotted] (350)--(40);
\draw[-] (40)--(450);

\draw[-] (450)--(550);

\draw[-] (550)--(60);
\draw[dotted] (60)--(650);
\end{tikzpicture}
\end{equation*}

The degenerations are given by
\begin{align*}
\widetilde{A}_1 &\rightsquigarrow \widetilde{A}_1' + E + H + H_p + \dotsb + H_2 + H_1 \\
\widetilde{A}_2 &\rightsquigarrow \widetilde{A}_2' + E + H + H_p + \dotsb + H_2 \\
\vdots \\
\widetilde{A}_p &\rightsquigarrow \widetilde{A}_1' + E + H + H_p \\
\widetilde{A}_{p+1} &\rightsquigarrow \widetilde{A}_1' + E + H
\end{align*}
So the red $(-1)$-curve $E$ induces the second gray column of $M_3$:

\begin{center}
\includegraphics{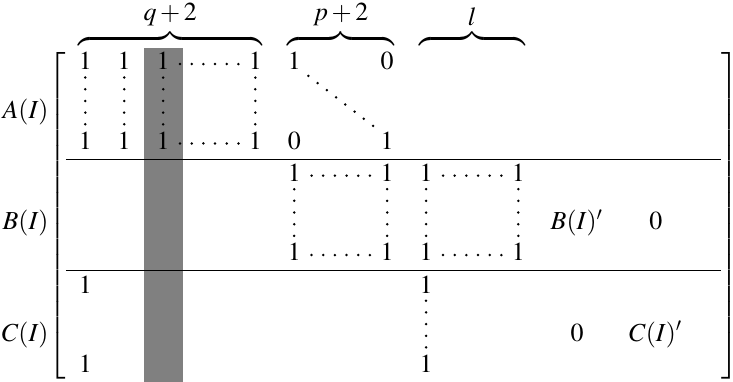}
\end{center}


Then we apply a divisorial contraction to the $(-1)$-curve $E$.

Repeating the above process until there is no Wahl singularity $[p+4,2,\dotsc,2]$, we obtain the following $(q+1)$ gray columns of $M_3$:

\begin{center}
\includegraphics{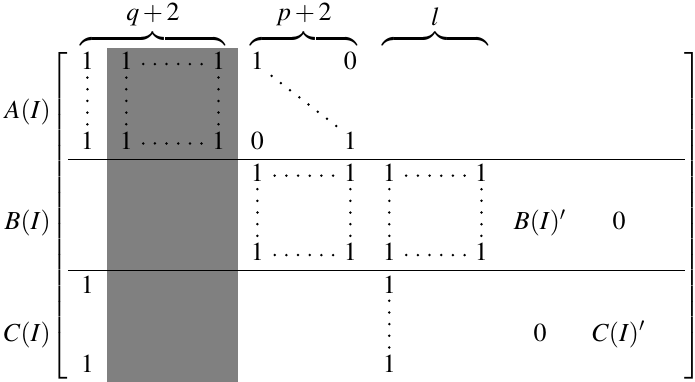}
\end{center}


After that, we have the configuration (combined with Equation~\eqref{equation:Wpqr-M(3)-after-BC-arms}):

\begin{equation*}
\begin{tikzpicture}[scale=0.75]

\node[bullet] (-4505) at (-4.5,0.5) [label=left:{$\widetilde{A}_1'$}] {};
\node[bullet] (-45-05) at (-4.5,-0.5) [label=left:{$\widetilde{A}_{p+1}'$}] {};
\node[bullet] (-45-1) at (-4.5,-1) [label=left:{$\widetilde{A}_{p+2}$}] {};

\node[bullet] (-350) at (-3.5,0) [labelAbove={$-2$},label=below:{\phantom{$-$}\ $H_{p+1}$}] {};
\node[empty] (-30) at (-3,0) []{};
\node[empty] (-250) at (-2.5,0) []{};
\node[bullet] (-20) at (-2,0) [labelAbove={$-2$},label=below:{$H_1$}] {};

\node[bullet] (-10) at (-1,0) [labelAbove={\color{red}{$-1$}},label=below:{\color{red}{$E$}}] {};

\draw[-] (00)--(-10);
\draw[-] (-10)--(-20);
\draw[-] (-20)--(-250);
\draw[dotted] (-250)--(-30);
\draw[-] (-30)--(-350);
\draw[-] (-350)--(-4505);
\draw[-] (-350)--(-45-05);
\draw[-] (-350)--(-45-1);
\draw[dotted] (-4.5,0.35)--(-4.5,-0.35);

\node[bullet] (00) at (0,0) [labelAbove={$-p-3$},label=below left:{$F$}] {};

\node[bullet] (105) at (1,0.5) [label=right:{$\widetilde{B}_1'$}] {};
\node[bullet] (1-05) at (1,-0.5) [label=right:{$\widetilde{B}_{q+2}'$}] {};

\node[bullet] (-05-1) at (-0.5,-1) [label=below:{$\widetilde{C}_1$}] {};
\node[bullet] (05-1) at (0.5,-1) [label=below:{$\widetilde{C}_{r+2}$}] {};

\draw[-] (-10)--(00);
\draw[-] (00)--(105);
\draw[-] (00)--(1-05);
\draw[-] (00)--(-05-1);
\draw[-] (00)--(05-1);

\draw[dotted] (1,0.35)--(1,-0.35);
\draw[dotted] (0.35,-1)--(-0.35,-1);

\end{tikzpicture}
\end{equation*}
and the degenerations are given as follows:
\begin{align*}
\widetilde{A}_1 &\rightsquigarrow \widetilde{A}_1' + H_{p+1} + \dotsb + H_2 + H_1 \\
\widetilde{A}_2 &\rightsquigarrow \widetilde{A}_2' + H_{p+1} + \dotsb + H_2 \\
\vdots \\
\widetilde{A}_{p+1} &\rightsquigarrow \widetilde{A}_1' + H_{p+1} \\
\widetilde{B}_1 &\rightsquigarrow \widetilde{B}_1' + F \\
\vdots \\
\widetilde{B}_{q+2} &\rightsquigarrow \widetilde{B}_{q+2}' + F
\end{align*}

The red $(-1)$-curve $E$ induces the column $[1\ 0 \dotsb 0]^T$ of $M_3$. After divisorially contracting it, then the curve $H_1$ becomes another $(-1)$-curve, which gives us the column $[0\ 1 \ 0 \dotsb 0]^T$ of $M_3$. Repeating this process until the curve $F$ becomes a $(-1)$-curve $F'$, we obtain all columns of $M_3$ except the first one; but it is induced by the last $(-1)$-curve $F'$.
\end{proof}

\begin{remark}
In Figure~\ref{figure:Wpqr-P-for-M(3)-another}, we present another P-resolution for $M(3)$ with $B=B(I)$ for $r-p \ge 1$, $q \ge 1$. But we need one more condition $p \ge 1$ in order to guarantee that the canonical divisor is ample. One can also construct a P-resolution for $M(3)$ with $B(II)$ in a similar fashion.
\end{remark}

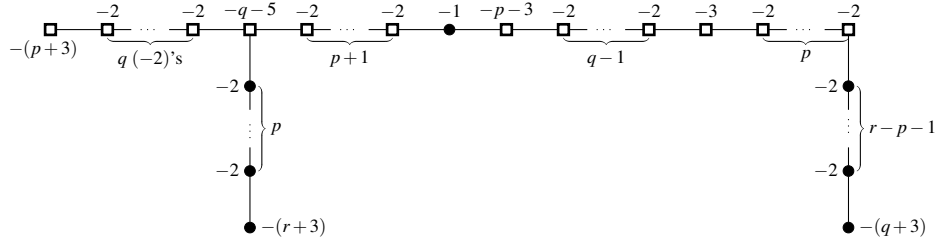
\begin{figure}
\centering
\begin{tikzpicture}[scale=0.75]
\tikzset{font = \scriptsize}
\node[rectangle] (00) at (0,0) [label=above:{$-q-5$}] {};


\node[rectangle] (-350) at (-3.5,0) [label=below:{$-(p+3)\phantom{-}$}] {};
\node[rectangle] (-250) at (-2.5,0) [label=above:{$-2$}] {};
\node[empty] (-20) at (-2,0) [] {};
\node[empty] (-150) at (-1.5,0) [] {};
\node[rectangle] (-10) at (-1,0) [label=above:{$-2$}] {};

\draw[-] (-350)--(-250);
\draw[-] (-250)--(-20);
\draw[dotted] (-20)--(-150);
\draw[-] (-150)--(-10);
\draw[-] (-10)--(00);
\draw [decorate, decoration = {calligraphic brace,mirror}] (-2.5,-0.15)--(-1,-0.15)  node[pos=0.5,below=0.1em,black]{$q$ $(-2)$'s};


\node[rectangle] (10) at (1,0) [label=above:{$-2$}] {}; 
\node[empty] (150) at (1.5,0) [] {};
\node[empty] (20) at (2,0) [] {};
\node[rectangle] (250) at (2.5,0) [label=above:{$-2$}] {}; 

\node[bullet] (350) at (3.5,0) [label=above:{$-1$}] {};

\node[rectangle] (450) at (4.5,0) [label=above:{$-p-3$}] {};

\node[rectangle] (550) at (5.5,0) [label=above:{$-2$}] {};  
\node[empty] (60) at (6,0) [] {};
\node[empty] (650) at (6.5,0) [] {};
\node[rectangle] (70) at (7,0) [label=above:{$-2$}] {}; 

\node[rectangle] (80) at (8,0) [label=above:{$-3$}] {};

\node[rectangle] (90) at (9,0) [label=above:{$-2$}] {};  
\node[empty] (950) at (9.5,0) [] {};
\node[empty] (100) at (10,0) [] {};
\node[rectangle] (1050) at (10.5,0) [label=above:{$-2$}] {}; 

\node[bullet] (105-1) at (10.5,-1) [label=left:{$-2$}] {};  
\node[empty] (105-15) at (10.5,-1.5) [] {};
\node[empty] (105-2) at (10.5,-2) [] {};
\node[bullet] (105-25) at (10.5,-2.5) [label=left:{$-2$}] {}; 

\node[bullet] (105-35) at (10.5,-3.5) [label=right:{$-(q+3)\phantom{-}$}] {};

\draw[-] (00)--(10);

\draw[-] (10)--(150);
\draw[dotted] (150)--(20);
\draw[-] (20)--(250);
\draw [decorate, decoration = {calligraphic brace,mirror}] (1,-0.15)--(2.5,-0.15)  node[pos=0.5,below=0.1em,black]{$p+1$};

\draw[-] (250)--(350);
\draw[-] (350)--(450);

\draw[-] (450)--(550);
\draw[-] (550)--(60);
\draw[dotted] (60)--(650);
\draw[-] (650)--(70);
\draw [decorate, decoration = {calligraphic brace,mirror}] (5.5,-0.15)--(7,-0.15)  node[pos=0.5,below=0.1em,black]{$q-1$};

\draw[-] (70)--(80);

\draw[-] (80)--(90);
\draw[-] (90)--(950);
\draw[dotted] (950)--(100);
\draw[-] (100)--(1050);
\draw [decorate, decoration = {calligraphic brace,mirror}] (9,-0.15)--(10.5,-0.15)  node[pos=0.5,below=0.1em,black]{$p$};

\draw[-] (1050)--(105-1);
\draw[-] (105-1)--(105-15);
\draw[dotted] (105-15)--(105-2);
\draw[-] (105-2)--(105-25);
\draw [decorate, decoration = {calligraphic brace, mirror}] (10.65,-2.5)--(10.65,-1)  node[pos=0.5,right=0.1em,black]{$r-p-1$};

\draw[-] (105-25)--(105-35);


\node[bullet] (0-35) at (0,-3.5) [label=right:{$-(r+3)\phantom{-}$}] {};
\node[bullet] (0-25) at (0,-2.5) [label=left:{$-2$}] {};
\node[empty] (0-2) at (0,-2) [] {};
\node[empty] (0-15) at (0,-1.5) [] {};
\node[bullet] (0-1) at (0,-1) [label=left:{$-2$}] {};

\draw[-] (0-35)--(0-25);
\draw[-] (0-25)--(0-2);
\draw[dotted] (0-2)--(0-15);
\draw[-] (0-15)--(0-1);
\draw[-] (0-1)--(00);
\draw [decorate, decoration = {calligraphic brace,mirror}] (0.15,-2.5)--(0.15,-1)  node[pos=0.5,right=0.1em,black]{$p$};
\end{tikzpicture}
\caption{Another P-resolution for $M(3)$}
\label{figure:Wpqr-P-for-M(3)-another}
\end{figure}

\subsection{The P-modifications for the incidence matrices $M(6)$ and $M(7)$}

A smoothing of $W(p,q,r)$ whose Milnor fiber is a rational homology disk occurs on a one-dimensional smoothing component (called a \emph{$\mathbb{Q}$HD component}) of the deformation space of $W(p,q,r)$ and that such a smoothing can be chosen to be $\mathbb{Q}$-Gorenstein; Wahl~\cite[Theorem~3.4]{Wahl-2013}. According to Fowler~\cite{Fowler-2013} (refer also Wahl~\cite{Wahl-2022}), there is only one $\mathbb{Q}$HD component of $W(p,q,r)$ for any $p,q,r$ except $p=q=r$, whose corresponding ample embedding of the compactifying divisor of $W(p,q,r)$ is given in Figure~\ref{figure:W(pqr)-QHDS-I}. If $p=q=r$, there is \emph{another} $\mathbb{Q}$HD component, whose corresponding embedding is given in Figure~\ref{figure:W(pqr)-QHDS-II}.

\begin{figure}
\centering
\begin{tikzpicture}
\tikzset{font=\scriptsize}

\node[bullet] (00) at (0,0) [labelBelow={$+1$}] {};


\node[bullet] (-10) at (-1,0) [labelBelow={$-(p+2)$}] {};
\node[bullet] (-20) at (-2,0) [] {};

\node[empty] (-250) at (-2.5,0) [] {};
\node[empty] (-30) at (-3,0) [] {};

\node[bullet] (-350) at (-3.5,0) [] {};

\draw[-] (00)--(-10);
\draw[-] (-10)--(-20);
\draw[-] (-20)--(-250);
\draw[dotted] (-250)--(-30);
\draw[-] (-30)--(-350);

\draw [decorate, decoration = {calligraphic brace,mirror}] (-2,0.15)--(-3.5,0.15)  node[pos=0.5,above]{$r+1$};


\node[bullet] (01) at (0,1) [label=left:{$-(q+2)$}] {};
\node[bullet] (02) at (0,2) [] {};

\node[empty] (025) at (0,2.5) [] {};
\node[empty] (03) at (0,3) [] {};

\node[bullet] (035) at (0,3.5) [] {};

\draw[-] (00)--(01);
\draw[-] (01)--(02);
\draw[-] (02)--(025);
\draw[dotted] (025)--(03);
\draw[-] (03)--(035);

\draw [decorate, decoration = {calligraphic brace,mirror}] (0.15,2)--(0.15,3.5)  node[pos=0.5,right]{$p+1$};


\node[bullet] (10) at (1,0) [labelAbove={$-(r+2)$}] {};
\node[bullet] (20) at (2,0) [] {};

\node[empty] (250) at (2.5,0) [] {};
\node[empty] (30) at (3,0) [] {};

\node[bullet] (350) at (3.5,0) [] {};

\draw[-] (00)--(10);
\draw[-] (10)--(20);
\draw[-] (20)--(250);
\draw[dotted] (250)--(30);
\draw[-] (30)--(350);

\draw [decorate, decoration = {calligraphic brace,mirror}] (2,-0.15)--(3.5,-0.15)  node[pos=0.5,below]{$q+1$};

\node[circle] (-135) at (-1,3.5) [labelAbove={$-1$}] {};
\node[circle] (-35-1) at (-3.5,-1) [labelBelow={$-1$}] {};
\node[circle] (351) at (3.5,1) [labelAbove={$-1$}] {};

\draw[-] (035)--(-135);
\draw[-] (-135)--(-10);

\draw[-] (-350)--(-35-1);
\draw[-] (-35-1)--(1,-1);
\draw[-] (1,-1)--(10);

\draw[-] (350)--(351);
\draw[-] (351)--(01);
\end{tikzpicture}

\caption{The $\mathbb{Q}$HD embedding of the compactifying divisor of $W(p,q,r)$}
\label{figure:W(pqr)-QHDS-I}
\end{figure}
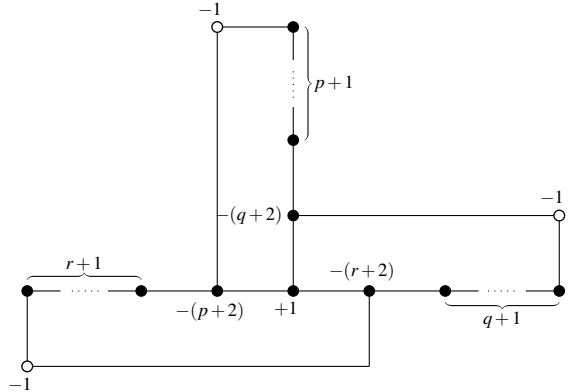

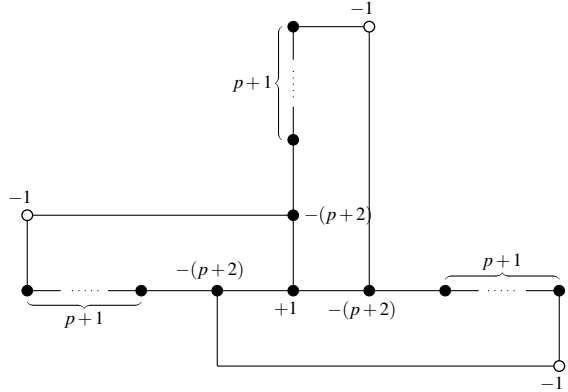
\begin{figure}

\begin{tikzpicture}
\tikzset{font=\scriptsize}

\node[bullet] (00) at (0,0) [labelBelow={$+1$}] {};


\node[bullet] (-10) at (-1,0) [labelAbove={$-(p+2)$}] {};
\node[bullet] (-20) at (-2,0) [] {};

\node[empty] (-250) at (-2.5,0) [] {};
\node[empty] (-30) at (-3,0) [] {};

\node[bullet] (-350) at (-3.5,0) [] {};

\draw[-] (00)--(-10);
\draw[-] (-10)--(-20);
\draw[-] (-20)--(-250);
\draw[dotted] (-250)--(-30);
\draw[-] (-30)--(-350);

\draw [decorate, decoration = {calligraphic brace}] (-2,-0.15)--(-3.5,-0.15)  node[pos=0.5,below]{$p+1$};


\node[bullet] (01) at (0,1) [label=right:{$-(p+2)$}] {};
\node[bullet] (02) at (0,2) [] {};

\node[empty] (025) at (0,2.5) [] {};
\node[empty] (03) at (0,3) [] {};

\node[bullet] (035) at (0,3.5) [] {};

\draw[-] (00)--(01);
\draw[-] (01)--(02);
\draw[-] (02)--(025);
\draw[dotted] (025)--(03);
\draw[-] (03)--(035);

\draw [decorate, decoration = {calligraphic brace}] (-0.15,2)--(-0.15,3.5)  node[pos=0.5,left]{$p+1$};


\node[bullet] (10) at (1,0) [labelBelow={$-(p+2)$}] {};
\node[bullet] (20) at (2,0) [] {};

\node[empty] (250) at (2.5,0) [] {};
\node[empty] (30) at (3,0) [] {};

\node[bullet] (350) at (3.5,0) [] {};

\draw[-] (00)--(10);
\draw[-] (10)--(20);
\draw[-] (20)--(250);
\draw[dotted] (250)--(30);
\draw[-] (30)--(350);

\draw [decorate, decoration = {calligraphic brace}] (2,0.15)--(3.5,0.15)  node[pos=0.5,above]{$p+1$};

\node[circle] (135) at (1,3.5) [labelAbove={$-1$}] {};
\node[circle] (-351) at (-3.5,1) [labelAbove={$-1$}] {};
\node[circle] (35-1) at (3.5,-1) [labelBelow={$-1$}] {};

\draw[-] (035)--(135);
\draw[-] (135)--(10);

\draw[-] (-350)--(-351);
\draw[-] (-351)--(0,1);

\draw[-] (350)--(35-1);
\draw[-] (35-1)--(-1,-1);
\draw[-] (-1,-1)--(-10);
\end{tikzpicture}

\caption{Another $\mathbb{Q}$HD embedding of the compactifying divisor of $W(p,q,r)$ for $p=q=r$}
\label{figure:W(pqr)-QHDS-II}
\end{figure}

\begin{proposition}\label{proposition:Wpqr-P-modifications-M(6)}
There is a P-modification of $W(p,q,r)$ that corresponds to the (combinatorial) incidence matrix $M(6)$. In case $p=q=r$, there is another P-modification that corresponds to $M(7)$.
\end{proposition}

\begin{proof}
For any $p, q, r$ except $p=q=r$, $M(6)$ is the only combinatorial incidence matrix whose Milnor number is zero. Therefore the $\mathbb{Q}$-Gorenstein smoothing occurring on the unique $\mathbb{Q}$HD component must be the P-modification corresponding to $M(6)$.

Suppose that $p=q=r$. There are two $\mathbb{Q}$HD components. One can check that the homology matrices $\Delta^+ M(6)$ and $\Delta^+ M(7)$ correspond to the smoothings of negative weights given in Figure~\ref{figure:W(pqr)-QHDS-I} and Figure~\ref{figure:W(pqr)-QHDS-II}, respectively. Therefore, the $\mathbb{Q}$-Gorenstein smoothings of $W(p,p,p)$ associated to the smoothings of negative weights in given in Figure~\ref{figure:W(pqr)-QHDS-I} and Figure~\ref{figure:W(pqr)-QHDS-II} are the P-modifications correspond to $M(6)$ and $M(7)$, respectively.
\end{proof}

\begin{remark}
One can obtain $M(6)$ from the rational homology disk smoothing of $W(p,q,r)$ via the semi-stable MMP. The key point is that there is a flip along a $(-1)$-curve intersecting the $(-p-3)$-curve. In Figure~\ref{figure:Wpqr-flip}, we describe how the central fiber is changed after applying a flip along the flipping curve $E^-$. In a similar fashion, one can describe the flip for $M(7)$ in case $p=q=r$.
\end{remark}

\begin{figure}
\centering
\begin{tikzpicture}[scale=0.75]
\tikzset{font=\scriptsize}
\begin{scope}[shift={(0,0)}]
\node[rectangle] (00) at (0,0) [labelAbove={$-4$}] {};


\node[bullet] (-450) at (-4.5,0) [labelAbove={$-1$},label=below:{$E^-$}]{};

\node[rectangle] (-350) at (-3.5,0) [labelAbove={$-(p+3)$}] {};
\node[rectangle] (-250) at (-2.5,0) [labelAbove={$-2$}] {};
\node[empty] (-20) at (-2,0) [] {};
\node[empty] (-150) at (-1.5,0) [] {};
\node[rectangle] (-10) at (-1,0) [labelAbove={$-2$}] {};

\draw[-] (-450)--(-350);
\draw[-] (-350)--(-250);
\draw[-] (-250)--(-20);
\draw[dotted] (-20)--(-150);
\draw[-] (-150)--(-10);
\draw[-] (-10)--(00);
\draw [decorate, decoration = {calligraphic brace,mirror}] (-2.5,-0.15)--(-1,-0.15)  node[pos=0.5,below=0.1em,black]{$q$};


\node[rectangle] (350) at (3.5,0) [labelAbove={$-(q+3)$}] {};
\node[rectangle] (250) at (2.5,0) [labelAbove={$-2$}] {};
\node[empty] (20) at (2,0) [] {};
\node[empty] (150) at (1.5,0) [] {};
\node[rectangle] (10) at (1,0) [labelAbove={$-2$}] {};

\draw[-] (350)--(250);
\draw[-] (250)--(20);
\draw[dotted] (20)--(150);
\draw[-] (150)--(10);
\draw[-] (10)--(00);
\draw [decorate, decoration = {calligraphic brace,mirror}] (1,-0.15)--(2.5,-0.15)  node[pos=0.5,below=0.1em,black]{$r$};


\node[rectangle] (0-35) at (0,-3.5) [label=left:{$-(r+3)$}] {};
\node[rectangle] (0-25) at (0,-2.5) [label=left:{$-2$}] {};
\node[empty] (0-2) at (0,-2) [] {};
\node[empty] (0-15) at (0,-1.5) [] {};
\node[rectangle] (0-1) at (0,-1) [label=left:{$-2$}] {};

\draw[-] (0-35)--(0-25);
\draw[-] (0-25)--(0-2);
\draw[dotted] (0-2)--(0-15);
\draw[-] (0-15)--(0-1);
\draw[-] (0-1)--(00);
\draw [decorate, decoration = {calligraphic brace,mirror}] (0.15,-2.5)--(0.15,-1)  node[pos=0.5,right=0.1em,black]{$p$};
\end{scope}


\begin{scope}[shift={(0,-5)}]

\draw[->] (-5.5,0)--(-4.5,0);
\node at (-5,0) [label=above:{flip}] {};

\node[rectangle] (00) at (0,0) [labelAbove={$-4$}] {};


\node[rectangle] (-350) at (-3.5,0) [labelAbove={$-(p+2)$}] {};
\node[rectangle] (-250) at (-2.5,0) [labelAbove={$-2$}] {};
\node[empty] (-20) at (-2,0) [] {};
\node[empty] (-150) at (-1.5,0) [] {};
\node[rectangle] (-10) at (-1,0) [labelAbove={$-2$}] {};

\draw[-] (-350)--(-250);
\draw[-] (-250)--(-20);
\draw[dotted] (-20)--(-150);
\draw[-] (-150)--(-10);
\draw[-] (-10)--(00);
\draw [decorate, decoration = {calligraphic brace,mirror}] (-2.5,-0.15)--(-1,-0.15)  node[pos=0.5,below=0.1em,black]{$q$};


\node[rectangle] (350) at (3.5,0) [labelBelow={$-(q+3)$}] {};
\node[rectangle] (250) at (2.5,0) [labelAbove={$-2$}] {};
\node[empty] (20) at (2,0) [] {};
\node[empty] (150) at (1.5,0) [] {};
\node[rectangle] (10) at (1,0) [labelAbove={$-2$}] {};

\draw[-] (350)--(250);
\draw[-] (250)--(20);
\draw[dotted] (20)--(150);
\draw[-] (150)--(10);
\draw[-] (10)--(00);
\draw [decorate, decoration = {calligraphic brace,mirror}] (1,-0.15)--(2.5,-0.15)  node[pos=0.5,below=0.1em,black]{$r$};


\node[rectangle] (0-8) at (0,-8) [label=left:{$-(r+4)$}] {};
\node[rectangle] (0-7) at (0,-7) [label=left:{$-2$}] {};
\node[empty] (0-65) at (0,-6.5) [] {};
\node[empty] (0-6) at (0,-6) [] {};
\node[rectangle] (0-55) at (0,-5.5) [label=left:{$-2$}] {};

\node[bullet] (0-45) at (0,-4.5) [label=left:{$-1$},label=right:{$E^+$}] {};

\node[rectangle] (0-35) at (0,-3.5) [label=left:{$-(r+3)$}] {};
\node[rectangle] (0-25) at (0,-2.5) [label=left:{$-2$}] {};
\node[empty] (0-2) at (0,-2) [] {};
\node[empty] (0-15) at (0,-1.5) [] {};
\node[rectangle] (0-1) at (0,-1) [label=left:{$-2$}] {};

\draw[-] (0-7)--(0-8);
\draw[-] (0-65)--(0-7);
\draw[dotted] (0-6)--(0-65);
\draw[-] (0-55)--(0-6);
\draw[-] (0-45)--(0-55);

\draw [decorate, decoration = {calligraphic brace,mirror}] (0.15,-7)--(0.15,-5.5)  node[pos=0.5,right=0.1em,black]{$r$};

\draw[-] (0-35)--(0-45);

\draw[-] (0-35)--(0-25);
\draw[-] (0-25)--(0-2);
\draw[dotted] (0-2)--(0-15);
\draw[-] (0-15)--(0-1);
\draw[-] (0-1)--(00);
\draw [decorate, decoration = {calligraphic brace,mirror}] (0.15,-2.5)--(0.15,-1)  node[pos=0.5,right=0.1em,black]{$p-1$};
\end{scope}
\end{tikzpicture}

\caption{A flip of the rational homology disk smoothing of $W(p,q,r)$}
\label{figure:Wpqr-flip}
\end{figure}
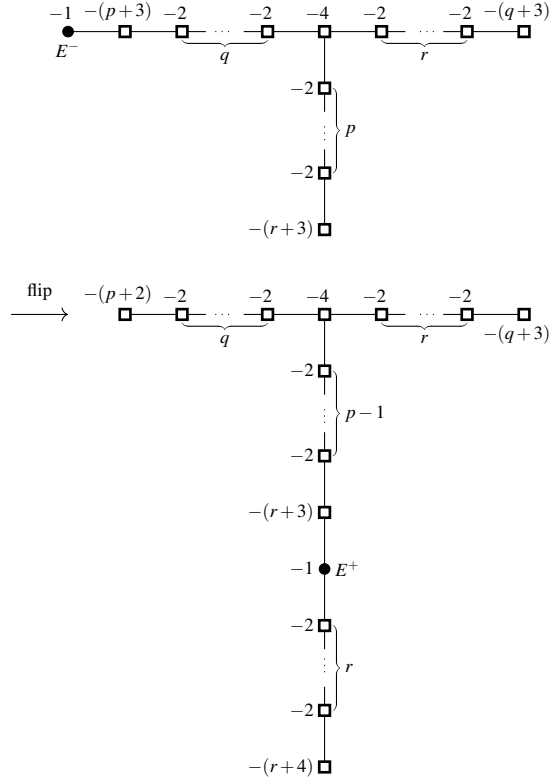

\begin{remark}
With Jaekwan Jeon, the authors prove Kollár conjecture for all weighted homogeneous surface singularities admitting rational homology disk smoothings in JPS~\cite{Jeon-Park-Shin-2022} in a similar fashion. Due to the lengthy arguments, however, we decide to write a second paper.
\end{remark}

\section{Miscellaneous examples}\label{section:misc-examples}

We show through several examples that the results obtained in this paper can be applied to various cases.

\subsection{An example of using a non-usual flip}\label{section:example-non-usual-flip}

Let $(X,p)$ be a weighted homogeneous surface singularity with the dual graph
\begin{equation}\label{equation:WHSS-non-usual-flip}
\begin{aligned}
\begin{tikzpicture}[scale=0.75]
\node[bullet] (-10) at (-1,0) [labelAbove={$-3$}] {};
\node[bullet] (00) at (0,0) [labelAbove={$-5$}] {};
\node[bullet] (10) at (1,0) [labelAbove={$-2$}] {};

\node[bullet] (0-1) at (0,-1) [labelBelow={$-2$}] {};

\draw[-] (-10)--(00)--(10);
\draw[-] (0-1)--(00);
\end{tikzpicture}
\end{aligned}
\end{equation}

We may choose its sandwiched structure as in Figure~\ref{figure:non-usual-flip}(a) in a non-usual way. The decorated curve $(C,l)$ consists of four plane curves $C_1, C_2, C_3, C_4$ whose equations are given by $y^2=a_ix^3$ for $i=1, 2, 3, 4$ and of the decorations given by  $l_1=l_2=l_3=5$ and $l_4=6$. Notice that $C_i \cdot C_j = 6$ for all $i, j$. One can verify that $(X,p)$ has four incidence matrices $A_1, \dotsc, A_4$ and four M-resolutions $M_1, \dotsc, M_4$, which are given in Figure~\ref{figure:non-usual-flip-Fours}.

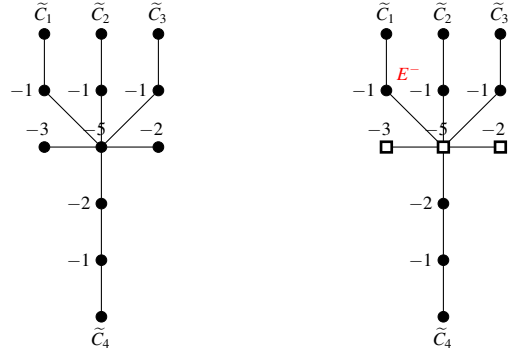
\begin{figure}
\centering
\begin{tikzpicture}[scale=0.75]
\begin{scope}[shift={(0,0)}]
\node[bullet] (0-3) at (0,-3) [label=below:{$\widetilde{C}_4$}] {};

\node[bullet] (0-2) at (0,-2) [label=left:{$-1$}] {};

\node[bullet] (0-1) at (0,-1) [label=left:{$-2$}] {};

\node[bullet] (-10) at (-1,0) [labelAbove={$-3$}] {};
\node[bullet] (00) at (0,0) [labelAbove={$-5$}] {};
\node[bullet] (10) at (1,0) [labelAbove={$-2$}] {};

\node[bullet] (-11) at (-1,1) [label=left:{$-1$}] {};
\node[bullet] (01) at (0,1) [label=left:{$-1$}] {};
\node[bullet] (11) at (1,1) [label=left:{$-1$}] {};

\node[bullet] (-12) at (-1,2) [label=above:{$\widetilde{C}_1$}] {};
\node[bullet] (02) at (0,2) [label=above:{$\widetilde{C}_2$}] {};
\node[bullet] (12) at (1,2) [label=above:{$\widetilde{C}_3$}] {};

\draw[-] (0-3)--(0-2)--(0-1)--(00);

\draw[-] (-10)--(00)--(10);

\draw[-] (00)--(-11);
\draw[-] (00)--(01);
\draw[-] (00)--(11);

\draw[-] (-11)--(-12);
\draw[-] (01)--(02);
\draw[-] (11)--(12);

\node[] at (0,-4.5) [] {(a) A sandwiched structure};
\end{scope}

\begin{scope}[shift={(6,0)}]
\node[bullet] (0-3) at (0,-3) [label=below:{$\widetilde{C}_4$}] {};

\node[bullet] (0-2) at (0,-2) [label=left:{$-1$}] {};

\node[bullet] (0-1) at (0,-1) [label=left:{$-2$}] {};

\node[rectangle] (-10) at (-1,0) [labelAbove={$-3$}] {};
\node[rectangle] (00) at (0,0) [labelAbove={$-5$}] {};
\node[rectangle] (10) at (1,0) [labelAbove={$-2$}] {};

\node[bullet] (-11) at (-1,1) [label=left:{$-1$},label=above right:{\color{red}{$E^-$}}] {};
\node[bullet] (01) at (0,1) [label=left:{$-1$}] {};
\node[bullet] (11) at (1,1) [label=left:{$-1$}] {};

\node[bullet] (-12) at (-1,2) [label=above:{$\widetilde{C}_1$}] {};
\node[bullet] (02) at (0,2) [label=above:{$\widetilde{C}_2$}] {};
\node[bullet] (12) at (1,2) [label=above:{$\widetilde{C}_3$}] {};

\draw[-] (0-3)--(0-2)--(0-1)--(00);

\draw[-] (-10)--(00)--(10);

\draw[-] (00)--(-11);
\draw[-] (00)--(01);
\draw[-] (00)--(11);

\draw[-] (-11)--(-12);
\draw[-] (01)--(02);
\draw[-] (11)--(12);

\node[] at (0,-4.5) [] {(b) The compactification $Z$};
\end{scope}
%
%
%
%
%
%
%
%
%
%
%
%
\end{tikzpicture}
\caption{An example of using a non-usual flip}
\label{figure:non-usual-flip}
\end{figure}

\begin{figure}
\[\begin{aligned}
A_1 &= \begin{bmatrix}
2 & 1 & 1 & 1 & 0 & 0 & 0 \\
2 & 1 & 1 & 0 & 1 & 0 & 0 \\
2 & 1 & 1 & 0 & 0 & 1 & 0 \\
2 & 1 & 1 & 0 & 0 & 0 & 1
\end{bmatrix} &
A_2 &= \begin{bmatrix}
2 & 1 & 1 & 1 & 0 & 0 \\
2 & 1 & 1 & 0 & 1 & 0 \\
2 & 1 & 0 & 1 & 1 & 0 \\
2 & 0 & 1 & 1 & 1 & 1
\end{bmatrix}\\
A_3 &= \begin{bmatrix}
2 & 1 & 1 & 1 & 0 & 0 \\
2 & 1 & 1 & 0 & 1 & 0 \\
2 & 1 & 1 & 0 & 1 & 1 \\
1 & 2 & 1 & 1 & 1 & 1
\end{bmatrix} &
A_4 &= \begin{bmatrix}
2 & 1 & 1 & 1 & 0 \\
1 & 2 & 1 & 1 & 0 \\
1 & 1 & 2 & 1 & 0 \\
1 & 1 & 1 & 2 & 1
\end{bmatrix}
\end{aligned}\]

\centering

\begin{tikzpicture}[scale=0.75]
\begin{scope}[shift={(0,0)}]
\node[empty] (-11) at (-1,1) []{$M_1$};

\node[bullet] (-10) at (-1,0) [labelAbove={$-3$}] {};
\node[bullet] (00) at (0,0) [labelAbove={$-5$}] {};
\node[bullet] (10) at (1,0) [labelAbove={$-2$}] {};

\node[bullet] (0-1) at (0,-1) [labelBelow={$-2$}] {};

\draw[-] (-10)--(00)--(10);
\draw[-] (0-1)--(00);
\end{scope}

\begin{scope}[shift={(4,0)}]
\node[empty] (-11) at (-1,1) []{$M_2$};

\node[bullet] (-10) at (-1,0) [labelAbove={$-3$}] {};
\node[rectangle] (00) at (0,0) [labelAbove={$-5$}] {};
\node[rectangle] (10) at (1,0) [labelAbove={$-2$}] {};

\node[bullet] (0-1) at (0,-1) [labelBelow={$-2$}] {};

\draw[-] (-10)--(00)--(10);
\draw[-] (0-1)--(00);
\end{scope}

\begin{scope}[shift={(8,0)}]
\node[empty] (-11) at (-1,1) []{$M_3$};

\node[bullet] (-10) at (-1,0) [labelAbove={$-3$}] {};
\node[rectangle] (00) at (0,0) [labelAbove={$-5$}] {};
\node[bullet] (10) at (1,0) [labelAbove={$-2$}] {};

\node[rectangle] (0-1) at (0,-1) [labelBelow={$-2$}] {};

\draw[-] (-10)--(00)--(10);
\draw[-] (0-1)--(00);
\end{scope}

\begin{scope}[shift={(12,0)}]
\node[empty] (-11) at (-1,1) []{$M_4$};

\node[rectangle] (-10) at (-1,0) [labelAbove={$-3$}] {};
\node[rectangle] (00) at (0,0) [labelAbove={$-5$}] {};
\node[rectangle] (10) at (1,0) [labelAbove={$-2$}] {};

\node[bullet] (0-1) at (0,-1) [labelBelow={$-2$}] {};

\draw[-] (-10)--(00)--(10);
\draw[-] (0-1)--(00);
\end{scope}
\end{tikzpicture}
\caption{Four incidence matrices and four M-resolutions}
\label{figure:non-usual-flip-Fours}
\end{figure}

One can check that the M-resolutions $M_1, M_2, M_3$ correspond to the incidence matrices $A_1, A_2, A_3$. We will show that the M-resolution $M_4$ corresponds to the incidence matrix $A_4$. Nonetheless, $M_4$ requires a non-usual flip.

\begin{lemma}\label{lemma:degeneration}
The mk1A $[3,\overline{5},2]$ is of flipping type and its flips is $[4]-1-[5,2]$. If $C$ is an irreducible curve intersecting with the flipping $(-1)$-curve $E^-$ of $[3,\overline{5},2]$, then, after the flip, the curve $C$ in a general fiber is degenerated to $C'+2E^+$ in the central fiber; that is,
\begin{equation*}
C \rightsquigarrow C'+2E^+
\end{equation*}
\end{lemma}

\begin{proof}
By Proposition~\ref{proposition:mk1A->mk2A}, the mk1A $[3,\overline{5},2]$ is converted to the mk2A $[3,5,2]-1-[4]$. One can check that this mk2A is an initial one by Proposition~\ref{proposition:flip-for-Mori-sequence}. By the same proposition, one can show that the flip of $[3,\overline{5},2]$ is the M-resolution $[4]-1-[5,2]$.

It remains to show that the degeneration coefficient $\beta$ (defined in Proposition~\ref{proposition:degeneration} is equal $2$.

As in Equation~\eqref{equation:degeneration}, we have
\begin{equation*}
K_Z \cdot C = K_{Z_t} \cdot = K_{Z^+_t} \cdot C = K_{Z^+} \cdot (C' + \beta E^+)
\end{equation*}
Let $\pi \colon \widetilde{Z} \to Z$ and $\pi^+ \colon \widetilde{Z}^+ \to Z^+$ be the minimal resolutions of the singularities in $Z$ and $Z^+$, respectively. At first we have
\begin{equation*}
K_Z \cdot C = K_{\widetilde{Z}} \cdot C.
\end{equation*}
Next, we have
\begin{align*}
K_{Z^+} \cdot C' &= (\pi^+)^{\ast}(K_{Z^+}) \cdot C' \\
K_{Z^+} \cdot E^+ &= (\pi^+)^{\ast}(K_{Z^+}) \cdot E^+
\end{align*}
From Figure~\ref{figure:non-usual-flip}(c), we have
\begin{equation*}
(\pi^+)^{\ast}(K_{Z^+}) = K_{\widetilde{Z}^+} + \frac{1}{2} E_1 + \frac{2}{3} E_2 + \frac{1}{3} E_3
\end{equation*}
Therefore
\begin{equation*}
K_{Z^+} \cdot (C' + \beta E^+) = K_{\widetilde{C}^+} \cdot C' - \frac{2}{3} + \beta \left(-1 + \frac{2}{3} + \frac{1}{2}  \right)
\end{equation*}
On the other hand, the birational morphism $\widetilde{Z} \dashrightarrow \widetilde{Z}^+$ is a composition of a blow-down followed by one blow-up. In particular, the morphism $C \to C'$ is just a blow-up. So
\begin{equation*}
K_{\widetilde{Z}} \cdot C = K_{\widetilde{Z}^+} \cdot C' - 1
\end{equation*}
Therefore, combining the above equations, we have $\beta=2$.
\end{proof}

Let $Z$ the compactified M-resolution of $M_4$, which is given in Figure~\ref{figure:non-usual-flip}(b). In Figure~\ref{figure:non-usual-flip-procedure}, we present the procedure for identifying the corresponding incidence matrix from $Z$ via the semi-stable MMP. However we draw only central fibers for simplicity. Each step in the figure is fairly straightforward. However, since there is a problem with multiplicities during degenerations caused by flips, we will explain each steps briefly.

In Step~1, the red $(-1)$-curve induces the column $[0\ 0\ 0\ 1]^T$ of $A_4$. We can apply the divisorial contraction to it. In Step~2, we apply a flip to the red $(-1)$-curve. By Lemma~\ref{lemma:degeneration}, we have Step~3. Notice that the flip induces the following degeneration
\begin{equation*}
\widetilde{C}_1 \rightsquigarrow \widetilde{C}_1' + 2(1)
\end{equation*}
where $(1)$ denotes the $(-1)$-curve in Step~3 and $2$ denotes its degeneration coefficient $\beta$.

We apply a usual flip to the red $(-1)$-curve in Step~3. In Step~4, the new degenerations occur:
\begin{align*}
\widetilde{C}_1 &\rightsquigarrow \widetilde{C}_1' + 2(1) + (2) \\
\widetilde{C}_2 &\rightsquigarrow \widetilde{C}_2' + (2)
\end{align*}
After applying a usual flip again, we get Step~5 where $\widetilde{C}_3$ is also degenerated:
\begin{align*}
\widetilde{C}_1 &\rightsquigarrow \widetilde{C}_1' + 2(1) + 2(3) + (2) \\
\widetilde{C}_2 &\rightsquigarrow \widetilde{C}_2' + (3) + (2) \\
\widetilde{C}_3 &\rightsquigarrow \widetilde{C}_3' + (3)
\end{align*}
Then the red $(-1)$-curve induces the column $[2\ 1\ 1\ 1]^T$ of the incidence matrix $A_4$.

Applying a divisorial contraction to the red $(-1)$-curve, we get the red $(-1)$-curve in Step~6, which can be flipped. In Step~7, we get a non-singular central fiber. Over there the degenerations are as follows.
\begin{align*}
\widetilde{C}_1 &\rightsquigarrow \widetilde{C}_1' + (3) + 2(1) + (2) \\
\widetilde{C}_2 &\rightsquigarrow \widetilde{C}_2' + (3) + (2) + (2) \\
\widetilde{C}_3 &\rightsquigarrow \widetilde{C}_3' + (3) + (1)
\end{align*}
Here the red $(-1)$-curve induces the column $[1\ 2\ 1\ 1]^T$ of $A_4$.

Step~8 is obtained by applying a divisorial contraction. Here we draw the curves as lines because we need how they are intersect with each other explicitly. Since the degenerations are given by
\begin{align*}
\widetilde{C}_1 &\rightsquigarrow \widetilde{C}_1' + (2) + (1) \\
\widetilde{C}_2 &\rightsquigarrow \widetilde{C}_2' + (2) + (1) \\
\widetilde{C}_3 &\rightsquigarrow \widetilde{C}_3' + (2)
\end{align*}
the red $(-1)$-curve in Step~8 induces the column $[1 \ 1 \ 2 \ 1]^T$ of $A_4$.

Finally, we again contract the $(-1)$-curve, then we have Step~9 with degenerations:
\begin{align*}
\widetilde{C}_1 &\rightsquigarrow \widetilde{C}_1' + (1) \\
\widetilde{C}_2 &\rightsquigarrow \widetilde{C}_2' + (1) \\
\widetilde{C}_3 &\rightsquigarrow \widetilde{C}_3' + (1)
\end{align*}
Notice that all curves $\widetilde{C}_i$ intersects with the red $(-1)$-curve twice at the same point. Therefore the $(-1)$-curves induces the column $[1\ 1\ 1\ 2]^T$ of $A_4$. That's all. A final divisorial contraction gives us the picture deformation corresponding to the incidence matrix $A_4$.

\begin{figure}
\centering
\includegraphics[scale=1.25]{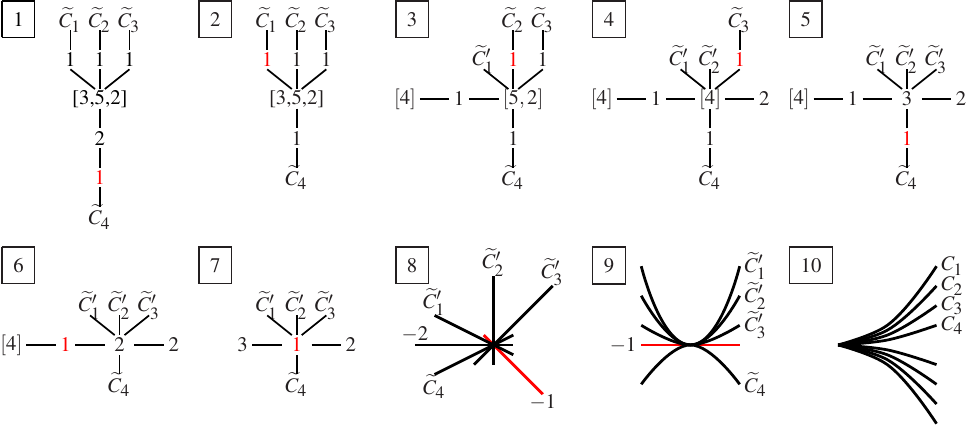} 
\caption{The procedure for identifying the incidence matrix}
\label{figure:non-usual-flip-procedure}
\end{figure}

\begin{proposition}
Kollár conjecture holds for a weighted homogeneous surface singularity whose dual graph is given in Equation~\eqref{equation:WHSS-non-usual-flip}.
\end{proposition}

\begin{proof}
There are four components in $\mathcal{C}(X)$ because the map $\phi_I \colon \mathcal{C}(X) \to \mathcal{I}(X)$ is injective by Theorem~\ref{theorem:phi_I-injective} (if we choose an usual sandwiched structure of $X$). In the above we find exactly four P-resolutions corresponding to each components of $\mathcal{C}(X)$. Therefore Kollár conjecture holds as asserted.
\end{proof}

\subsection{An example with a non-P-resolution} \label{section:non-P-resolution}

We investigate Example~6.3.6 in Kollár~\cite{Kollar-1991}. Let $(X,p)$ be a rational surface singularity with the following dual resolution graph:
\begin{equation}\label{equation:non-T-graph}
\begin{aligned}
\begin{tikzpicture}[scale=0.75]
\node[bullet] (01) at (0,1) [labelAbove={$-2$}] {};
\node[bullet] (00) at (0,0) [label=below left:{$-3$}] {};
\node[bullet] (0-1) at (0,-1) [labelBelow={$-2$}] {};

\node[bullet] (-20) at (-2,0) [labelBelow={$-2$}] {};
\node[bullet] (-10) at (-1,0) [labelBelow={$-2$}] {};
\node[bullet] (10) at (1,0) [labelBelow={$-4$}] {};

\draw[-] (01)--(00)--(0-1);
\draw[-] (-20)--(-10)--(00)--(10);
\end{tikzpicture}
\end{aligned}
\end{equation}
It is a sandwiched surface singularity. We attach $(-1)$-vertices on the $(-2)$-vertices in North and South, respectively. Consider small curvetta, say $C_1$ and $C_2$, passing through the $(-1)$-curves, respectively. Then we have a decorated curve $(C=C_1 \cup C_2, l=\{l_1=8, l_2=8\})$ of $(X,p)$. Indeed the decorated curve can be realized as plane curve singularities, for example, $C_1=y^3-x^4$ and $C_2=y^3-2x^4$. Notice that $\delta(C_i)=3$ for $i=1,2$ and $C_1 \cdot C_2 = 12$.

Using the necessary conditions on incidence matrices in Equation~\eqref{equation:combinatorial-incidence-matrix}, one can show that there are exactly six combinatorial incidence matrices:
\begin{align*}
A_1 &=\begin{bmatrix}
1 & 0 & 1 & 0 & 1 & 1 & 1 & 3 \\
0 & 1 & 0 & 1 & 1 & 1 & 1 & 3
\end{bmatrix} &
A_2 &=\begin{bmatrix}
1 & 0 & 1 & 0 & 2 & 2 & 2\\
0 & 1 & 0 & 1 & 2 & 2 & 2
\end{bmatrix} \\
A_3 &=\begin{bmatrix}
1 & 0 & 2 & 1 & 2 & 2 \\
0 & 1 & 1 & 2 & 2 & 2
\end{bmatrix} &
A_4 &=\begin{bmatrix}
0 & 1 & 1 & 2 & 2 & 3 \\
1 & 1 & 1 & 1 & 1 & 3
\end{bmatrix} \\
A_5 &=\begin{bmatrix}
1 & 1 & 1 & 1 & 1 & 3 \\
0 & 1 & 1 & 2 & 2 & 3
\end{bmatrix} &
A_6 &=\begin{bmatrix}
2 & 2 & 2 & 1 & 1\\
2 & 1 & 1 & 2 & 2
\end{bmatrix}
\end{align*}

On the other hand, there are five P-resolutions dominated by its minimal resolution described in the below and their corresponding M-resolutions are given in Figure~\ref{figure:5-M-resolutions}:
\begin{enumerate}[(1)]
\item The minimal Du Val resolution.
\item Contract all $(-2)$-curves and the $(-4)$-curve.
\item Contract any of the three configurations $2-3-4$ and the $(-2)$-curve on the left if necessary.
\end{enumerate}

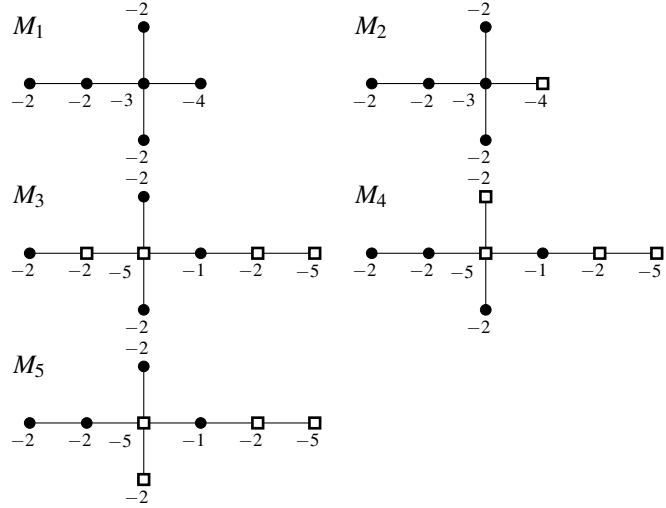
\begin{figure}
\begin{tikzpicture}[scale=0.75]
\begin{scope}
\node[empty] (-21) at (-2,1) []{$M_1$};

\node[bullet] (01) at (0,1) [labelAbove={$-2$}] {};
\node[bullet] (00) at (0,0) [label=below left:{$-3$}] {};
\node[bullet] (0-1) at (0,-1) [labelBelow={$-2$}] {};

\node[bullet] (-20) at (-2,0) [labelBelow={$-2$}] {};
\node[bullet] (-10) at (-1,0) [labelBelow={$-2$}] {};
\node[bullet] (10) at (1,0) [labelBelow={$-4$}] {};

\draw[-] (01)--(00)--(0-1);
\draw[-] (-20)--(-10)--(00)--(10);
\end{scope}

\begin{scope}[shift={(6,0)}]
\node[empty] (-21) at (-2,1) []{$M_2$};

\node[bullet] (01) at (0,1) [labelAbove={$-2$}] {};
\node[bullet] (00) at (0,0) [label=below left:{$-3$}] {};
\node[bullet] (0-1) at (0,-1) [labelBelow={$-2$}] {};

\node[bullet] (-20) at (-2,0) [labelBelow={$-2$}] {};
\node[bullet] (-10) at (-1,0) [labelBelow={$-2$}] {};
\node[rectangle] (10) at (1,0) [labelBelow={$-4$}] {};

\draw[-] (01)--(00)--(0-1);
\draw[-] (-20)--(-10)--(00)--(10);
\end{scope}

\begin{scope}[shift={(0,-3)}]
\node[empty] (-21) at (-2,1) []{$M_3$};

\node[bullet] (01) at (0,1) [labelAbove={$-2$}] {};
\node[rectangle] (00) at (0,0) [label=below left:{$-5$}] {};
\node[bullet] (0-1) at (0,-1) [labelBelow={$-2$}] {};

\node[bullet] (-20) at (-2,0) [labelBelow={$-2$}] {};
\node[rectangle] (-10) at (-1,0) [labelBelow={$-2$}] {};
\node[bullet] (10) at (1,0) [labelBelow={$-1$}] {};
\node[rectangle] (20) at (2,0) [labelBelow={$-2$}] {};
\node[rectangle] (30) at (3,0) [labelBelow={$-5$}] {};

\draw[-] (01)--(00)--(0-1);
\draw[-] (-20)--(-10)--(00)--(10)--(20)--(30);
\end{scope}

\begin{scope}[shift={(6,-3)}]
\node[empty] (-21) at (-2,1) []{$M_4$};

\node[rectangle] (01) at (0,1) [labelAbove={$-2$}] {};
\node[rectangle] (00) at (0,0) [label=below left:{$-5$}] {};
\node[bullet] (0-1) at (0,-1) [labelBelow={$-2$}] {};

\node[bullet] (-20) at (-2,0) [labelBelow={$-2$}] {};
\node[bullet] (-10) at (-1,0) [labelBelow={$-2$}] {};
\node[bullet] (10) at (1,0) [labelBelow={$-1$}] {};
\node[rectangle] (20) at (2,0) [labelBelow={$-2$}] {};
\node[rectangle] (30) at (3,0) [labelBelow={$-5$}] {};

\draw[-] (01)--(00)--(0-1);
\draw[-] (-20)--(-10)--(00)--(10)--(20)--(30);
\end{scope}

\begin{scope}[shift={(0,-6)}]
\node[empty] (-21) at (-2,1) []{$M_5$};

\node[bullet] (01) at (0,1) [labelAbove={$-2$}] {};
\node[rectangle] (00) at (0,0) [label=below left:{$-5$}] {};
\node[rectangle] (0-1) at (0,-1) [labelBelow={$-2$}] {};

\node[bullet] (-20) at (-2,0) [labelBelow={$-2$}] {};
\node[bullet] (-10) at (-1,0) [labelBelow={$-2$}] {};
\node[bullet] (10) at (1,0) [labelBelow={$-1$}] {};
\node[rectangle] (20) at (2,0) [labelBelow={$-2$}] {};
\node[rectangle] (30) at (3,0) [labelBelow={$-5$}] {};

\draw[-] (01)--(00)--(0-1);
\draw[-] (-20)--(-10)--(00)--(10)--(20)--(30);
\end{scope}

\end{tikzpicture}
\caption{Five M-resolutions}
\label{figure:5-M-resolutions}
\end{figure}

There is also a one-parameter family of P-modifications given in Kollár~\cite[Example~6.3.6]{Kollar-1991}. We first blow up any point $x$ on the $(-3)$-curve that is not intersection points. We have a $(-1)$-curve and the following configuration.

\begin{equation}\label{equation:M6-configuration}
\begin{aligned}
\begin{tikzpicture}[scale=0.75]
\tikzset{label distance=-0.15em}
\tikzset{font=\scriptsize}
\node[bullet] (01) at (0,1) [labelAbove={$-2$}] {};
\node[bullet] (00) at (0,0) [label=below left:{$-4$}] {};
\node[bullet] (0-1) at (0,-1) [labelBelow={$-2$}] {};

\node[bullet] (-20) at (-2,0) [labelBelow={$-2$}] {};
\node[bullet] (-10) at (-1,0) [labelBelow={$-2$}] {};
\node[bullet] (10) at (1,0) [labelBelow={$-4$}] {};

\draw[-] (01)--(00)--(0-1);
\draw[-] (-20)--(-10)--(00)--(10);
\end{tikzpicture}
\end{aligned}
\end{equation}

\begin{lemma}[{Kollár~\cite[Lemma~6.3.7]{Kollar-1991}}]\label{lemma:M6-equation}
Any singularity with the dual resolution graph as in Equation~\eqref{equation:M6-configuration} has a $\mathbb{Q}$-Gorenstein component.
\end{lemma}

\begin{proof}
We briefly summary the proof in Kollár~\cite[Lemma~6.3.7]{Kollar-1991}. Any singularity with the dual graph \eqref{equation:M6-configuration} is a quotient of the elliptic singularity $\{x^2z + xy^2 + z^5 + ay^2z^2 =0 \} \subset \mathbb{C}^3$ by the $\mathbb{Z}_5$-action with weight $(1,2,3)$. Then the $\mathbb{Q}$-Gorenstein smoothing is given by
\begin{equation*}
\{x^2z + xy^2 + z^5 + ay^2z^2 + t=0 \} / \mathbb{Z}_5. \qedhere
\end{equation*}
\end{proof}

We contract all the other curves except the $(-1)$-curve, which induces the following P-modification $M_6$.
\begin{equation*}
\begin{tikzpicture}[scale=0.75]
\tikzset{label distance=-0.15em}
\tikzset{font=\scriptsize}
\node[empty] (-21) at (-2,1) []{$M_6$};

\node[rectangle] (01) at (0,1) [labelAbove={$-2$}] {};
\node[rectangle] (00) at (0,0) [label=below left:{$-4$}] {};
\node[rectangle] (0-1) at (0,-1) [labelBelow={$-2$}] {};

\node[rectangle] (-20) at (-2,0) [labelBelow={$-2$}] {};
\node[rectangle] (-10) at (-1,0) [labelBelow={$-2$}] {};
\node[rectangle] (10) at (1,0) [labelBelow={$-4$}] {};

\node[bullet] (0505) at (0.5,0.5) [labelAbove={$-1$}] {};

\draw[-] (01)--(00)--(0-1);
\draw[-] (-20)--(-10)--(00)--(10);
\draw[-] (00)--(0505);
\end{tikzpicture}
\end{equation*}

\subsection{The correspondence}

It is not difficult to show that, for each $i=1,\dotsc,5$, the M-resolution $M_i$ corresponds to the incidence matrix $A_i$. On the other hand, we can also apply the semi-stable MMP machinery to the M-resolution $M_6$ even though it has a singularity which is not a T-singularity; Proposition~\ref{proposition:non-T}.

\begin{lemma}\label{lemma:mu=2}
The Milnor number of the $\mathbb{Q}$-Gorenstein smoothing in Lemma~\ref{lemma:M6-equation} is $3$.
\end{lemma}

\begin{proof}
At first, let $H=\{f(x,y,z)=x^2z + xy^2 + z^5 + ay^2z^2 =0 \}$. Since it is a hypersurface singularity, its Milnor number is given by the length of the Jacobian algebra
\begin{equation*}
\begin{split}
J &= \mathbb{C}[[x,y,z]]/\langle \partial f/\partial x, \partial f/\partial y, \partial f/\partial z \rangle \\
&= \mathbb{C}[[x,y,z]]/\langle 2xz+y^2, 2xy+2ayz^2, x^2+5z^4+2ay^2z\rangle.
\end{split}
\end{equation*}
Then one can check that its length is $19$.

Let $H_t=\{f(x,y,z)=x^2z + xy^2 + z^5 + ay^2z^2 +t =0 \}$ for a fixed $t \neq 0$. If $M$ is the Milnor fiber of the $\mathbb{Q}$-Gorenstein smoothing in Lemma~\ref{lemma:M6-equation}, then there is a unramified connected covering $H_t \to M$ of degree $5$. Then, since $b_1(M)=b_1(H_t)=0$, we have
\begin{equation*}
5e(M)=e(H_t)=1+b_2(H_t)=20.
\end{equation*}
Therefore $b_2(M)=3$.
\end{proof}

\begin{lemma}\label{lemma:flip-non-T}
Let $(E^{-} \subset \mathcal{Z}) \to (Q \in \mathcal{Y})$ be an extremal neighborhood given by the following dual graph
\begin{equation*}
\begin{tikzpicture}[scale=0.75]
\tikzset{label distance=-0.15em}
\tikzset{font=\scriptsize}

\node[bullet] (02) at (0,2) [label=left:{$-1$},label=right:{$E^{-}$}] {};
\node[rectangle] (01) at (0,1) [label=left:{$-2$}] {};
\node[rectangle] (00) at (0,0) [label=below left:{$-4$}] {};
\node[rectangle] (0-1) at (0,-1) [labelBelow={$-2$}] {};

\node[rectangle] (-20) at (-2,0) [labelBelow={$-2$}] {};
\node[rectangle] (-10) at (-1,0) [labelBelow={$-2$}] {};
\node[rectangle] (10) at (1,0) [labelBelow={$-4$}] {};

\draw[-] (02)--(01)--(00)--(0-1);
\draw[-] (-20)--(-10)--(00)--(10);
\end{tikzpicture}
\end{equation*}
Then there is a flip $(E^+ \subset \mathcal{Z}^+)$ given by the dual graph
\begin{equation}\label{equation:M6-flip}
\begin{aligned}
\begin{tikzpicture}[scale=0.75]
\tikzset{label distance=-0.15em}
\tikzset{font=\scriptsize}

\node[rectangle] (00) at (0,0) [label=below left:{$-3$}] {};
\node[rectangle] (0-1) at (0,-1) [labelBelow={$-2$}] {};

\node[rectangle] (-20) at (-2,0) [labelBelow={$-2$}] {};
\node[bullet] (-10) at (-1,0) [labelBelow={$-2$},labelAbove={$E^+$}] {};
\node[rectangle] (10) at (1,0) [labelBelow={$-4$}] {};

\draw[-] (00)--(0-1);
\draw[-] (-20)--(-10)--(00)--(10);
\end{tikzpicture}
\end{aligned}
\end{equation}
\end{lemma}

\begin{proof}
The singularity $Q \in Y$ has a dual graph
\begin{equation*}
\begin{tikzpicture}[scale=0.75]
\tikzset{label distance=-0.15em}
\tikzset{font=\scriptsize}

\node[bullet] (00) at (0,0) [label=below left:{$-3$}] {};
\node[bullet] (0-1) at (0,-1) [labelBelow={$-2$}] {};

\node[bullet] (-20) at (-2,0) [labelBelow={$-2$}] {};
\node[bullet] (-10) at (-1,0) [labelBelow={$-2$}] {};
\node[bullet] (10) at (1,0) [labelBelow={$-4$}] {};

\draw[-] (00)--(0-1);
\draw[-] (-20)--(-10)--(00)--(10);
\end{tikzpicture}
\end{equation*}
and the extremal neighborhood $(E^{-} \subset \mathcal{Z})$ induces a smoothing $\mathcal{Y} \to \Delta$ of $Q \in Y$ whose Milnor number is $3$ by Lemma~\ref{lemma:mu=2}.

Notice that $Q \in Y$ is a sandwiched surface singularity. If we choose its sandwiched $(-1)$-curves by one attached to the $(-3)$-curve and the other one attached to the $(-2)$-curve in south position, then we have four incidence matrices
\begin{align*}
A_1' &=\begin{bmatrix}
0 & 1 & 0 & 1 & 1 & 1 & 3 \\
1 & 0 & 1 & 1 & 1 & 1 & 3
\end{bmatrix} &
A_2' &=\begin{bmatrix}
0 & 1 & 0 & 2 & 2 & 2\\
1 & 0 & 1 & 2 & 2 & 2
\end{bmatrix} \\
A_3' &=\begin{bmatrix}
0 & 2 & 1 & 2 & 2 \\
1 & 1 & 2 & 2 & 2
\end{bmatrix} &
A_5' &=\begin{bmatrix}
1 & 1 & 1 & 1 & 3 \\
1 & 1 & 2 & 2 & 3
\end{bmatrix}
\end{align*}
Here $A_3'$ and $A_5'$ correspond to one-parameter smoothings of $Q \in Y$ with Milnor number $3$. But one can check via the MMP machinery in the paper that the following $P$-resolutions $P_3'$ and $P_5'$ correspond to $A_3'$ and $A_5'$, respectively:
\begin{equation*}
\begin{tikzpicture}[scale=0.75]
\tikzset{label distance=-0.15em}
\tikzset{font=\scriptsize}

\begin{scope}[shift={(0,0)}]
\node[empty] (-21) at (-2,0.5) []{$P_3'$};

\node[rectangle] (00) at (0,0) [label=below left:{$-3$}] {};
\node[bullet] (0-1) at (0,-1) [labelBelow={$-2$}] {};

\node[bullet] (-20) at (-2,0) [labelBelow={$-2$}] {};
\node[rectangle] (-10) at (-1,0) [labelBelow={$-2$}] {};
\node[rectangle] (10) at (1,0) [labelBelow={$-4$}] {};

\draw[-] (00)--(0-1);
\draw[-] (-20)--(-10)--(00)--(10);
\end{scope}

\begin{scope}[shift={(5,0)}]
\node[empty] (-21) at (-2,0.5) []{$P_5'$};

\node[rectangle] (00) at (0,0) [label=below left:{$-3$}] {};
\node[rectangle] (0-1) at (0,-1) [labelBelow={$-2$}] {};

\node[rectangle] (-20) at (-2,0) [labelBelow={$-2$}] {};
\node[bullet] (-10) at (-1,0) [labelBelow={$-2$}] {};
\node[rectangle] (10) at (1,0) [labelBelow={$-4$}] {};

\draw[-] (00)--(0-1);
\draw[-] (-20)--(-10)--(00)--(10);
\end{scope}
\end{tikzpicture}
\end{equation*}

According to Kollár--Mori~\cite[Theorem~13.5]{KM-1992}, the exceptional curve $E^+$ is irreducible. Therefore the flip of the extremal neighborhood $(E^{-} \subset \mathcal{Z}) \to (Q \in \mathcal{Y})$ is given by Equation~\eqref{equation:M6-flip} as asserted.
\end{proof}

\begin{lemma}\label{lemma:flip-non-T-degeneration}
Let $C$ be an irreducible curve contained in the central fiber of the extremal neighborhood $(E^{-} \subset \mathcal{Z})$ in Lemma~\ref{lemma:flip-non-T} such that $C \cdot E^{-}=1$ and $C$ does not pass through the singular point. Let $C'$ be the proper transform of $C$ after the flip along $E^{-}$. Then the curve $C$ in the general fiber $Z^+_t$ ($t \neq 0$) degenerates to $C'+2E^{+}$ in the central fiber $Z^+_0$.
\end{lemma}

\begin{proof}
Suppose that $C$ in a general fiber $Z^+_t$ degenerates to $C'+\beta E^+$ in the central fiber $Z^+_0$. By Proposition~\ref{proposition:degeneration}, we have
\begin{equation*}
K_{Z_0} \cdot C = K_{Z_t} \cdot C = K_{Z^+_t} \cdot {C} = K_{Z^+_0} \cdot (C'+\beta E^+) = K_{Z^+_0} \cdot C' + \beta K_{Z^+_0} \cdot E^+.
\end{equation*}

Let $\pi \colon \widetilde{Z}_0 \to Z_0$ and $\pi^+ \colon \widetilde{Z}^+_0 \to Z^+_0$ be the minimal resolutions, respectively. Let $\widetilde{C}$ and $\widetilde{C}'$ be the proper transforms of $C$ and $C'$, respectively. Then we have
\begin{equation*}
K_{Z_0} \cdot C = K_{\widetilde{Z}_0} \cdot \widetilde{C}.
\end{equation*}
On the other hand, notice that
\begin{equation*}
(\pi^+)^{\ast} K_{Z^+_0} = K_{\widetilde{Z}^+_0} + \frac{1}{3} E_1 + \frac{2}{3} E_2 + \frac{2}{3} E_3
\end{equation*}
where $[E_1, E_2, E_3]$ are the exceptional divisors $[2,3,4]$. So we have
\begin{equation*}
K_{Z^+_0} \cdot C' + \beta K_{Z^+_0} \cdot E^+ = K_{\widetilde{Z}^+_0} \cdot \widetilde{C}' + \frac{2}{3} + \frac{2}{3} \beta.
\end{equation*}
But $\widetilde{Z}_0 \to \widetilde{Z}^+_0$ is a composition of two blowing-ups. Hence
\begin{equation*}
K_{\widetilde{Z}^+_0} \widetilde{C}' = K_{Z_0} \cdot C = K_{\widetilde{Z}_0} \cdot \widetilde{C} - 2
\end{equation*}
Therefore $\beta=2$.
\end{proof}

\begin{proposition}\label{proposition:non-T}
The 6th P-modification $M_6$ corresponds to the 6th incidence matrix $A_6$.
\end{proposition}

\begin{figure}
\centering
\includegraphics[scale=1.25]{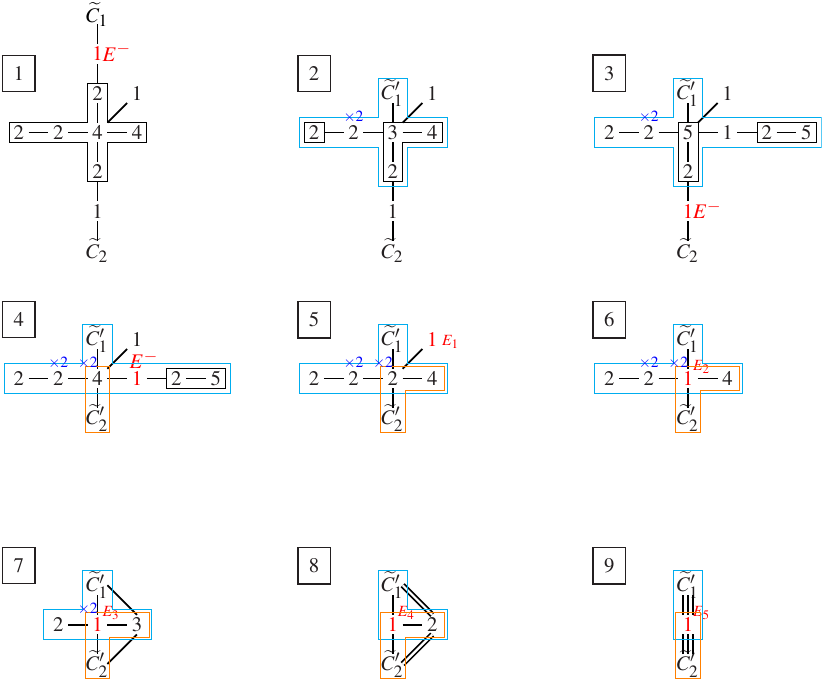} 
\caption{A P-modification $M_6$ with a non-T-singularity}
\label{figure:PtoI-WHSS-M6}
\end{figure}

\begin{proof}
We present in Figure~\ref{figure:PtoI-WHSS-M6} the process of the semi-stable minimal model program for identifying the picture deformation corresponding to the P-modification $M_6$. We briefly explain each steps. In Step~1, we apply the non-usual flip described in Lemma~\ref{lemma:flip-non-T}. Then we have Step~2. By Lemma~\ref{lemma:flip-non-T-degeneration}, we have a degeneration
\begin{equation*}
\widetilde{C}_1 \rightsquigarrow \widetilde{C}_1' + 2(2)
\end{equation*}
where $(2)$ denotes the $(-2)$-curve in Step~2. We illustrate the degeneration of $\widetilde{C}_1$ by enclosing them with blue lines. Then, we compute the crepant M-resolution of Step~2. In Step~3, we perform a usual flip on $E^-$. Then $\widetilde{C}_2$ is degenerated, as shown by the orange lines in Step~4. We apply again a usual flip to $E^-$ in Step~4. Then we resolve all T-singularities. Each $(-1)$-curves will now generate columns of $A_6$, which is easily calculable.
\end{proof}

\begin{remark}
Kollár conjecture for this singularity if we can prove that $\phi_I \colon \mathcal{C}(X) \to \mathcal{I}(X)$ is injective. But we don't know how to prove it yet.
\end{remark}

\subsection{An example with a nonnormal P-modification}

We investigate Example~6.6.3 in Kollár~\cite{Kollar-1991}. Let $(X,p)$ be a rational surface singularity whose dual graph of the minimal resolution is given by
\begin{equation*}
\begin{tikzpicture}[scale=0.75]
\tikzset{label distance=-0.15em}
\tikzset{font=\scriptsize}
\node[bullet] (01) at (0,1) [labelAbove={$-2$}] {};
\node[bullet] (00) at (0,0) [label=below left:{$-4$}] {};
\node[bullet] (0-1) at (0,-1) [labelBelow={$-2$}] {};

\node[bullet] (-10) at (-1,0) [labelBelow={$-3$}] {};
\node[bullet] (10) at (1,0) [labelBelow={$-2$}] {};
\node[bullet] (20) at (2,0) [labelBelow={$-2$}] {};

\draw[-] (01)--(00)--(0-1);
\draw[-] (-10)--(00)--(10)--(20);
\end{tikzpicture}
\end{equation*}
It is a sandwiched surface singularity and its sandwiched structure is given in Figure~\ref{figure:WHSS-sandwiched-structure}.

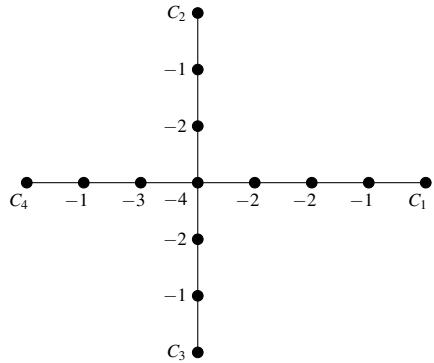
\begin{figure}
\centering
\begin{tikzpicture}[scale=0.75]
\tikzset{font=\scriptsize}
\node[bullet] (01) at (0,1) [label=left:{$-2$}] {};
\node[bullet] (00) at (0,0) [label=below left:{$-4$}] {};
\node[bullet] (0-1) at (0,-1) [label=left:{$-2$}] {};

\node[bullet] (-10) at (-1,0) [labelBelow={$-3$}] {};
\node[bullet] (10) at (1,0) [labelBelow={$-2$}] {};
\node[bullet] (20) at (2,0) [labelBelow={$-2$}] {};

\node[bullet] (30) at (3,0) [labelBelow={$-1$}] {};
\node[bullet] (40) at (4,0) [labelBelow={$C_1$}] {};

\node[bullet] (02) at (0,2) [label=left:{$-1$}] {};
\node[bullet] (03) at (0,3) [label=left:{$C_2$}] {};

\node[bullet] (0-2) at (0,-2) [label=left:{$-1$}] {};
\node[bullet] (0-3) at (0,-3) [label=left:{$C_3$}] {};

\node[bullet] (-20) at (-2,0) [labelBelow={$-1$}] {};
\node[bullet] (-30) at (-3,0) [labelBelow={$C_4$}] {};

\draw[-] (03)--(02)--(01)--(00)--(0-1)--(0-2)--(0-3);
\draw[-] (-30)--(-20)--(-10)--(00)--(10)--(20)--(30)--(40);
\end{tikzpicture}
\caption{A sandwiched structure}
\label{figure:WHSS-sandwiched-structure}
\end{figure}

The decorated curve is given by $(C_1 \cup C_2 \cup C_3 \cup C_4, \{5,4,4,2\})$ and they satisfy the following incidence relations: $C_i \cdot C_j =2$ for $i \neq j$ with $i, j \le 3$ and $C_k \cdot C_4=1$ for all $k \le 3$:

\begin{center}
\includegraphics[scale=1.25]{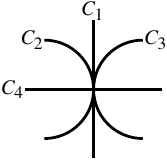} 
\end{center}

There are six picture deformations of $(X,p)$. In Figure~\ref{figure:WHSS-six-picture-deformations}, we present the decorated curves in a general fiber for each picture deformations. On the other hand, there are also six P-modifications: Fives are P-resolutions, but the sixth is a nonnormal P-modification. We present the dual graphs of five P-resolutions in Figure~\ref{figure:WHSS-five-P-resolutions}.

\begin{figure}
\centering
\includegraphics[scale=1.25]{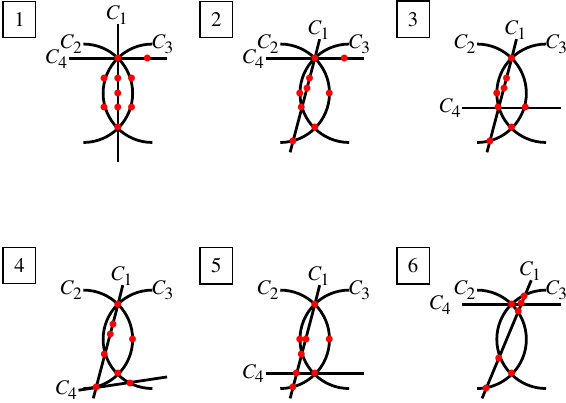} 
\caption{Six picture deformations}
\label{figure:WHSS-six-picture-deformations}
\end{figure}

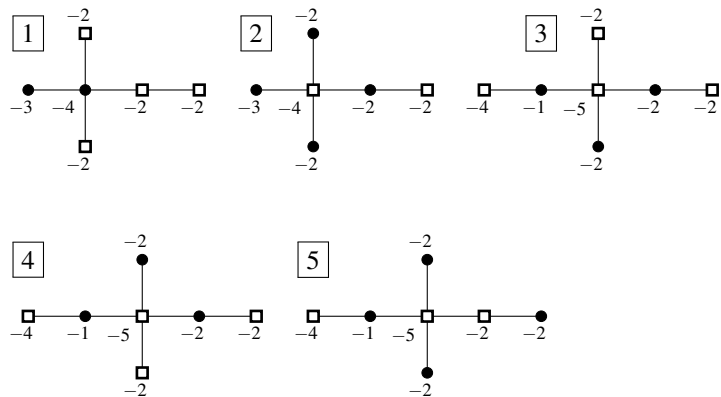
\begin{figure}
\centering
\begin{tikzpicture}[scale=0.75]
\begin{scope}
\node[empty] (-11) at (-1,1) []{$\fbox{1}$};

\node[rectangle] (01) at (0,1) [labelAbove={$-2$}] {};
\node[bullet] (00) at (0,0) [label=below left:{$-4$}] {};
\node[rectangle] (0-1) at (0,-1) [labelBelow={$-2$}] {};

\node[bullet] (-10) at (-1,0) [labelBelow={$-3$}] {};
\node[rectangle] (10) at (1,0) [labelBelow={$-2$}] {};
\node[rectangle] (20) at (2,0) [labelBelow={$-2$}] {};

\draw[-] (01)--(00)--(0-1);
\draw[-] (-10)--(00)--(10)--(20);
\end{scope}

\begin{scope}[shift={(4,0)}]
\node[empty] (-11) at (-1,1) []{$\fbox{2}$};

\node[bullet] (01) at (0,1) [labelAbove={$-2$}] {};
\node[rectangle] (00) at (0,0) [label=below left:{$-4$}] {};
\node[bullet] (0-1) at (0,-1) [labelBelow={$-2$}] {};

\node[bullet] (-10) at (-1,0) [labelBelow={$-3$}] {};
\node[bullet] (10) at (1,0) [labelBelow={$-2$}] {};
\node[rectangle] (20) at (2,0) [labelBelow={$-2$}] {};

\draw[-] (01)--(00)--(0-1);
\draw[-] (-10)--(00)--(10)--(20);
\end{scope}

\begin{scope}[shift={(9,0)}]
\node[empty] (-11) at (-1,1) []{$\fbox{3}$};

\node[rectangle] (01) at (0,1) [labelAbove={$-2$}] {};
\node[rectangle] (00) at (0,0) [label=below left:{$-5$}] {};
\node[bullet] (0-1) at (0,-1) [labelBelow={$-2$}] {};

\node[rectangle] (-20) at (-2,0) [labelBelow={$-4$}] {};
\node[bullet] (-10) at (-1,0) [labelBelow={$-1$}] {};
\node[bullet] (10) at (1,0) [labelBelow={$-2$}] {};
\node[rectangle] (20) at (2,0) [labelBelow={$-2$}] {};

\draw[-] (01)--(00)--(0-1);
\draw[-] (-20)--(-10)--(00)--(10)--(20);
\end{scope}

\begin{scope}[shift={(1,-4)}]
\node[empty] (-21) at (-2,1) []{$\fbox{4}$};

\node[bullet] (01) at (0,1) [labelAbove={$-2$}] {};
\node[rectangle] (00) at (0,0) [label=below left:{$-5$}] {};
\node[rectangle] (0-1) at (0,-1) [labelBelow={$-2$}] {};

\node[rectangle] (-20) at (-2,0) [labelBelow={$-4$}] {};
\node[bullet] (-10) at (-1,0) [labelBelow={$-1$}] {};
\node[bullet] (10) at (1,0) [labelBelow={$-2$}] {};
\node[rectangle] (20) at (2,0) [labelBelow={$-2$}] {};

\draw[-] (01)--(00)--(0-1);
\draw[-] (-20)--(-10)--(00)--(10)--(20);
\end{scope}

\begin{scope}[shift={(6,-4)}]
\node[empty] (-21) at (-2,1) []{$\fbox{5}$};

\node[bullet] (01) at (0,1) [labelAbove={$-2$}] {};
\node[rectangle] (00) at (0,0) [label=below left:{$-5$}] {};
\node[bullet] (0-1) at (0,-1) [labelBelow={$-2$}] {};

\node[rectangle] (-20) at (-2,0) [labelBelow={$-4$}] {};
\node[bullet] (-10) at (-1,0) [labelBelow={$-1$}] {};
\node[rectangle] (10) at (1,0) [labelBelow={$-2$}] {};
\node[bullet] (20) at (2,0) [labelBelow={$-2$}] {};

\draw[-] (01)--(00)--(0-1);
\draw[-] (-20)--(-10)--(00)--(10)--(20);
\end{scope}
\end{tikzpicture}
\caption{Five P-resolutions}
\label{figure:WHSS-five-P-resolutions}
\end{figure}

It is not difficult to show that the $i$-th P-resolution corresponds to the $i$-th picture deformation for $i=1,\dotsc,5$ by applying the MMP machinery.

\subsubsection{A nonnormal P-modification}

Kollár~\cite[Example~6.3.3]{Kollar-1991} also present a nonnormal P-modification of $(X,p)$. We briefly recall the construction. There are four distinguished points on the central $(-4)$-curve $C$ corresponding to the four intersection points with another exceptional curves. We denote the intersection points by $N, E, S, W$ according to the direction of the exceptional curves given in the above resolution graph. There is a unique involution $\tau$ on $C$ such that $\tau(N)=S$ and $\tau(E)=W$. We first contract all curves except $C$. Then we have a normal surface $C'' \subset X''$. Then for each $x \in C''$ we identify $x$ and $\tau(x)$ to obtain a nonnormal surface germ $g \colon C' \subset X' \to (0 \in X)$. Along $C'$ we have generically normal crossing points. There are also two pinch points corresponding to branch points of the involution $\tau$ and two singular points of the form
\begin{equation*}
\text{$(xy=0) \subset \mathbb{C}^3/\mathbb{Z}_2(1,-1,1)$ and $(xy=0) \subset \mathbb{C}^3/\mathbb{Z}_3(1,-1,1)$}
\end{equation*}
Kollár~\cite[6.3.3.4]{Kollar-1991} show that the singularities on $X'$ have $\mathbb{Q}$-Gorenstein smoothings and $X'$ is indeed a (nonnormal) P-modification of $X$.

\begin{remark}
The picture deformation \#6 seems to correspond to the non-normal P-modification. But it is not clear whether one can see the correspondence via the semistable minimal model program.
\end{remark}

\end{document}